\documentclass[a4paper,10pt]{amsart}

\usepackage[backrefs,lite]{amsrefs}
\usepackage{mathtools}
\usepackage{amssymb}
\usepackage{lmodern,bm}
\usepackage{longtable}
\usepackage{tikz}
    \usetikzlibrary{math,calc,matrix,arrows,shapes}
    \tikzset{node distance=2cm, auto}
    \newcounter{arraycard}
\newcommand{\rhotwopic}[3]{
\begin{tikzpicture}[scale=0.6]
\def\Nlist{#1}

\tikzmath{
\nmu=#2;
integer \nnmu;
\nnmu=2*#2;
\xlambda=#3;
let \len=3cm;
\xphi1=10;
\xphi2=120;
}

\setcounter{arraycard}{0}
\foreach \x in \Nlist {
\stepcounter{arraycard}
}

\tikzmath{
\srays=\value{arraycard};
\phidiff=(\xphi2 - \xphi1) / (\srays-1);
}

\foreach \i[remember=\i as \lasti (initially 1)] in {1, ...,\srays}
{
    \tikzmath{
    coordinate \c;
    coordinate \lastc;
    \c = (\i*\phidiff - \phidiff + \xphi1:\len);
    \lastc = (\lasti*\phidiff - \phidiff + \xphi1:\len);
    }
    
	\path[top color=gray!10!, bottom color=gray!50!] 
	(0,0)--(\lastc)--(\c)--cycle;
}

    \tikzmath{
    coordinate \rightboundary;
    coordinate \leftboundary;
    \rightboundary = (\xlambda*\phidiff - \phidiff + \xphi1:\len);
    \i=\xlambda + 1;
    \leftboundary = (\i*\phidiff - \phidiff + \xphi1:\len);
    }
    
    \path[top color=gray!40!, bottom color=gray!90!]
	(0,0)--(\rightboundary)--(\leftboundary)--cycle;
	
	\draw[color=white] ($0.375*(\leftboundary) + 0.375*(\rightboundary)$) node {\large$\lambda$};

\foreach[count=\i] \N in \Nlist
{
    \tikzmath{
    coordinate \c;
    \c = (\i*\phidiff - \phidiff + \xphi1:\len);
    }
    
    \draw[thick,color=black] (0,0) -- (\c);

    \foreach \d in {1,...,\N}
    {
    \tikzmath{
    int \factN;
    if \N > 3 then {
        \factN = \N;
    } else {
        \factN = 3;
    };
    coordinate \u;
    \ulen=(0.75+(\d-1)*0.5);
    \u = (\i*\phidiff - \phidiff + \xphi1:\len/\factN*\ulen);
    }
        \fill (\u) circle (0.5ex);
    }
}

\ifnum \nnmu > 0
    \tikzmath{
        coordinate \reldeg;
        \reldeg = (\nmu*\phidiff - \phidiff + \xphi1:0.75*\len);
        }
        
    \filldraw[fill=white, draw=black] (\reldeg) circle (0.5ex);
\fi
\end{tikzpicture}   
}
\usepackage[all]{xy}
\usepackage{dsfont}
\usepackage{pdflscape}
\usepackage{bigdelim}
\usepackage{multirow}
\usepackage{cancel}
\usepackage{array}
\usepackage{float}
\usepackage{subcaption}

\usepackage{booktabs}
\usetikzlibrary{cd}

\usepackage[backref=page,
            hidelinks     
            ]{hyperref}
\usepackage[capitalize]{cleveref}

\CompileMatrices

\newcolumntype{C}{>{$}c<{$}}
\newcolumntype{L}{>{$}l<{$}}
\newcolumntype{R}{>{$}r<{$}}

\newtheorem{theorem}{Theorem}[section]
\newtheorem{mainthm}{Main Theorem}[section]

\newtheorem{lemma}[theorem]{Lemma}
\newtheorem{proposition}[theorem]{Proposition}

\newtheorem{class-list}[theorem]{Classification list}

\theoremstyle{definition}

\newtheorem{construction}[theorem]{Construction}

\newtheorem{remark}[theorem]{Remark}

\newtheorem{setting}[theorem]{Setting}

\numberwithin{equation}{theorem}

\newcommand{\git}{\mathbin{
  \mathchoice{/\mkern-6mu/}
    {/\mkern-6mu/}
    {/\mkern-5mu/}
    {/\mkern-5mu/}}}

\def\vector2#1#2{\left(\begin{array}{c} #1 \\ #2 \end{array}\right)}

\def\Cl{{\rm Cl}}

\def\O{{\rm O}}

\def\CC{{\mathbb C}}
\def\KK{{\mathbb K}}

\def\ZZ{{\mathbb Z}}

\def\QQ{{\mathbb Q}}
\def\PP{{\mathbb P}}

\newcommand{\KKK}{\mathcal{K}}

\def\quot{/\!\!/}

\def\bangle#1{{\langle #1 \rangle}}

\def\Mov{{\rm Mov}}
\def\Eff{{\rm Eff}}

\def\Pic{{\rm Pic}}

\def\Spec{{\rm Spec}}

\def\cone{{\rm cone}}

\def\pr{{\rm pr}}

\def\lin{{\rm lin}}

\def\rlv{{\rm rlv}}

\DeclareMathOperator{\rk}{rank}

\DeclareMathOperator{\Cox}{\mathcal{R}}
\DeclareMathOperator{\Ample}{Ample}
\DeclareMathOperator{\SAmple}{SAmple}

\makeatletter
\newcommand*{\defeq}{\mathrel{\rlap{%
                     \raisebox{0.3ex}{$\m@th\cdot$}}%
                     \raisebox{-0.3ex}{$\m@th\cdot$}}%
                     =}
\makeatother

\title[On locally factorial Fano fourfolds of Picard number two]{On locally factorial Fano fourfolds\\
of Picard number two}

\author[Andreas B\"auerle, Christian Mauz]{Andreas B\"auerle, Christian Mauz}

\address{Mathematisches Institut, Universit\"at T\"ubingen,
Auf der Morgenstelle 10, 72076 T\"ubingen, Germany}
\email{baeuerle@math.uni-tuebingen.de}

\email{mauz@math.uni-tuebingen.de}


\begin{document}

\begin{abstract}
We classify the locally factorial Fano fourfolds of Picard number two with a hypersurface Cox ring that admit an effective action of a three-dimensional torus. 
\end{abstract}

\maketitle


\section{Introduction}\label{sec:intro}

This article contributes to the classification of Fano
varieties admitting an action of an algebraic torus.
An intensely studied example class is given by the
toric Fano varieties, i.e.\ the torus action has
a full-dimensional orbit.
In this setting, the classification is entirely
combinatorial and basically runs in terms of the
associated Fano polytopes; for instance, the smooth
toric Fano varieties are classified up to dimension
nine \cites{Ba99,KrNi,Ob,Pa}.

We focus on the case of Picard number two.
In this situation, Kleinschmidt~\cite{Kl88} gave
a complete description of all smooth toric
varieties, which leads in particular
to complete classifications of the
Fano ones in any dimension.
Via linear Gale duality, Kleinschmidt's approach can be turned into a study of two-dimensional combinatorial structures, see \cite{BeHa04}*{Prop.~1.11}.
The latter point of view applies as well to
torus actions of higher complexity,~i.e. higher
maximal orbit codimension \cites{HaHe,HaSu,HaHiWr}.
This allows for instance to extend Kleinschmidt's
description to smooth varieties with a torus action
of complexity one and gives complete classifications
of smooth Fano varieties with torus action of complexity
one in any dimension \cite{FaHaNi}.
Further work in this spirit concerns smooth
intrinsic quadrics, general arrangement
varieties and intrinsic Grassmannians \cites{FaHa, HaHiWr, QuWr22}.

In the present article we leave the smooth case
and consider more generally locally factorial
varieties, meaning that every Weil divisor is locally
principal.
Whereas in the toric case smoothness and local
factoriality coincide, the latter setting
turns out to be much more general for torus
actions of complexity one;
for instance, the varieties need not be log
terminal any more
and we find infinite series of
non-isomorphic Fanos in fixed dimensions. 
We settle the case of dimension four, complexity
one and a Cox ring defined by a single relation.
Our main result considerably extends the corresponding
one in the smooth case~\cite{FaHaNi}*{Thm.~1.2}.

\begin{mainthm}
\label{thm:fano-4d-lf-rho=2}
There are~$447$ sporadic cases and $106$ infinite series of locally factorial Fano fourfolds of Picard number two coming with an effective action of a three-dimensional torus and a Cox ring defined by a single relation.
\end{mainthm}

Our varieties in question are uniquely determined by the
generator degrees and the relation in their Cox ring.
\emph{Classification lists \ref{class:s=2} to \ref{class:s=6-series}
provide the complete and redundancy free presentation
of the specifying data for the \cref{thm:fano-4d-lf-rho=2}.}
A data file containing the complete classification data is also available at \cite{BaMa24}.

For the proof of Theorem \ref{thm:fano-4d-lf-rho=2} we
distinguish two main cases.
The first one treats an ample relation degree.
There, we provide a smoothing procedure via
Bertini's theorem which allows us to infer first
constraints on the relevant invariants from the
classification of smooth Fano fourfoulds of
Picard number two with a hypersurface
Cox ring \cite{HLM}*{Thm.~1.1}.
The situation becomes more involved when the relation
degree is not ample.
In this situation, we have to classify case by case
according to the possible constellations of the Cox
ring generator degrees in the effective cone, making
heavy use of the combinatorial description
of varieties with a torus action of complexity one
from \cites{HaSu,HaHe}, see also~\cite{ADHL}*{Sec.~3.4}.

The article is organized as follows. \cref{sec:basics}
serves to provide the necessary background on Cox rings.
In \cref{sec:picandsmooth} we establish two
general facts, essentially supporting our classification:
First, Proposition \ref{prop:torfree} shows that in our
setting we always have a torsion-free Picard group 
and second, Proposition \ref{prop:general-is-smooth} supplies
us with an explicit smoothing procedure.
An example case of the proof of Theorem~\ref{thm:fano-4d-lf-rho=2}
is presented in Section
\ref{sec:proof}. The full elaboration will appear
in~\cite{Ba24}. Finally, in \cref{sec:classlists} we
present all the data of the classification.


\section{Background on Cox rings}
\label{sec:basics}

By a \emph{Mori dream space} we mean an irreducible, normal, projective complex variety $X$ with finitely generated divisor class group $\Cl(X)$ and finitely generated Cox ring $\Cox(X)$. We give a brief summary on the combinatorial approach \cites{ADHL, BeHa07, Ha08} to Mori dream spaces, adapted to our needs.
By a \emph{$K$-graded affine algebra}, where~$K$ is a finitely generated abelian group, we mean an affine $\CC$-algebra~$R$ coming with a direct sum decomposition into~$\CC$-vector subspaces
\begin{equation*}
    R = \bigoplus\limits_{w \in K} R_w
\end{equation*}
such that $R_w R_{w'} \subseteq R_{w+w'}$ holds for all $w, w' \in K$.
An element $f \in R$ is called \emph{homogeneous} if~$f \in R_w$ holds for some $w \in K$.
In that case $w$ is the \emph{degree} of $f$ and we write~$w = \deg(f)$.
Geometrically, we have the affine variety $\bar{X}$ with $R$ as its algebra of global functions and the quasitorus $H$ with $K$ as its character group:
\begin{equation*}
    \bar{X} \ = \ \Spec\, R, \qquad\qquad H \ = \ \Spec\, \CC[K].
\end{equation*}
The $K$-grading of $R$ defines an algebraic action of $H$ on $\bar{X}$.
By~$K_\QQ := K\otimes_\ZZ \QQ$ we denote the $\QQ$-vector space associated with~$K$.
\begin{enumerate}
\item
The \emph{effective cone} of $R$ is
\begin{equation*}
    \Eff(R) \ := \ \cone( \, w \in K; \ R_w \ne 0 \, ) \ \subseteq \ K_\QQ.
\end{equation*}
\item
For $x \in \bar{X}$ we have the \emph{orbit cone}
\begin{equation*}
    \omega_x \ := \ \cone( \, w \in K; \ f(x) \ne 0 \text{ for some } f \in R_w \, ) \ \subseteq \ K_\QQ.
\end{equation*}
\item
For $w \in \Eff(R)$ we have the \emph{GIT-cone}
\begin{equation*}
    \lambda_w \ := \ \bigcap\limits_{x \in \bar{X}, w \in \omega_x} \omega_x \ \subseteq \ K_\QQ.
\end{equation*}
\end{enumerate}
The $K$-grading of $R$ is called \emph{pointed} if $R_0 = \CC$ holds and $\Eff(R)$ contains no line.
The effective cone, as well as orbit cones and GIT-cones are convex polyhedral cones and there are only finitely many of them.
The GIT-cones form a (quasi-) fan~$\Lambda(R)$ in~$K_\QQ$ called the \emph{GIT-fan} of $R$, having the effective cone $\Eff(R)$ as its support. We recall Cox's quotient presentation \cite{Cox} for projective toric varieties.

\begin{construction}
\label{constr:toric-cox}
Let $S = \CC[T_1,\dots,T_r]$ and consider a pointed $K$-grading on $S$, such that the variables $T_1,\dots,T_r$ are homogeneous. Write $w_i := \deg(T_i) \in K$ for the generator degrees, also when considered in $K_\QQ$. We denote the \emph{grading map} by
\begin{equation*}
    Q\,\colon \, \ZZ^r \ \rightarrow \ K,\qquad e_i \ \mapsto \ w_i.
\end{equation*}
We have the action of the quasitorus $H$ on the affine toric variety $\bar{Z}$, where
\begin{equation*}
    H \ := \ \Spec\, \CC[K],
    \qquad\qquad
    \bar{Z} \ := \ \Spec\, S \ = \ \CC^r.
\end{equation*}
We assume that any~$r-1$ of the degrees $w_1,\dots,w_r$ generate $K$ as a group, i.e. the~$K$-grading is \emph{almost free}.
Moreover we assume that the \emph{moving cone}
\begin{equation*}
    \Mov(S) \ := \ \bigcap\limits_{i=1}^r \cone( \, w_j; \ j \ne i \, ) \ \subseteq \ K_\QQ
\end{equation*}
is of full dimension.
Fix a GIT-cone $\tau \in \Lambda(S)$ with $\tau^\circ \subseteq \Mov(S)^\circ$.
There is the~$H$-invariant open set of \emph{semistable points} $\widehat{Z}$ and the corresponding \emph{good quotient} $Z$:
\begin{equation*}
    \widehat{Z} \ := \ \bar{Z}^{ss}(\tau) \ = \ \{\, x \in \bar{Z};\ \lambda \subseteq \omega_x \,\},
    \qquad\qquad
    Z \ := \ \widehat{Z}\quot H.
\end{equation*}
The quotient variety $Z$ is a projective toric variety of dimension~$ r - \dim(K_\QQ)$ with divisor class group $\Cl(Z) = K$ and Cox ring $\Cox(Z) = S$.
\end{construction}

The following construction produces Mori dream spaces as hypersurfaces in projective toric varieties. A $K$-graded algebra~$R$ is called \emph{$K$-factorial}, or the $K$-grading of $R$ is called \emph{factorial}, if $R$ is integral and every homogeneous non-zero non-unit is a product of~$K$-primes. A \emph{$K$-prime} is a homogeneous non-zero non-unit $f \in R$ with the property that $f \mid g h$ for homogeneous~$g, h \in R$ implies that $f \mid g$ or $f \mid h$ holds.

\begin{construction}
\label{constr:hypersurface}
See \cite{ADHL}*{Sec.~3.2,~3.3} and \cite{HLM}*{Constr.~4.1,~Rem.~4.2}.
In the setting of \cref{constr:toric-cox} fix~$0 \ne \mu \in K$ and~$g \in S_\mu$ and set
\begin{equation*}
    R_g \, := \, S/\bangle{g},
    \qquad
    \bar{X}_g \, := \, V(g) \, \subseteq \, \bar{Z},
    \qquad
    \widehat{X}_g \, := \, \bar{X}_g \cap \widehat{Z},
    \qquad
    X_g \, := \, \widehat{X}_g \quot H \subseteq Z.
\end{equation*}
Then the factor algebra $R_g$ inherits a $K$-grading from $S$ and the quotient $X_g \subseteq Z$ is a closed subvariety. Moreover, there is a GIT-cone $\lambda \in \Lambda(R_g)$ with
\begin{equation*}
\widehat{X}_g \ = \ \bar{X}_g^{ss}(\lambda) \ = \ \{\, x \in \bar{X}_g;\ \lambda \subseteq \omega_x \,\}.
\end{equation*}
We assume that $R_g$ is integral and normal with $R_g^* = \CC^*$, the induced~$K$-grading is factorial and $T_1,\dots,T_r$ define a minimal system of pairwise non-associated $K$-primes in~$R_g$.
Then $X_g$ is a normal, projective variety with dimension, divisor class group and Cox ring given by
\begin{equation*}
    \dim(X_g) \ = \ \dim(Z)-1,
    \qquad
    \Cl(X_g) \ = \ K,
    \qquad
    \Cox(X_g) \ = \ R_g.
\end{equation*}
Moreover, the cones of effective, movable, semiample and ample divisor classes of~$X$ are given in the rational divisor class group~$\Cl(X_g)_\QQ = K_\QQ$ by
\begin{equation*}
    \Eff(X_g) \ = \ \Eff(R_g), 
    \qquad\qquad
    \Mov(X_g) \ = \ \Mov(R_g) \ = \ \Mov(S),
\end{equation*}
\begin{equation*}
    \SAmple(X_g) \ = \ \lambda,
    \qquad\qquad
    \Ample(X_g) \ = \ \lambda^\circ.
\end{equation*}
\end{construction}

\begin{remark}
\label{rem:ambtor}
Let $X = X_g$ as in \cref{constr:hypersurface}. The \emph{minimal ambient toric variety} of $X$ is the unique minimal open toric subvariety $Z_g \subseteq Z$ containing $X$.
For the ample cones of $X$, $Z_g$ and $Z$ we have
\begin{equation*}
    \tau^\circ \ = \ \Ample(Z) \ \subseteq \ \Ample(Z_g) \ = \ \Ample(X) \ = \ \lambda^\circ.
\end{equation*}
\end{remark}

\begin{remark}
\label{rem:hsmds}
A Mori dream space $X$ with divisor class group $\Cl(X) = K$ has a \emph{hypersurface Cox ring} if there is an \emph{irredundant} $K$-graded presentation
\begin{equation*}
    \Cox(X) \ = \ \CC[T_1,\dots,T_r] / \bangle{g},
\end{equation*}
meaning that the ideal $\bangle{g}$ contains no element $T_i-h_i$ with $h_i \in \KK[T_1,\dots,T_r]$ not depending on $T_i$. If such a presentation exists, then we have $X = X_g$ as in \cref{constr:hypersurface}.
\end{remark}

\begin{proposition}\label{prop:antican}
See \cite{ADHL}*{Prop.~3.3.3.2}.
Let $X = X_g$ as in \cref{constr:hypersurface}. Then the anticanonical class of $X$ is given in~$K = \Cl(X)$ by
\begin{equation*}
    -\KKK \ = \ w_1 + \dots + w_r - \mu.
\end{equation*}
\end{proposition}

There is a decomposition of $X = X_g$ into locally closed subsets as follows. Denote by~$\gamma$ the positive orthant in $\QQ^r$. A face $\gamma_0 \preceq \gamma$ is called an \emph{$\bar{X}_g$-face} if there is $x \in \bar{X}_g$ with
\begin{equation*}
    x_i \ne 0 \iff e_i \in \gamma_0.
\end{equation*}
The orbit cones of $\bar{X}_g$ are precisely the cones $Q(\gamma_0)$, where $\gamma_0$ is an $\bar{X}_g$-face. An~\emph{$X$-face} is an $\bar{X}_g$-face $\gamma_0$ with $\lambda^\circ \subseteq Q(\gamma_0)^\circ$. Write $\rlv(X)$ for the set of $X$-faces. We have
\begin{equation*}
    X \ = \bigcup\limits_{\gamma_0 \in \rlv(X)} X(\gamma_0),
    \qquad
    X(\gamma_0) \ := \ \{ x \in \bar{X}_g;\ x_i \ne 0 \iff e_i \in \gamma_0 \} \quot H.
\end{equation*}

\begin{proposition}
\label{prop:locprops}
See \cite{ADHL}*{Cor.~3.3.1.8}.
For $X = X_g$ as in \cref{constr:hypersurface} the following hold.
\begin{enumerate}
    \item 
    $X$ is $\QQ$-factorial if and only if the cone $\lambda = \SAmple(X) \subseteq K_\QQ$ is full-dimensional.
    \item
    $X$ is locally factorial if and only if for every $X$-face~$\gamma_0 \preceq \gamma$, the group $K$ is generated by $Q(\gamma_0 \cap \ZZ^r)$.
\end{enumerate}
\end{proposition}

For locally factorial $X$ of Picard number two, \cref{prop:locprops} (ii) in particular yields the following two statements.

\begin{lemma}
\label{lemma:2gen}
See \cite{HLM}*{Lemma~5.6}. Let $X = X_g$ as in \cref{constr:hypersurface}. Assume $X$ is locally factorial and of Picard number two. Let $1 \le i,j \le r$ with $\lambda \subseteq \cone (w_i,w_j)$. Then either $w_i,w_j$ generate $K$ as a group, or $g$ has precisely one monomial of the form $T_i^{l_i} T_j^{l_j}$, where $l_i+l_j > 0$. 
\end{lemma}

Let $X = X_g$ from \cref{constr:hypersurface} be of Picard number two. Then we decompose the effective cone into the convex sets
\begin{equation*}
    \Eff(R_g) \ = \ \lambda^- \cup \lambda^\circ \cup \lambda^+,
\end{equation*}
where $\lambda^-$ and $\lambda^+$ are the convex polyhedral cones not intersecting $\lambda^\circ$ and the intersection $\lambda^+ \cap \lambda^-$ consists only of the origin.

\begin{lemma}
\label{lemma:3gen}
See \cite{HLM}*{Lemma~5.7}. Let $X = X_g$ as in \cref{constr:hypersurface}. Assume $X$ is locally factorial and of Picard number two. Let $1 \leq i < j < k \leq r$.
Then~$w_i, w_j, w_k$ generate $K$ as a group, provided that one of the following holds:
\begin{enumerate}
\item 
$w_i, w_j \in \lambda^-$, $w_k \in \lambda^+$ and $g$ 
has no monomial of the form $T_k^{l_k}$,
\item 
$w_i \in \lambda^-$, $w_j, w_k \in \lambda^+$ and $g$ 
has no monomial of the form $T_i^{l_i}$,
\item
$w_i \in \lambda^-$, $w_j \in \lambda^\circ$, $w_k \in \lambda^+$.
\end{enumerate}
Moreover, if (iii) holds,
then $g$ has a monomial of the form $T_j^{l_j}$
where $l_j$ is divisible by the order of the
factor group $K / \langle w_i, w_k \rangle$.
\end{lemma}

We turn to rational varieties with a complexity one torus action. For the general theory see \cites{HaHe, HaHiWr, HaSu}, also~\cite{ADHL}*{Sec.~3.4}. Here we focus on the case of hypersurface Cox rings.

\goodbreak

\begin{proposition}
\label{prop:cplx1-hscr}
For a Mori dream space $X$ with a hypersurface Cox ring the following are equivalent:
\begin{enumerate}
    \item $X$ admits a torus action of complexity one.
    \item The Cox ring of $X$ has an irredundant $\Cl(X)$-graded presentation
    \begin{equation*}
        \Cox(X) \ = \ \CC[T_1,\dots,T_r] / \bangle{g},
    \end{equation*}
    where $g$ is a trinomial consisting of pairwise coprime monomials.
\end{enumerate}
\end{proposition}

\begin{proof}
We write $K = \Cl(X)$ for the divisor class group and $R = \Cox(X)$ for the Cox ring of $X$. Assume that (i) holds. Then by~\cite{ADHL}*{Thm.~4.4.1.6} there is an irredundant~$K$-graded presentation
\begin{equation*}
    R \ = \ \CC[T_1,\dots,T_{r}] / \bangle{g_1,\dots,g_t},
\end{equation*}
such that the variables define pairwise non-associated $K$-prime generators and the polynomials $g_1,\dots,g_{t} \in \CC[T_1,\dots,T_{r}]$ are homogeneous trinomials, each one consisting of pairwise coprime monomials. Moreover, since $X$ has a hypersurface Cox ring, there is an irredundant $K$-graded presentation
\begin{equation*}
    R \ = \ \CC[T_1,\dots,T_{r'}] / \bangle{g}.
\end{equation*}
The graded isomorphism between these two presentations of $R$ lifts to a graded isomorphism $\CC[T_1,\dots,T_{r}] \rightarrow \CC[T_1,\dots,T_{r'}]$, see \cite{HaKeWo}*{Lemma~2.4}. This yields~$r = r'$ and $t=1$.
Now assume that (ii) holds. Using \cite{ADHL}*{Constr.~3.2.4.2} we construct the \emph{gradiator} of $g$. This is the maximal effective grading of $\KK[T_1,\dots,T_r]$ for which the variables~$T_1,\dots, T_r$ and the polynomial $g$ are homogeneous. Geometrically, this grading yields an effective action of a quasitorus $H_g$ on $X$. The coprimeness of the monomials of $g$ guarantees that the quasitorus $H_g$ contains a torus $T$ of dimension~$\dim(T) = \dim(X) - 1$. Thus $X$ admits an effective torus action of complexity one.
\end{proof}

We conclude this section by quoting two results used in the proof of Theorem \ref{thm:fano-4d-lf-rho=2}.
For a torsion-free grading group $K$ the notions of $K$-factoriality and factoriality coincide, see \cite{ADHL}*{Thm.~3.4.1.11}.

\begin{remark}\label{rem:facttrinom}
See \cite{HaHe}*{Thm.~1.1}.
For $l_1,l_2,l_3 \in \ZZ_{\ge 0}^r$ assume that the monomials~$T^{l_1}, T^{l_2}, T^{l_3} \in \CC[T_1, \dotsc, T_r]$ are pairwise coprime. Then the ring
\begin{equation*}
    R \ = \ \CC[T_1, \dotsc, T_r] / \langle T^{l_1} + T^{l_2} + T^{l_3} \rangle
\end{equation*}
is a unique factorization domain if and only if the integers
$\gcd(l_1)$, $\gcd(l_2)$ and~$\gcd(l_3)$ are pairwise coprime.
\end{remark}

\begin{remark}
\label{prop:hypmov}
See \cite{HLM}*{Prop.~2.4}.
Let $X = X_g$ as in \cref{constr:hypersurface}. Then we have
\begin{equation*}
    \mu \ \in \bigcap\limits_{1 \le i < j \le r} \cone(\, w_k;\ k \ne i, k \ne j\, ) \ \subseteq \ K_\QQ.
\end{equation*}
\end{remark}

\section{Picard group and smoothability}
\label{sec:picandsmooth}

In this section we establish two general facts, being essential for our proof of Theorem \ref{thm:fano-4d-lf-rho=2}. The first is \cref{prop:torfree}, which shows that in our setting the Picard group is always torsion-free. The second is \cref{prop:general-is-smooth}, which in particular gives rise to an explicit smoothing procedure in the case of an ample relation degree $\mu = \deg(g)$ in the Cox ring; see also \cref{rem:comp-smooth-hs}.

\begin{proposition}
\label{prop:torfree}
Let $X = X_g$ as in \cref{constr:hypersurface}. Assume that $X$ is $\QQ$-factorial, Fano, of Picard number two and admits a torus action of complexity one. Then the Picard group $\Pic(X)$ is torsion-free.
\end{proposition}

\begin{proof}
By \cite{ADHL}*{Cor.~3.3.1.6} we have the identity
\begin{equation*}
    \Pic(X)
    \ = \ 
    \bigcap\limits_{\gamma_0\, X\text{-face}}
    Q(\lin(\gamma_0)\cap\ZZ^r).
\end{equation*}
It therefore suffices to show that there is a two-dimensional $X$-face.
By \cref{prop:cplx1-hscr} we may assume that $g \in S$ is a trinomial consisting of pairwise coprime monomials.
We write $\rho_1,\dots,\rho_s$ for the rays generated by the generator degrees~$w_1,\dots,w_7$. The effective cone of $R := R_g$ is given by $\Eff(R) = \rho_1 + \rho_s$.
We distinguish two cases. First, assume $\mu = \deg(g)$ is contained in $\Eff(R)^\circ$. In this case we can find generator degrees $w_i, w_j$ that satisfy the following conditions:
\begin{enumerate}
    \item $\lambda^\circ \subseteq \cone(w_i,w_j)$.
    \item $\mu \in \cone(w_i,w_j)^\circ$.
    \item $g$ does not contain a monomial of the form $T_i^{l_i}T_j^{l_j}$.
\end{enumerate}
Explicitly, we do the following: Taking $w_i \in \rho_1$ and $w_j \in \rho_s$ satisfies (i) and (ii). If $g$ contains a monomial of the form $T_i^{l_i}T_j^{l_j}$, then, since $\mu$ is contained in the interior of $\Eff(R)$, the exponents $l_i$ and $l_j$ are positive. The definition of $\Mov(R)$ and \cref{prop:hypmov} ensure that we can replace either $w_i$ or $w_j$ with a different generator degree, such that the new pair $(w_{i'},w_{j'})$ still satisfies (i) and (ii). Since the monomials of $g$ are pairwise coprime, this pair also satisfies (iii). The face~$\gamma_0$ of the positive orthant $\gamma$ spanned by $e_{i'}$ and $e_{j'}$ is thus a two-dimensional $X$-face.

Now assume that $\mu$ lies on one of the bounding rays of $\Eff(R_g)$. We may assume that $\mu \in \rho_1$ holds. Since $g$ is a trinomial and its monomials are pairwise coprime, the ray $\rho_1$ contains at least three generator degrees. If $\rho_1$ contains four or more generator degrees, then there is $w_i \in \rho_1$ such that $g$ does not contain a monomial of the form $T_i^{l_i}$. Choose any $w_j \in \rho_s$. Then the face~$\gamma_0 := \cone(e_i,e_j)$ is again a two-dimensional $X$-face. Now assume that $\rho_1$ contains exactly three generator degrees, say $w_1,w_2$ and $w_3$. The ample cone $\lambda$ of~$X$ is of the form $\lambda = \rho_k + \rho_{k+1}$ for some~$k = 1,\dots,s-1$. If $\lambda \ne \rho_1+\rho_2$, then we take~$w_i$ and $w_j$ from each of the bounding rays of $\lambda$. The face $\gamma_0 := \cone(e_i,e_j)$ is again a two-dimensional~$X$-face. It remains to consider the case $\lambda = \rho_1+\rho_2$. Applying a unimodular transformation, we achieve that $\rho_1$ is the ray generated by $e_1$. We write $w_1 = (a_1,0)$, $w_2 = (a_2,0)$ and $w_3 = (a_3,0)$. Switching the roles of $w_1,w_2$ and $w_3$ if necessary, we may assume that $a_1 \ge a_2 \ge a_3$ holds. Homogeneity of $g$ yields $l_1a_1 = l_2a_2 = l_3a_3$. As $X$ is Fano, it's anticanonical class is ample. By \cref{prop:antican}, this means
\begin{equation*}
    (1-l_1)w_1 + w_2 + w_3 + w_4 + \dots + w_r \ \in \lambda^\circ \ = \ (\rho_1+\rho_2)^\circ.
\end{equation*}
The point $w := w_4+\dots+w_r$ is contained in the cone $\rho_2+ (-\rho_1)$. Thus, for the sum to lie in the interior of $\lambda$, we must have $(1-l_1)w_1 + w_2 + w_3 \in \rho_1$. This is equivalent to $(l_1-1)a_1 < a_2+a_3$. Since $a_1$ is at least as big as $a_2$ and $a_3$, this yields $l_1 = 2$. Homogeneity of $g$ thus yields $a_2 = 2a_1/l_2$ and $a_3 = 2a_1/l_3$. With this, the exponents $l_2$ and $l_3$ satisfy the following inequality:
\begin{equation*}
    1 \ < \ \frac{2}{l_2} + \frac{2}{l_3}.
\end{equation*}
Since $a_2 \ge a_3$, we have $l_3 \ge l_2$. Moreover, both exponents are at least two by irredundancy of the presentation of $R$. The triple $(l_1,l_2,l_3)$ is therefore one of the following:
\begin{equation*}
    (l_1,l_2,l_3) \ = \ (2,2,y),
    \qquad
    (l_1,l_2,l_3) \ = \ (2,3,3),
\end{equation*}
\begin{equation*}
    (l_1,l_2,l_3) \ = \ (2,3,4),
    \qquad
    (l_1,l_2,l_3) \ = \ (2,3,5),
\end{equation*}
where $y \ge 2$. Thus, by \cite{ArBrHaWr}*{Thm.~3.13} the variety $X$ has at most log terminal singularities. With this, we are in the situation of \cite{PaSha99}*{Prop.~2.1.2}, which tells us that the Picard group $\Pic(X)$ is torsion-free.
\end{proof}

\begin{proposition}
\label{prop:general-is-smooth}
Let $X = X_g$ as in \cref{constr:hypersurface} be locally factorial and of Picard number two. Assume that
\begin{equation*}
    \mu \ \in \ \SAmple(X) \cap \Mov(X)^\circ.
\end{equation*}
Then there is a non-empty open subset $U \subseteq S_\mu$ such that for all $h \in U$ the variety~$X_h \subseteq Z$ is smooth with divisor class group~$\Cl(X_h) = K$ and Cox ring~$\Cox(X_h) = R_h$.
\end{proposition}

The remainder of this section is devoted to the proof of \cref{prop:general-is-smooth}. We adopt the notation of \cref{constr:toric-cox} and \cref{constr:hypersurface}. A homogeneous polynomial $h \in S_\mu$ is called \emph{spread}, if every monomial $T^\nu \in S$ of degree $\mu = \deg(h)$ is a convex combination of monomials of $h$.
We say that $R_h$ is spread, if $h$ is spread, see \cite{HLM}*{Def.~4.3}.
Here we identify a monomial $T^\nu = T_1^{\nu_1}\cdots T_r^{\nu_r}$ with its exponent vector $\nu \in \QQ^r$.
If $h,h' \in S_\mu$ are spread, then the minimal ambient toric varieties~$Z_h$ of $X_h$ and $Z_{h'}$ of $X_{h'}$ coincide. Thus the toric variety $Z_\mu := Z_{h}$ is well-defined. It is called the \emph{$\mu$-minimal ambient toric variety}, see \cite{HLM}*{Def.~4.18}. The following two Propositions, originally \cite{HLM}*{Prop.~4.11} and \cite{HLM}*{Cor.~4.19}, are essential to the proof of \cref{prop:general-is-smooth}.

\begin{proposition}
\label{prop:bertini}
See \cite{HLM}*{Cor.~4.19}. In the setting of \cref{constr:hypersurface}, assume~$\rk(K) = 2$ and that $Z_\mu \subseteq Z$ is smooth. If $\mu \in \tau$ holds, then $\mu$ is basepoint free. Moreover, then there is a non-empty open subset of polynomials $g \in S_\mu$ such that $X_g$ is smooth.
\end{proposition}

\begin{proposition}
\label{prop:lefshetz}
See \cite{HLM}*{Prop.~4.11}. In the setting of \cref{constr:hypersurface} assume that $K$ is of rank at most $r-4$ and torsion-free, there is $g \in S_\mu$ such that $T_1,\dots,T_r$ define primes in $R_g$, we have $\mu \in \tau^\circ$ and $\mu$ is basepoint free on $Z$. Then there is a non-empty open subset of polynomials $g \in S_\mu$ such that $R_g$ is a UFD.
\end{proposition}

For the rest of this section it is assumed that we have $X = X_g$ as in \cref{constr:hypersurface} and that $X$ is locally factorial and of Picard number two. For any homogeneous $h \in S_\mu$ we denote by $\lambda_h \in \Lambda(R_h)$ the smallest GIT-cone that contains~$\tau$. Note that local factoriality of $X$ in particular implies $\QQ$-factoriality. Thus by \cref{prop:locprops} (i) the cone $\lambda$ is full-dimensional.

\begin{lemma}
\label{lemma:ample-cones-inclusion}
Let $h \in S_\mu$ such that each monomial of $g$ is also a monomial of $h$. Then $\lambda \subseteq \lambda_h$ holds.
\end{lemma}

\begin{proof}
The cone $\lambda_h \in \Lambda(R_h)$ is the smallest GIT-cone that contains $\tau$.
Thus in the case $\tau = \lambda$ there is nothing to show.
So we may assume that $\tau \subsetneq \lambda$ holds.
We write~$\lambda = \cone(w_i,w_j)$ and $\tau = \cone(w_k,w_l)$.
Since $\tau$ is a proper subset of $\lambda$, one of its ray generators is contained in the interior of $\lambda$, say $w_k \in \lambda^\circ$.
By \cite{HLM}*{Prop.~2.8} the degree~$w_k$ is the only generator degree that is contained in the interior of~$\lambda$.
Moreover, $\mu$ lies on the ray of $w_k$ and the relation $g$ contains a monomial of the form $T_k^{l_k}$.
Since $h$ contains each monomial of $g$, it also contains the monomial~$T_k^{l_k}$.
Therefore, each projected $\bar{X}_h$-face $Q(\gamma_0)$, that contains $\tau$, necessarily also contains generator degrees on both sides of $w_k$. Since $w_k$ is the only generator degree in the interior of $\lambda$, the cone $Q(\gamma_0)$ also contains $\lambda$.
This implies the assertion.
\end{proof}

\begin{proposition}
\label{prop:mumin-is-smooth}
The $\mu$-minimal ambient toric variety $Z_\mu \subseteq Z$ is smooth.
\end{proposition}

\begin{proof}
Let $h \in S_\mu$ spread such that each monomial of $g$ is also a monomial of $h$ and let $\gamma_0 \preceq \gamma$ with $\lambda_h^\circ \subseteq Q(\gamma_0)^\circ$.
Write $\gamma_0 = \cone(e_{i_1},\dots,e_{i_m})$.
By \cref{prop:locprops}~(ii) we have to show that either $w_{i_1},\dots,w_{i_m}$ generate $K$ as a group, or $\gamma_0$ is not an~$\bar{X}_{h}$-face.
Assume that $w_{i_1},\dots,w_{i_m}$ do not generate $K$.
We show that $\gamma_0$ is not an $\bar{X}_h$-face.
By \cref{lemma:ample-cones-inclusion}, we have $\lambda^\circ \subseteq Q(\gamma_0)^\circ$.
Since $\lambda$ is of full dimension, $\gamma_0$ is at least two-dimensional.
In particular we have $m \ge 2$.
None of the degrees~$w_{i_1},\dots,w_{i_m}$ lies in~$\lambda^\circ$:
If one of them did, say $w_{i_1} \in \lambda^\circ$, then by \cite{HLM}*{Prop.~2.8} it is the only generator degree in the interior of $\lambda$ and $g$ contains a monomial of the form $T_{i_1}^{l_{i_1}}$. Moreover, in this case we have $m \ge 3$.
We may assume that $w_{i_2}$ and $w_{i_3}$ each lie in one of the bounding rays of $Q(\gamma_0)$.
The degrees $w_{i_2}, w_{i_3}$ do not generate~$K$ as a group.
Thus, since $X$ is locally factorial, by \cref{prop:locprops}~(ii) the cone spanned by $e_{i_2}$ and~$e_{i_3}$ is not a $\bar{X}$-face.
This means that $g$ contains a monomial of the form $T_{i_2}^{l_{i_2}}T_{i_3}^{l_{i_3}}$.
But then the cone spanned by $e_{i_1},e_{i_2}$ and $e_{i_3}$ is an $\bar{X}$-face and thus~$w_{i_1}, w_{i_2}, w_{i_3}$ generate~$K$.
A contradiction.
Thus none of the degrees~$w_{i_1},\dots,w_{i_m}$ lies in $\lambda^\circ$.
We may assume that the generator degrees are sorted in such a way that for all~$v \in \lambda^\circ$ we have~$\det(v,w_{i_j}) < 0$ if~$j \le k$ and $\det(v,w_{i_j}) > 0$ if $j > k$ for some fixed~$1 \le k \le m$.
We show that either $k=1$ or $k+1=m$ holds.
If $1 < k$ and $k+1 < m$, then we have $m \ge 4$ and neither $w_{i_1},w_{i_{m-1}}$, nor $w_{i_2},w_{i_m}$ generate $K$.
By local factoriality of $X$ and \cref{prop:locprops}~(ii), the relation $g$ then contains monomials of the form~$T_{i_1}^{l_{i_1}}T_{i_{m-1}}^{i_{m-1}}$ and $T_{i_2}^{l_{i_2}}T_{i_m}^{i_m}$.
Thus $\gamma'_0 = \cone(w_{i_1},w_{i_2},w_{i_{m-1}},w_{i_m})$ is an~$\bar{X}$-face and~$w_{i_1},w_{i_2},w_{i_{m-1}},w_{i_m}$ generate $K$. A contradiction.
We may thus assume that~$k=1$ holds.
If $m = 2$, then $g$ contains a monomial of the form $T_{i_1}^{l_{i_1}}T_{i_2}^{l_{i_2}}$ and this is the only monomial in $S_\mu$ only depending on these two variables. Therefore~$\gamma_0$ is not an $\bar{X}_h$-face.
If $m > 2$, then by \cref{prop:locprops}~(ii), the relation $g$ contains a monomial of the form $T_{i_1}^{l_{i_1}}$ and this is the only monomial in $S_\mu$ only depending on the variables $T_{i_1},\dots,T_{i_m}$.
Thus also in this case $\gamma_0$ is not an $\bar{X}_{h}$-face.
\end{proof}

\begin{lemma}
\label{lemma:degree-semiample}
If $\mu \in \lambda$ holds, then we also have $\mu \in \tau$.
\end{lemma}

\begin{proof}
In case $\tau$ and $\lambda$ coincide, there is nothing to show.
So we assume that~$\tau \subsetneq \lambda$ holds.
Write $\lambda = \cone(w_i,w_j)$ and $\tau = \cone(w_k,w_l)$. Since $\tau$ is a proper subset of~$\lambda$, one of the generator degrees in its bounding rays lies in the interior of $\lambda$, say~$w_k \in \lambda^\circ$.
By \cite{HLM}*{Prop.~2.8}, it is the only generator degree that lies in $\lambda^\circ$ and~$g$ contains a monomial of the form $T_k^{l_k}$.
This shows that $\mu \in \tau$ holds.
\end{proof}

\begin{proposition}
\label{prop:general-is-ufd}
Assume that $\mu \in \lambda \cap \Mov(R_g)^\circ$ holds. Then there is a non-empty open subset $U \subseteq S_\mu$ such that $R_h$ is a UFD for each $h \in U$.
\end{proposition}

\begin{proof}
By \cref{lemma:degree-semiample} the relation degree $\mu$ is contained in the cone $\tau$.
We distinguish two cases.
First assume that $\mu \in \tau^\circ$ holds.
As $Z_\mu$ is smooth by \cref{prop:mumin-is-smooth}, the class $\mu$ is basepoint free by \cref{prop:bertini}.
We can thus apply~\cref{prop:lefshetz}, which yields the assertion.
Now assume that $\mu \in \partial\tau$ holds. Let $\tau_\mu \in \Lambda(S)$ the unique GIT-cone that contains $\mu$ in its interior, ie. $\tau_\mu$ is the bounding ray of~$\tau$ containing $\mu$.
We write $Z(\tau_\mu) := \bar{Z}^{ss}(\tau_\mu)\quot H$ for the projective toric variety associated with $\tau_\mu$ as in \cref{constr:toric-cox}.
We show that $\mu$ is basepoint free on $Z(\tau_\mu)$.
Note that $\mu$ is semiample. Thus by \cite{CoLiSch11}*{Thm.~6.3.12} it suffices to show that $\mu$ is Cartier on $Z(\tau_\mu)$. By \cite{ADHL}*{Cor.~3.3.1.6} this is the case if and only if
\begin{equation*}
    \mu \ \in \ \bigcap\limits_{\gamma_0 \in \rlv(Z(\tau_\mu))} Q(\gamma_0 \cap \ZZ^r)
\end{equation*}
holds.
Let $\gamma_0 \in \rlv(Z(\tau_\mu))$.
If $Q(\gamma_0)$ is two-dimensional, then~$\lambda^\circ \subseteq Q(\gamma_0)^\circ$ holds and by \cref{prop:mumin-is-smooth} we have $Q(\gamma_0 \cap \ZZ^r) = K$.
So assume that $Q(\gamma_0)$ is one-dimensional, ie. $Q(\gamma_0) = \tau_\mu$.
We distinguish two cases.
First assume $\mu \in \lambda^\circ$.
Then by \cite{HLM}*{Prop.~2.8}, $\tau_\mu$ contains a single generator degree $w_k$ and $\mu$ is a multiple of~$w_k$.
Thus in this case $\mu \in Q(\gamma_0\cap\ZZ^r)$ holds.
Now we assume that $\mu \in \partial\lambda$ holds. Then $\tau_\mu$ is one of the bounding rays of $\lambda$.
Write $\gamma_0 = \cone(e_{i_1},\dots,e_{i_m})$ and let~$w_k$ a generator degree in the other bounding ray of $\mu$.
If $w_{i_1},\dots,w_{i_m},w_k$ generate $K$, then $\mu$ is a linear combination of $w_{i_1},\dots,w_{i_m}$ and thus $\mu \in Q(\gamma_0\cap\ZZ^r)$ holds.
If they do not generate $K$, then by \cref{prop:locprops}~(ii) the relation $g$ contains a single monomial only depending on $T_{i_1},\dots,T_{i_m},T_k$.
Since $\mu$ is contained in $\tau_{\mu}$, this monomial can not depend on $T_k$.
Thus $\mu$ is again a linear combination of the degrees $w_{i_1},\dots,w_{i_m}$ and thus $\mu \in Q(\gamma_0\cap\ZZ^r)$ holds. This shows that $\mu$ is basepoint free on $Z(\tau_\mu)$. We again apply \cref{prop:lefshetz}, which yields the assertion.
\end{proof}

\begin{proof}[Proof of \cref{prop:general-is-smooth}]
The set $U_1 \subseteq S_\mu$ of polynomials $h$ such that $R_h$ is a UFD is open and non-empty by \cref{prop:general-is-ufd}.
The set $U_2 \subseteq S_\mu$ of polynomials $h$ such that $T_1,\dots,T_r$ form a minimal system of non-associated $K$-prime generators in $R_h$ is open by \cite{HLM}*{Prop.~4.10} and since $R_g$ has that property, the set $U_2$ is non-empty.
Finally, the set $U_3 \subseteq S_\mu$ of polynomials $h$ such that $X_h$ is smooth is open and non-empty by \cref{prop:bertini} and \cref{prop:mumin-is-smooth}.
Now for any $h$ in the intersection
\begin{equation*}
    U \ = \ U_1 \cap U_2 \cap U_3
\end{equation*}
the affine $K$-algebra $R_h$ is a UFD and the variables $T_1,\dots,T_r$ define pairwise non-associated primes in $R_h$. Being a UFD implies that $R_h$ is normal and that the $K$-grading is factorial. The grading is also pointed, as this is inherited from $S$. In particular this implies that $R_h^* = \CC^*$ holds. We are thus in the situation of \cref{constr:hypersurface}. So $R_h$ is the Cox ring of the smooth projective variety $X_h$.
\end{proof}

\section{Proof of the main result: An example case}
\label{sec:proof}

We enter the proof of Theorem \ref{thm:fano-4d-lf-rho=2}. We restrict the discussion to an example case, see \cref{thm:s=3}. However, we will still encounter all types of arguments of the full proof. The complete elaboration will appear in~\cite{Ba24}. We start by fixing the setting.

\begin{setting}
\label{set:rho2}
Let $X$ be a locally factorial Fano fourfold of Picard number $\rho = 2$ with a hypersurface Cox ring $R = \Cox(X)$.
Write $K = \Cl(X)$ and let
\begin{equation*}
    R \ = \ \CC[T_1, \dotsc, T_7] / \langle g \rangle
\end{equation*}
an irredundant $K$-graded presentation of $R$ such that the variables $T_1,\dots,T_7$ define pairwise non-associated $K$-prime generators of $R$.
We have $X = X_g$ as in \cref{constr:hypersurface}.
We assume that $X$ is of complexity~$c = 1$.
By \cref{prop:torfree} the group~$K$ is torsion-free and we identify $K = \ZZ^2$.
By \cite{ADHL}*{Thm.~3.4.1.11} the ring $R$ is a UFD.
By \cref{prop:cplx1-hscr} and \cref{rem:facttrinom} we may thus assume that $g$ satisfies the following two conditions.
\begin{enumerate}
    \item[(C1)]
    The relation $g$ is of the form
    \begin{equation*}
        g \ = \ T^{l_1} + T^{l_2} + T^{l_3}
    \end{equation*}
    with $l_1, l_2, l_3 \in \ZZ_{\ge 0}^{7}$ such that each variable $T_1,\dots,T_7$ divides at most one monomial of $g$.
    
    \item[(C2)]
    The integers $\gcd(l_1)$, $\gcd(l_2)$ and $\gcd(l_3)$ are pairwise coprime.
\end{enumerate}
We turn to the grading map $Q$ of $R$. Write $w_i := Q(e_i) = \deg(T_i)$ and $\mu := \deg(g)$ for the degrees in $K$, also when regarded in~$K_\QQ$.
Suitably ordering $w_1,\dots,w_7$ we ensure
\begin{equation*}
    \det(w_i, w_j) \ \geq \ 0
\end{equation*}
whenever $i \le j$.
Some of the degrees $w_i$ may share a common ray.
We denote by $s$ the number of distinct rays $\rho_1,\dots,\rho_s$ generated by the degrees~$w_1,\dots,w_7$,
\begin{equation*}
    s \ := \ \#\{ \, \cone(w_i);\ i = 1,\dots,7 \, \}.
\end{equation*}
Moreover, we denote the number of generator degrees $w_j$ contained in the ray $\rho_i$ by $n_i$.
We have $s \le 7$ and
\begin{equation*}
    n_1 + \dots + n_s \ = \ 7.
\end{equation*}
Each ray $\rho_j$ in the GIT-fan $\Lambda(R)$ is of the form $\rho_j = \cone(w_i)$ for some $w_i$, but the converse may not hold.
As $X$ is locally factorial, it is in particular $\QQ$-factorial.
By \cref{prop:locprops}~(i) this means that the cone $\lambda = \SAmple(X)$ is full-dimensional.
As a GIT-cone in $K_\QQ = \QQ^2$, the cone $\lambda$ is the intersection of two projected $\bar{X}_g$-faces and thus each bounding ray of $\lambda$ contains at least one of the degrees $w_i$.
We decompose the effective cone $\Eff(R)$ into the three convex sets
\begin{equation*}
    \Eff(X) \ = \ \lambda^- \cup \lambda^\circ \cup \lambda^+,
\end{equation*}
where $\lambda^-$ and $\lambda^+$ are the convex polyhedral cones not intersecting $\lambda^\circ = \Ample(X)$ and the intersection $\lambda^+ \cap \lambda^-$ consists only of the origin.
Each of the cones~$\lambda^+$ and~$\lambda^-$ contains at least two of the generator degrees~$w_1,\dots,w_7$.
However, $\lambda^+$ as well as $\lambda^-$ may be one-dimensional. The following picture illustrates the situation for the case $s = 4$.
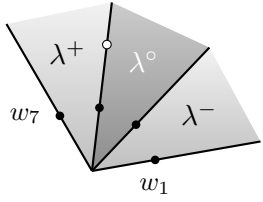
\hfill\break
\begin{center}
\begin{tikzpicture}[scale=0.75]

\tikzmath{
\nmu=3;
\xlambda=2;
let \len=3cm;
\xphi1=10;
\xphi2=120;
\srays=4;
\phidiff=(\xphi2 - \xphi1) / 3;
}

\foreach \i[remember=\i as \lasti (initially 1)] in {1, ...,\srays}
{
    \tikzmath{
    coordinate \c;
    coordinate \lastc;
    \c = (\i*\phidiff - \phidiff + \xphi1:\len);
    \lastc = (\lasti*\phidiff - \phidiff + \xphi1:\len);
    }
    
	\path[top color=gray!10!, bottom color=gray!50!] 
	(0,0)--(\lastc)--(\c)--cycle;
}

    \tikzmath{
    coordinate \rightboundary;
    coordinate \leftboundary;
    \rightboundary = (\xlambda*\phidiff - \phidiff + \xphi1:\len);
    \i=\xlambda + 1;
    \leftboundary = (\i*\phidiff - \phidiff + \xphi1:\len);
    }
    
    \path[top color=gray!40!, bottom color=gray!90!]
	(0,0)--(\rightboundary)--(\leftboundary)--cycle;
	
	\draw[color=white] ($0.375*(\leftboundary) + 0.375*(\rightboundary)$) node {\large$\lambda^\circ$};
    \draw[color=black] ($0.375*(\xphi1:\len) + 0.375*(\rightboundary)$) node {\large$\lambda^-$};
    \draw[color=black] ($0.375*(\leftboundary) + 0.375*(\xphi2:\len)$) node {\large$\lambda^+$};


\tikzmath{
coordinate \c;
\c = (\xphi1:\len);
}

\draw[thick,color=black] (0,0) -- (\c);

\tikzmath{
coordinate \u;
\ulen=0.75;
\u = (\xphi1:\len/2*\ulen);
}
\fill (\u) circle (0.5ex) node[label=below:$w_1$]{};


\tikzmath{
coordinate \c;
\c = (\phidiff + \xphi1:\len);
}

\draw[thick,color=black] (0,0) -- (\c);

\tikzmath{
coordinate \u;
\ulen=0.75;
\u = (\phidiff + \xphi1:\len/2*\ulen);
}
\fill (\u) circle (0.5ex);


\tikzmath{
coordinate \c;
\c = (2*\phidiff + \xphi1:\len);
}

\draw[thick,color=black] (0,0) -- (\c);

\tikzmath{
coordinate \u;
\ulen=0.75;
\u = (2*\phidiff + \xphi1:\len/2*\ulen);
}
\fill (\u) circle (0.5ex);


\tikzmath{
coordinate \c;
\c = (\xphi2:\len);
}

\draw[thick,color=black] (0,0) -- (\c);

\tikzmath{
coordinate \u;
\ulen=0.75;
\u = (\xphi2:\len/2*\ulen);
}
\fill (\u) circle (0.5ex) node[label=left:$w_7$]{};

\tikzmath{
coordinate \reldeg;
\reldeg = (\nmu*\phidiff - \phidiff + \xphi1:0.75*\len);
}
\filldraw[fill=white, draw=black] (\reldeg) circle (0.5ex);

\end{tikzpicture}
\end{center}
The black dots represent the generator degrees $w_1,\dots,w_7$. The white dot represents the relation degree $\mu$. In this example the cones $\lambda^+$ and $\lambda^-$ are full-dimensional.
\end{setting}

\begin{remark}
Let $X$ be as in \cref{set:rho2}.
\begin{enumerate}
    \item
    The variety $X$ is uniquely determined by its \emph{specifying data} $(Q,g)$: The variety~$X(Q,g) := X_g$ as in \cref{constr:hypersurface} satisfies $X \cong X(Q,g)$.
    \item Up to reversing order, the tuple $(n_1,\dots,n_s)$ is invariant under automorphisms of $K$. We call it the \emph{degree constellation} of $X$.
\end{enumerate}
\end{remark}

\begin{remark}
\label{rem:isotest}
Given specifying data $(Q,g)$ and $(Q',g')$, we need criteria to decide computationally whether or not the varieties $X(Q,g)$ and $X(Q',g')$ are isomorphic. Here we can make use of \cite{BaHa22}*{Prop.~3.4}: If $X(Q,g)$ and $X(Q',g')$ are isomorphic, then~$g$ and~$g'$ coincide up to permutation of variables.
\end{remark}

\cref{set:rho2} divides the proof of Theorem \ref{thm:fano-4d-lf-rho=2} into six cases, according to the number $s$ of rays spanned by the degrees $w_1,\dots,w_7$. Here we treat the case $s = 3$ as an example, showing basically all arguments used in the full proof, see \cite{Ba24}.

\begin{theorem}\label{thm:s=3}
The tables from \ref{class:s=3-sample}, \ref{class:s=3-nsample} and \ref{class:s=3-series} provide specifying data~$(Q,g)$ for~$223$ sporadic cases and $4$ infinite series of locally factorial Fano fourfolds of Picard number $\rho = 2$ and complexity $c = 1$ with a hypersurface Cox ring and~$s = 3$.
Moreover, any locally factorial Fano fourfold with a hypersurface Cox ring and invariants $(\rho,c,s) = (2,1,3)$ is isomorphic to precisely one~$X(Q,g)$ with $(Q,g)$ from these tables.
\end{theorem}

The proof of \cref{thm:s=3} splits into two parts.
First, with the tools provided in \cref{sec:basics} we verify that each specifying data $(Q,g)$ from the tables in \ref{class:s=3-sample}, \ref{class:s=3-nsample} and \ref{class:s=3-series} defines a locally factorial Fano fourfold $X(Q,g)$ with a hypersurface Cox ring and invariants $(\rho,c,s) = (2,1,3)$.
Moreover, with the help of \cref{rem:isotest} one verifies that distinct specifying data from the tables in \ref{class:s=3-sample}, \ref{class:s=3-nsample} and \ref{class:s=3-series} define pairwise non-isomorphic varieties.
The second part is to show that any locally factorial Fano fourfold with a hypersurface Cox ring and invariants $(\rho,c,s) = (2,1,3)$ is isomorphic to some $X(Q,g)$ with $(Q,g)$ from these tables.
We divide the proof of this into the two general cases
\begin{equation*}
    \mu \ \in \ \SAmple(X),
    \qquad\qquad
    \mu \ \not\in \ \SAmple(X).
\end{equation*}
The case $\mu \in \SAmple(X)$ will be treated in \cref{prop:s=3-sample}. In \cref{prop:s=3-nsample} we treat the case $\mu \not\in \SAmple(X)$. We start with some general observations that will be used throughout the proof.

\begin{lemma}
\label{lemma:fano-ufp}
Let $ b > a > 1$ be coprime integers. If $ab \le 2 + a + b$ holds, then we have~$a=2$ and $b = 3$.
\end{lemma}

\begin{proof}
Dividing both sides of the inequality in the assertion by $ab$, we obtain
\begin{equation*}
    1 \ \le \ \frac{2}{ab} + \frac{1}{b} + \frac{1}{a}.
\end{equation*}
We have $ab \ge 6$ as well as $1/b < 1/a$. With this we obtain $a = 2$. The original inequality turns into $2b \le 4 + b$. As $b > a$ holds and $a$ and $b$ are coprime, this is only fulfilled for $b = 3$.
\end{proof}

\begin{lemma}
\label{lemma:primgen}
In the situation of \cref{set:rho2}, for each ray $\rho_j$, at most two of the generator degrees $w_i$ contained in $\rho_j$ are non-primitive lattice points.
\end{lemma}

\begin{proof}
Assume that $\rho_j$ contains three non-primitive generator degrees $w_{i_1}, w_{i_2}$ and~$w_{i_3}$. Applying a unimodular transformation if necessary, we may assume that~$\rho_j$ is the ray generated by the first standard basis vector. Write
\begin{equation*}
    w_{i_1} \ = \ (a_{i_1},0),
    \qquad
    w_{i_2} \ = \ (a_{i_2},0),
    \qquad
    w_{i_3} \ = \ (a_{i_3},0).
\end{equation*}
As the degrees $w_{i_1}, w_{i_2}, w_{i_3}$ are non-primitive, we have $a_{i_1},a_{i_2},a_{i_3} > 1$. By \cite{HLM}*{Prop.~2.8} the ray $\rho_j$ does not lie in the interior of $\lambda$. Thus there is a generator degree~$w_k$ such that $\lambda$ is contained in the two-dimensional cone $\tau_m = \cone(w_{i_m},w_k)$ for $m = 1,2,3$. Since $w_{i_m}$ is not primitive, the degrees $w_{i_m}$ and $w_k$ do not generate~$K$ as a group. \cref{lemma:2gen} thus tells us that the relation $g$ contains monomials of the form $T_{i_m}^{l_{i_m}}T_k^{l_k(m)}$. As the monomials of $g$ a pairwise coprime, at least two of the exponent $l_k(m)$ must be zero. By homogeneity of $g$ we obtain~$l_k(1) = l_k(2) = l_k(3) = 0$ and $\mu$ lies in $\rho_j$. In particular, we have
\begin{equation*}
    \mu \ = \ l_{i_1}w_{i_1} \ = \ l_{i_2}w_{i_2} \ = \ l_{i_3}w_{i_3}.
\end{equation*}
By the condition (C2) from \cref{set:rho2}, the integers $l_{i_1},l_{i_2},l_{i_3}$ are pairwise coprime. Moreover, they are all bigger than one by irredundancy of the presentation of $R$. Let $p$ a prime divisor of $l_{i_1}$. Then $p$ must divide both $a_{i_2}$ and $a_{i_3}$. In particular, $a_{i_2}$ and $a_{i_3}$ are not coprime. Thus the three degrees $w_{i_2}, w_{i_3}$ and $w_k$ do not generate $K$ as a group. By \cref{lemma:3gen} the relation~$g$ must therefore contain a monomial of the form $w_k^{l_k}$. This is a contradiction to the position of $\mu$. Thus at least one of $w_{i_1}, w_{i_2}, w_{i_3}$ is primitive.
\end{proof}

\begin{lemma}
\label{lemma:nonsemiample}
In the situation of \cref{set:rho2}, assume that $\mu \in \lambda^+ \backslash \lambda$ holds.
Then the following hold.

\begin{enumerate}
\item
The cone $\lambda$ is regular and
every generator degree lying on its boundary is primitive.

\item
All generator degrees contained in $\lambda^-$ coincide. 
In particular, $n_1 \geq 2$ holds and $\lambda^- = \rho_1$ is a bounding ray of~$\lambda$.
\end{enumerate}
\end{lemma}

\begin{proof}
We prove~(i).
Let $1 \le i < j \le 7$ such that $\lambda = \cone(w_i,w_j)$ holds.
Since the relation degree $\mu$ is not contained in $\lambda$, the relation $g$ does not contain a monomial of the form $T_i^{l_i} T_j^{l_j}$.
By Lemma \ref{lemma:2gen}, the generator degrees $w_i$ and $w_j$ generate $K$ as a group.

We prove~(ii).
Let $1 \le i < j \le 7$ such that $\lambda = \cone(w_i,w_j)$ holds.
By (i) we may assume that $w_i = (1,0)$ and $w_j = (0,1)$ holds.
We write $w_1 = (a_1,-b_1)$ for some~$a_1, b_1 \in \ZZ_{\ge 0}$.
Applying \cref{lemma:2gen} to the generator degrees $w_1$ and~$w_j$ shows that $a_1 = 1$ holds.
So in order to verify item (ii) it suffices to show that $b_1 = 0$ holds.
For this we first show that $\Eff(R)$ contains a lattice point $v \in \ZZ^2$ of the form~$v = (-a, 1)$ for some $a \in \ZZ_{\geq 1}$.
Note that $w_5, w_6, w_7$ do not lie in $\lambda$.
Write~$w_k = (-a_k,b_k)$ for $k = 5, 6, 7$.
If~$b_k = 1$ holds for one of those, then we have found such a point $v$.
Otherwise we have $b_5, b_6, b_7 > 1$.
Thus $\det(w_i,w_k) = b_k > 1$ holds and by \cref{lemma:2gen}, the relation $g$ contains monomials of the form~$T_i^{k_i} T_5^{k_5}$, $T_i^{l_i} T_6^{l_6}$ and~$T_i^{m_i} T_7^{m_7}$.
Since only at most one monomial of $g$ is divisible by $T_i$, we conclude that two of the exponents $k_i, l_i$ and $m_i$ must be zero.
Homogeneity of $g$ thus implies that two of the generator degrees $w_5, w_6, w_7$ lie on the ray through~$\mu$.
Let $v = (v_1, v_2) \in \ZZ^2$ denote the primitive lattice vector on this ray.
We apply \cref{lemma:3gen} to $w_i$ and the two generator degrees on the ray through $\mu$ to infer~$v_2 = \det(w_i, v) = 1$.
Thus $v$ is a lattice point of the desired form.
From $w_1 \in \lambda^-$ and $v \in \lambda^+$ we infer
\begin{equation*}
    0 \ < \ \det(w_1,v) \ = \ 1 - a b_1.
\end{equation*}
As $a$ is positive, this inequality can only be fulfilled by $b_1 = 0$. Hence $w_1 = (1,0)$ holds, which proves the assertion.
\end{proof}

\begin{proposition}
\label{prop:s=3-sample}
Let $X$ be as in \cref{set:rho2} with~$s = 3$ and assume that $\mu \in \lambda$ holds. Then $X$ is isomorphic to $X(Q,g)$ with specifying data $(Q,g)$ appearing in \cref{class:s=3-sample}.
\end{proposition}

\begin{proof}
We divide the proof into the two cases $\mu \in \Mov(R)^\circ$ and $\mu \in \partial\Mov(R)$.

\hfill\break
\noindent\emph{Case:} $\mu \in \Mov(R)^\circ$.
We are in the situation of \cref{prop:general-is-smooth}. Thus, for a general polynomial~$h \in \CC[T_1,\dots,T_7]$ of degree $\deg(h) = \mu$, the projective variety~$X_h$ is smooth with divisor class group $\Cl(X_h) = K$ and Cox ring $\Cox(X_h) = R_h$. Moreover, by \cref{prop:antican} $X_h$ is Fano. Thus, $X_h$ is a smooth Fano fourfold of Picard number two with a spread hypersurface Cox ring. In particular, up to unimodular equivalence, the grading matrix $Q = (w_1,\dots,w_7)$ together with the relation degree~$\mu = \deg(g)$ appear in the classification list in \cite{HLM}*{Thm.~1.1}. For each such entry~$(Q,\mu)$ with $s = 3$ we determine all trinomials $g$ of degree~$\deg(g) = \mu$ that satisfy the conditions~(C1) and~(C2) from \cref{set:rho2} and filter the resulting list by isomorphy. This yields the specifying data no. $19$ to $80$ in \cref{class:s=3-sample}.

\hfill\break
\noindent\emph{Case:} $\mu \in \partial\Mov(R)$.
The relation degree $\mu$ is contained in one of the rays~$\rho_1,\rho_2,\rho_3$. Reversing the order of the variables if necessary, we may assume that $\lambda = \rho_1 + \rho_2$ holds. There are the two cases $\mu \in \rho_1$ and $\mu \in \rho_2$. Here we exemplarily treat the case $\mu \in \rho_2$.

\hfill\break
\noindent\emph{Case:} $\mu \in \partial\Mov(R)$, $\mu \in \rho_2$.
The ray $\rho_2$ is a bounding ray of $\Mov(R)$. Thus in this configuration the cones $\lambda$ and $\Mov(R)$ coincide. By the definition of $\Mov(R)$ we must have~$n_3 = 1$ and $n_1 \ge 2$. Moreover, \cref{prop:hypmov} yields $n_2 \ge 2$. Applying \cref{lemma:3gen} to the generator degrees $w_1,w_2,w_7$ shows that the cone $\Eff(R)$ is regular and that $w_7$ is primitive. We may thus assume that $\Eff(R)$ is the positive quadrant and that $w_7 = (0,1)$ holds. Applying \cref{lemma:3gen} to the triples $(w_1,w_5,w_6)$ and~$(w_2,w_5,w_6)$ shows that the primitive generator $v$ of $\rho_2$ is of the form $v = (c,1)$ for some $c \ge 1$. We obtain
\begin{equation*}
    w_1 \ = \ w_2 \ = \ (1,0),
    \qquad
    w_5 \ = \ (a_5 c, a_5),
    \qquad
    w_6 \ = \ (a_6 c, a_6),
\end{equation*}
where $a_5, a_6, c \in \ZZ_{\ge 1}$ and $a_5,a_6$ are coprime. We may assume that $a_6 \ge a_5$ holds. There are three possible degree constellations $(n_1,n_2,n_3)$ for $X$, displayed in the following pictures.
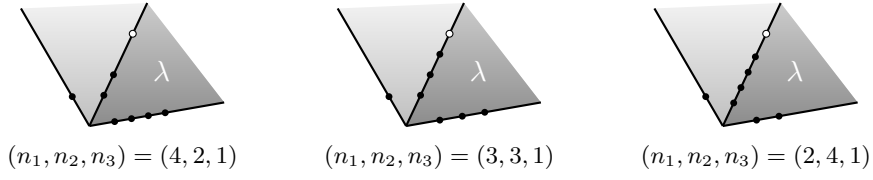
\begin{figure}[H]
    \subcaptionbox*{$(n_1,n_2,n_3) = (4,2,1)$}[0.32\linewidth]{\rhotwopic{4,2,1}{2}{1}}
    \subcaptionbox*{$(n_1,n_2,n_3) = (3,3,1)$}[0.32\linewidth]{\rhotwopic{3,3,1}{2}{1}}
    \subcaptionbox*{$(n_1,n_2,n_3) = (2,4,1)$}[0.32\linewidth]{\rhotwopic{2,4,1}{2}{1}}
\end{figure}
As an example we consider the degree constellation $(n_1,n_2,n_3) = (3,3,1)$.

\hfill\break
\noindent\emph{Case:} $\mu \in \partial\Mov(R)$, $\mu \in \rho_2$, $(n_1,n_2,n_3) = (3,3,1)$.
We apply \cref{lemma:3gen} to the triple~$(w_3,w_5,w_6)$ to obtain $w_3 = (1,0)$. By \cref{lemma:primgen} at least one of $w_4,w_5,w_6$ is primitive. We may assume that $w_4$ is primitive. Grading matrix and relation degree are given by
\begin{equation*}
    Q 
    \ = \
    \left[\begin{array}{rrrrrrr}
        1 & 1 & 1 & c & a_5 c & a_6 c & 0 \\
        0 & 0 & 0 & 1 &   a_5 &   a_6 & 1
    \end{array}\right],
    \qquad
    \mu \ = \ (k c, k),
\end{equation*}
for some $k \ge 2$. If $a_5 > 1$ holds, then by \cref{lemma:3gen}, applied to the tuple $(w_1,w_5)$, the relation $g$ has a monomial of the form $T_5^{l_5}$ with $l_5 \ge 2$. By homogeneity of $g$ we have $k = l_5a_5$. On the other hand, if $a_5 = 1$ holds, then clearly $k$ is a multiple of $a_5$. The same holds for $a_6$. Thus we have
\begin{equation*}
    k \ = \ l_5 a_5 \ = \ l_6 a_6
\end{equation*}
with $l_5 \ge l_6 \ge 2$. By \cref{prop:antican} the anticanonical class $-\KKK$ of $X$ is given by
\begin{equation*}
    -\KKK \ = \ 
    \left[\begin{array}{c}
        3 + (1+a_5+a_6-k)c \\
        2+a_5+a_6-k
    \end{array}\right].
\end{equation*}
From $X$ being Fano, ie. $-\KKK \in \lambda$, we infer the inequalities
\begin{eqnarray}
    \label{eq:s=3-sample-c=3-3-1-F1}
    k & \le & 1+a_5+a_6, \\
    \label{eq:s=3-sample-c=3-3-1-F2}
    c & \le & 2.
\end{eqnarray}
We distinguish three cases, depending on the values of $a_5$ and $a_6$.

\hfill\break
\noindent\emph{Case:} $\mu \in \partial\Mov(R)$, $\mu \in \rho_2$, $(n_1,n_2,n_3) = (3,3,1)$, $a_5 = a_6 = 1$.
Equation \ref{eq:s=3-sample-c=3-3-1-F1} yields the bound $k \le 3$. Grading matrix and relation degree are given by
\begin{equation*}
    Q 
    \ = \
    \left[\begin{array}{rrrrrrr}
        1 & 1 & 1 & c & c & c & 0 \\
        0 & 0 & 0 & 1 & 1 & 1 & 1
    \end{array}\right],
    \qquad
    \mu \ = \ (k c, k),
\end{equation*}
with $2 \le k \le 3$ and $1 \le c \le 2$. For the possible values of $k$ and $c$ we check each homogeneous trinomial $g$ of degree $\deg(g) = \mu$ for the conditions~(C1) and~(C2) from \cref{set:rho2} and filter by isomorphy. Depending on the values of $k$ and $c$ this leads to the following specifying data from \cref{class:s=3-sample}:
\begin{center}
\begin{tabular}{c|cccc}
    $(k,c)$ & $(2,1)$ & $(3,1)$ & $(2,2)$ & $(3,2)$ \\[2pt] \hline \\[-10pt]
    ID & 149 & 150-151 & 152-153 & 154-158
\end{tabular}
\end{center}

\hfill\break
\noindent\emph{Case:} $\mu \in \partial\Mov(R)$, $\mu \in \rho_2$, $(n_1,n_2,n_3) = (3,3,1)$, $a_5 = 1$, $a_6 > 1$.
Equation \ref{eq:s=3-sample-c=3-3-1-F1} yields $a_6 = 2$ and $l_6 = 2$. Grading matrix and relation degree are given by
\begin{equation*}
    Q 
    \ = \
    \left[\begin{array}{rrrrrrr}
        1 & 1 & 1 & c & c & 2c & 0 \\
        0 & 0 & 0 & 1 & 1 &  2 & 1
    \end{array}\right],
    \qquad
    \mu \ = \ (4c, 4).
\end{equation*}
For the two values of $c$ we check each homogeneous trinomial $g$ of degree $\deg(g) = \mu$ for the conditions~(C1) and~(C2) from \cref{set:rho2} and filter by isomorphy. For $c = 1$ this leads to specifying data no. 159 and 160. For $c = 2$ we get the specifying data no. 161 to no. 166.

\goodbreak

\hfill\break
\noindent\emph{Case:} $\mu \in \partial\Mov(R)$, $\mu \in \rho_2$, $(n_1,n_2,n_3) = (3,3,1)$, $a_5, a_6 > 1$.
We have $a_5 = l_6$, $a_6 = l_5$ and $k = a_5 a_6$. Thus we can apply \cref{lemma:fano-ufp} to obtain $a_5 = 2$, $a_6 = 3$ and $k = 6$. Grading matrix and relation degree are given by
\begin{equation*}
    Q 
    \ = \
    \left[\begin{array}{rrrrrrr}
        1 & 1 & 1 & c & 2c & 3c & 0 \\
        0 & 0 & 0 & 1 &  2 &  3 & 1
    \end{array}\right],
    \qquad
    \mu \ = \ (6c, 6).
\end{equation*}
For the two values of $c$ we check each homogeneous trinomial $g$ of degree $\deg(g) = \mu$ for the conditions~(C1) and~(C2) from \cref{set:rho2} and filter by isomorphy. For $c = 1$ this leads to specifying data no. 167 to 180. For $c = 2$ we get the specifying data no. 181 to no. 219.
\end{proof}


\begin{proposition}
\label{prop:s=3-nsample}
Let $X$ be as in \cref{set:rho2} with~$s = 3$ and assume that $\mu \not\in \lambda$ holds. Then $X$ is isomorphic to $X(Q,g)$ with specifying data $(Q,g)$ appearing in \cref{class:s=3-nsample} or in \cref{class:s=3-series}.
\end{proposition}

\begin{proof}
Reversing the order of the variables if necessary, we may assume that $\lambda = \rho_1 + \rho_2$ holds. By assumption $\mu$ is not contained in $\lambda$, so we have $\mu \in (\rho_2 + \rho_3)\backslash \rho_2$. By \cref{lemma:nonsemiample} we have $n_1 \ge 2$ and $w_1 = w_2$ is the primitive point in $\rho_1$. Moreover, by \cref{prop:hypmov} we have $n_3 \ge 3$. Applying \cref{lemma:3gen} to the triple $(w_1,w_6,w_7)$ shows that the cone $\Eff(R)$ is regular. We may thus assume that $\Eff(R)$ is the positive quadrant in $\QQ^2$ and that $w_1 = w_2 = (1,0)$ holds. We distinguish the two cases $\mu \in \rho_3$ and $\mu \in (\rho_2 + \rho_3)^\circ$. We exemplarily treat the case $\mu \in (\rho_2 + \rho_3)^\circ$.

\hfill\break
\noindent\emph{Case:} $\mu \in (\rho_2 + \rho_3)^\circ$.
By \cref{prop:hypmov} we have $n_3 \ge 3$.
Applying \cref{lemma:3gen} to the triples $(w_1,w_2,w_i)$ with $w_i \in \rho_2 \cup \rho_3$ and the triples $(w_j,w_6,w_7)$, where $w_j \in \rho_1$, shows that all generator degrees $w_1,\dots,w_7$ are primitive. In particular we have $w_5 = w_6 = w_7 = (0,1)$. The primitive point $v \in \rho_2$ is of the form~$v = (a,1)$ for some $a \ge 1$. There are three possible degree constellations $(n_1,n_2,n_3)$ for $X$, displayed in the following pictures.
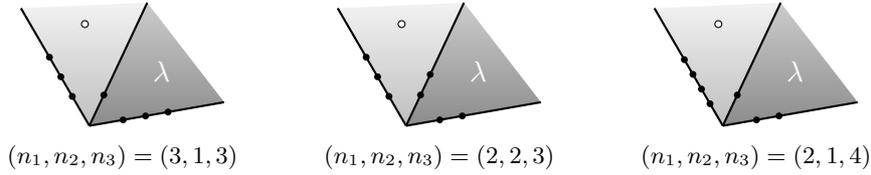
\begin{figure}[H]
    \subcaptionbox*{$(n_1,n_2,n_3) = (3,1,3)$}[0.32\linewidth]{\rhotwopic{3,1,3}{2.5}{1}}
    \subcaptionbox*{$(n_1,n_2,n_3) = (2,2,3)$}[0.32\linewidth]{\rhotwopic{2,2,3}{2.5}{1}}
    \subcaptionbox*{$(n_1,n_2,n_3) = (2,1,4)$}[0.32\linewidth]{\rhotwopic{2,1,4}{2.5}{1}}
\end{figure}
We consider the degree constellation $(n_1,n_2,n_3) = (3,1,3)$ as an example.

\hfill\break
\noindent\emph{Case:} $\mu \in (\rho_2 + \rho_3)^\circ$, $(n_1,n_2,n_3) = (3,1,3)$. Grading matrix and anticanonical class of~$X$ are given by
\begin{equation*}
    Q 
    \ = \
    \left[\begin{array}{rrrrrrr}
        1 & 1 & 1 & a & 0 & 0 & 0 \\
        0 & 0 & 0 & 1 & 1 & 1 & 1
    \end{array}\right],
    \qquad
    -\KKK
    \ = \ 
    \left[\begin{array}{c}
        3+a-\mu_1 \\
        4-\mu_2
    \end{array}\right].
\end{equation*}
From $X$ being Fano, ie. $-\KKK \in \lambda^\circ$, we infer the inequalities
\begin{eqnarray}
    \label{eq:s=3-c=3-1-3-F1}
    \mu_2 & \le & 3, \\
    \label{eq:s=3-c=3-1-3-F2}
    \mu_1 & \le & 2 - (3-\mu_2)a.
\end{eqnarray}
In particular, we have $2 \le \mu_2 \le 3$. We first consider the case $\mu_2 = 2$.
As $\mu_1$ is positive, Equation \ref{eq:s=3-c=3-1-3-F2} yields $a = 1$ and $\mu_1 = 1$. Grading matrix and relation degree are thus given by
\begin{equation*}
    Q 
    \ = \
    \left[\begin{array}{rrrrrrr}
        1 & 1 & 1 & 1 & 0 & 0 & 0 \\
        0 & 0 & 0 & 1 & 1 & 1 & 1
    \end{array}\right],
    \qquad
    \mu \ = \ (1,2).
\end{equation*}
Up to isomorphy this leads to specifying data no. 227 and 228.
Now consider the case $\mu_2 = 3$.
Then Equation \ref{eq:s=3-c=3-1-3-F2} yields $1 \le \mu_1 \le 2$.
We first consider the case $\mu = (1,3)$.
The relation $g$ is a trinomial with pairwise coprime monomials. Due to the position of $\lambda$, each monomial of $g$ is divisible by one of $T_1,\dots,T_4$. If $T_4$ divides a monomial of $g$, then by homogeneity we have $a = 1$. Up to isomorphy this yields specifying data no. $229$.
If $T_4$ does not appear in $g$, then each monomial of the relation is divisible by precisely one of $T_1,\dots,T_3$. Moreover, by the same argument, each monomial is divisible by precisely one of $T_5,\dots,T_7$. Thus up to permutation of variables, the relation $g$ is of the form
\begin{equation*}
    g \ = \ T_1 T_5^3 + T_2 T_6^3 + T_3 T_7^3.
\end{equation*}
Any choice for $a \ge 1$ yields valid specifying data. This is series ${\rm S}1$.
Finally we consider the case $\mu = (2,3)$.
Again, if $T_4$ divides a monomial of $g$, then homogeneity yields~$a \le 2$.
For the two possible values of $a$ we check each homogeneous trinomial $g$ of degree $\deg(g) = \mu$ for the conditions~(C1) and~(C2) from \cref{set:rho2} and filter by isomorphy. For $a = 1$ we get specifying data no. 230 to 232. For $a = 2$ we get specifying data no. 233 and 234.
If $T_4$ does not appear in $g$, then up to permutation of variables, the relation $g$ is of the form
\begin{equation*}
    g \ = \ T_1^2 T_5^3 + T_2^2 T_6^3 + T_3^2 T_7^3
\end{equation*}
Any choice for $a \ge 1$ yields valid specifying data. This is series ${\rm S}2$.
\end{proof}

\section{Classification lists}
\label{sec:classlists}

Here we provide the detailed presentation of our classification results. Let us briefly recall the background. Each locally factorial Fano fourfould $X$ of Picard number two and complexity one with a hypersurface Cox ring can be encoded by the degree matrix $Q$, that means the list $[w_1,\dots,w_7]$ of degrees of the Cox ring generators in $\Cl(X) = \ZZ^2$, and the defining trinomial $g$. Each such variety~$X$ is isomorphic to precisely one $X(Q,g)$ with specifying data $Q = [w_1,\dots,w_7]$ and $g$ appearing in the Classification lists \ref{class:s=2} to \ref{class:s=6-series}. Here $X(Q,g) = X_g$ is the variety from \cref{constr:hypersurface} associated with the $\ZZ^2$-graded $\CC$-algebra $R_g$, where the grading is given by $\deg(T_i) = w_i$.

To make the classification easier to navigate, we split it into several lists, each one containing the specifying data for a given number $s$ of rays generated by the degrees $[w_1,\dots,w_7]$, with either $\mu \in \lambda$ or $\mu \not\in \lambda$. Moreover, in the case $\mu \not\in \lambda$ we separate the sporadic cases from the infinite series of specifying data. Apart from the specifying data $(Q,g)$, the classification lists also contain the relation degree~$\mu = \deg(g)$, the anticanonical class $-\KKK \in \ZZ^2$ and, for the sporadic cases, also the anticanonical degree $\KKK^4$. A data file containing the complete classification data is also available at \cite{BaMa24}.

\medskip

\begin{class-list}\label{class:s=2} 
Locally factorial Fano fourfoulds of Picard number two with a hypersurface Cox ring and an effective three-torus action: Specifying data for the case with $s=2$.
\begin{center}
\small\setlength{\arraycolsep}{2pt}
\begin{longtable}{cccccl}
\toprule
\emph{ID} & $[w_1,\dots,w_7]$ & $\mu$ & $-\KKK$ & $\KKK^4$ & \multicolumn{1}{c}{$g$}
\\ \midrule

$1$
&
$\left[\tiny\begin{array}{rrrrrrr}
    1 & 1 & 1 & 1 & 0 & 0 & 0 \\
    0 & 0 & 0 & 0 & 1 & 1 & 1 
\end{array}\right]$
&
$(1,1)$
&
$(3,2)$
&
$432$
&
$T_{1}T_{6}+T_{2}T_{5}+T_{4}T_{7}$
\\ \midrule

$2$
&
\multirow{2}{*}{
$\left[\tiny\begin{array}{rrrrrrr}
    1 & 1 & 1 & 1 & 0 & 0 & 0 \\
    0 & 0 & 0 & 0 & 1 & 1 & 1 
\end{array}\right]$
}
&
\multirow{2}{*}{
$(2,1)$
}
&
\multirow{2}{*}{
$(2,2)$
}
&
\multirow{2}{*}{
$256$
}
&
$T_{1}^{2}T_{6}+T_{2}T_{3}T_{7}+T_{4}^{2}T_{5}$
\\

$3$
&
&
&
&
&
$T_{1}^{2}T_{6}+T_{3}^{2}T_{7}+T_{4}^{2}T_{5}$
\\ \midrule

$4$
&
\multirow{2}{*}{
$\left[\tiny\begin{array}{rrrrrrr}
    1 & 1 & 1 & 1 & 0 & 0 & 0 \\
    0 & 0 & 0 & 0 & 1 & 1 & 1 
\end{array}\right]$
}
&
\multirow{2}{*}{
$(3,1)$
}
&
\multirow{2}{*}{
$(1,2)$
}
&
\multirow{2}{*}{
$80$
}
&
$T_{1}^{3}T_{7}+T_{3}^{3}T_{5}+T_{4}^{3}T_{6}$
\\

$5$
&
&
&
&
&
$T_{1}^{3}T_{5}+T_{2}^{2}T_{4}T_{6}+T_{3}^{3}T_{7}$
\\ \bottomrule
\end{longtable}

\goodbreak

\begin{longtable}{cccccl}
\toprule
\emph{ID} & $[w_1,\dots,w_7]$ & $\mu$ & $-\KKK$ & $\KKK^4$ & \multicolumn{1}{c}{$g$}
\\ \midrule

$6$
&
$\left[\tiny\begin{array}{rrrrrrr}
    1 & 1 & 1 & 1 & 0 & 0 & 0 \\
    0 & 0 & 0 & 0 & 1 & 1 & 1 
\end{array}\right]$
&
$(1,2)$
&
$(3,1)$
&
$270$
&
$T_{1}T_{6}^{2}+T_{2}T_{7}^{2}+T_{3}T_{5}^{2}$
\\ \midrule

$7$
&
\multirow{2}{*}{
$\left[\tiny\begin{array}{rrrrrrr}
    1 & 1 & 1 & 1 & 0 & 0 & 0 \\
    0 & 0 & 0 & 0 & 1 & 1 & 1 
\end{array}\right]$
}
&
\multirow{2}{*}{
$(3,2)$
}
&
\multirow{2}{*}{
$(1,1)$
}
&
\multirow{2}{*}{
$26$
}
&
$T_{1}^{3}T_{7}^{2}+T_{2}^{3}T_{6}^{2}+T_{3}^{2}T_{4}T_{5}^{2}$
\\

$8$
&
&
&
&
&
$T_{1}^{3}T_{5}^{2}+T_{3}^{3}T_{6}^{2}+T_{4}^{3}T_{7}^{2}$
\\ \midrule

$9$
&
\multirow{2}{*}{
$\left[\tiny\begin{array}{rrrrrrr}
    1 & 1 & 1 & 1 & 3 & 0 & 0 \\
    0 & 0 & 0 & 0 & 0 & 1 & 1 
\end{array}\right]$
}
&
\multirow{2}{*}{
$(6,0)$
}
&
\multirow{2}{*}{
$(1,2)$
}
&
\multirow{2}{*}{
$16$
}
&
$T_{1}^{3}T_{3}^{3}+T_{2}^{5}T_{4}+T_{5}^{2}$
\\

$10$
&
&
&
&
&
$T_{1}T_{4}^{5}+T_{2}^{5}T_{3}+T_{5}^{2}$
\\ \midrule

$11$
&
\multirow{3}{*}{
$\left[\tiny\begin{array}{rrrrrrr}
    1 & 1 & 1 & 2 & 3 & 0 & 0 \\
    0 & 0 & 0 & 0 & 0 & 1 & 1 
\end{array}\right]$
}
&
\multirow{3}{*}{
$(6,0)$
}
&
\multirow{3}{*}{
$(2,2)$
}
&
\multirow{3}{*}{
$64$
}
&
$T_{1}^{4}T_{2}T_{3}+T_{4}^{3}+T_{5}^{2}$
\\

$12$
&
&
&
&
&
$T_{1}^{5}T_{3}+T_{4}^{3}+T_{5}^{2}$
\\

$13$
&
&
&
&
&
$T_{1}^{3}T_{2}^{2}T_{3}+T_{4}^{3}+T_{5}^{2}$
\\ \midrule

$14$
&
$\left[\tiny\begin{array}{rrrrrrr}
    1 & 1 & 1 & 1 & 2 & 0 & 0 \\
    0 & 0 & 0 & 0 & 0 & 1 & 1 
\end{array}\right]$
&
$(4,0)$
&
$(2,2)$
&
$128$
&
$T_{1}T_{3}^{3}+T_{2}^{3}T_{4}+T_{5}^{2}$
\\ \midrule

$15$
&
$\left[\tiny\begin{array}{rrrrrrr}
    1 & 1 & 1 & 1 & 1 & 0 & 0 \\
    0 & 0 & 0 & 0 & 0 & 1 & 1 
\end{array}\right]$
&
$(2,0)$
&
$(3,2)$
&
$432$
&
$T_{1}T_{3}+T_{2}T_{5}+T_{4}^{2}$
\\ \midrule

$16$
&
$\left[\tiny\begin{array}{rrrrrrr}
    1 & 1 & 1 & 1 & 1 & 0 & 0 \\
    0 & 0 & 0 & 0 & 0 & 1 & 1 
\end{array}\right]$
&
$(3,0)$
&
$(2,2)$
&
$192$
&
$T_{1}^{2}T_{4}+T_{2}^{2}T_{5}+T_{3}^{3}$
\\ \midrule

$17$
&
$\left[\tiny\begin{array}{rrrrrrr}
    1 & 1 & 1 & 1 & 1 & 0 & 0 \\
    0 & 0 & 0 & 0 & 0 & 1 & 1 
\end{array}\right]$
&
$(4,0)$
&
$(1,2)$
&
$32$
&
$T_{1}^{3}T_{2}+T_{3}^{3}T_{4}+T_{5}^{4}$
\\ \midrule

$18$
&
$\left[\tiny\begin{array}{rrrrrrr}
    1 & 1 & 2 & 3 & 0 & 0 & 0 \\
    0 & 0 & 0 & 0 & 1 & 1 & 1 
\end{array}\right]$
&
$(6,0)$
&
$(1,3)$
&
$54$
&
$T_{1}^{5}T_{2}+T_{3}^{3}+T_{4}^{2}$
\\ \bottomrule
\end{longtable}
\end{center}
\end{class-list}

\medskip

\begin{class-list}\label{class:s=3-sample}
Locally factorial Fano fourfoulds of Picard number two with a hypersurface Cox ring and an effective three-torus action: Specifying data for the cases with $s = 3$ and $\mu \in \lambda$.
\begin{center}
\small\setlength{\arraycolsep}{2pt}
\begin{longtable}{cccccl}
\toprule
\emph{ID} & $[w_1,\dots,w_7]$ & $\mu$ & $-\KKK$ & $\KKK^4$ & \multicolumn{1}{c}{$g$}
\\ \midrule

$19$
&
\multirow{2}{*}{
$\left[\tiny\begin{array}{rrrrrrr}
    1 & 1 & 1 & 1 & 0 & 0 & -1 \\
    0 & 0 & 0 & 0 & 1 & 1 & 1 
\end{array}\right]$
}
&
\multirow{2}{*}{
$(1,1)$
}
&
\multirow{2}{*}{
$(2,2)$
}
&
\multirow{2}{*}{
$416$
}
&
$T_{2}^{2}T_{7}+T_{1}T_{5}+T_{3}T_{6}$
\\

$20$
&
&
&
&
&
$T_{3}T_{4}T_{7}+T_{1}T_{6}+T_{2}T_{5}$
\\ \midrule

$21$
&
\multirow{2}{*}{
$\left[\tiny\begin{array}{rrrrrrr}
    1 & 1 & 1 & 1 & 0 & 0 & -1 \\
    0 & 0 & 0 & 0 & 1 & 1 & 1 
\end{array}\right]$
}
&
\multirow{2}{*}{
$(1,2)$
}
&
\multirow{2}{*}{
$(2,1)$
}
&
\multirow{2}{*}{
$163$
}
&
$T_{2}T_{3}^{2}T_{7}^{2}+T_{1}T_{5}^{2}+T_{4}T_{6}^{2}$
\\

$22$
&
&
&
&
&
$T_{3}^{3}T_{7}^{2}+T_{2}T_{5}^{2}+T_{4}T_{6}^{2}$
\\ \midrule

$26$
&
\multirow{2}{*}{
$\left[\tiny\begin{array}{rrrrrrr}
    1 & 1 & 1 & 1 & 0 & 0 & -2 \\
    0 & 0 & 0 & 0 & 1 & 1 & 1 
\end{array}\right]$
}
&
\multirow{2}{*}{
$(1,1)$
}
&
\multirow{2}{*}{
$(1,2)$
}
&
\multirow{2}{*}{
$464$
}
&
$T_{4}^{3}T_{7}+T_{2}T_{5}+T_{3}T_{6}$
\\

$27$
&
&
&
&
&
$T_{3}^{2}T_{4}T_{7}+T_{1}T_{6}+T_{2}T_{5}$
\\ \midrule

$31$
&
\multirow{2}{*}{
$\left[\tiny\begin{array}{rrrrrrr}
    1 & 1 & 1 & 1 & 1 & 0 & 0 \\
    0 & 0 & 0 & 0 & 1 & 1 & 1 
\end{array}\right]$
}
&
\multirow{2}{*}{
$(2,1)$
}
&
\multirow{2}{*}{
$(3,2)$
}
&
\multirow{2}{*}{
$352$
}
&
$T_{1}T_{4}T_{6}+T_{3}^{2}T_{7}+T_{2}T_{5}$
\\

$32$
&
&
&
&
&
$T_{1}^{2}T_{7}+T_{4}^{2}T_{6}+T_{3}T_{5}$
\\ \midrule

$33$
&
\multirow{2}{*}{
$\left[\tiny\begin{array}{rrrrrrr}
    1 & 1 & 1 & 1 & 1 & 0 & 0 \\
    0 & 0 & 0 & 0 & 1 & 1 & 1 
\end{array}\right]$
}
&
\multirow{2}{*}{
$(3,2)$
}
&
\multirow{2}{*}{
$(2,1)$
}
&
\multirow{2}{*}{
$81$
}
&
$T_{1}^{3}T_{7}^{2}+T_{2}^{2}T_{4}T_{6}^{2}+T_{3}T_{5}^{2}$
\\

$34$
&
&
&
&
&
$T_{3}^{3}T_{6}^{2}+T_{4}^{3}T_{7}^{2}+T_{1}T_{5}^{2}$
\\ \midrule

$44$
&
\multirow{2}{*}{
$\left[\tiny\begin{array}{rrrrrrr}
    1 & 1 & 1 & 1 & 0 & 0 & 0 \\
    -1 & 0 & 0 & 0 & 1 & 1 & 1 
\end{array}\right]$
}
&
\multirow{2}{*}{
$(1,1)$
}
&
\multirow{2}{*}{
$(3,1)$
}
&
\multirow{2}{*}{
$432$
}
&
$T_{1}T_{5}^{2}+T_{3}T_{6}+T_{4}T_{7}$
\\

$45$
&
&
&
&
&
$T_{2}T_{7}+T_{3}T_{5}+T_{4}T_{6}$
\\ \midrule

$46$
&
\multirow{2}{*}{
$\left[\tiny\begin{array}{rrrrrrr}
    1 & 1 & 1 & 1 & 0 & 0 & 0 \\
    0 & 0 & 1 & 1 & 1 & 1 & 1 
\end{array}\right]$
}
&
\multirow{2}{*}{
$(2,2)$
}
&
\multirow{2}{*}{
$(2,3)$
}
&
\multirow{2}{*}{
$272$
}
&
$T_{1}^{2}T_{5}T_{7}+T_{2}T_{4}T_{6}+T_{3}^{2}$
\\

$47$
&
&
&
&
&
$T_{1}^{2}T_{6}T_{7}+T_{2}^{2}T_{5}^{2}+T_{3}T_{4}$
\\ \midrule

$51$
&
\multirow{2}{*}{
$\left[\tiny\begin{array}{rrrrrrr}
    1 & 1 & 1 & 2 & 0 & 0 & 0 \\
    0 & 0 & 1 & 2 & 1 & 1 & 1 
\end{array}\right]$
}
&
\multirow{2}{*}{
$(4,4)$
}
&
\multirow{2}{*}{
$(1,2)$
}
&
\multirow{2}{*}{
$34$
}
&
$T_{2}^{4}T_{6}T_{7}^{3}+T_{1}T_{3}^{3}T_{5}+T_{4}^{2}$
\\

$52$
&
&
&
&
&
$T_{2}^{4}T_{5}^{3}T_{6}+T_{1}^{3}T_{3}T_{7}^{3}+T_{4}^{2}$
\\ \midrule

$72$
&
$\left[\tiny\begin{array}{rrrrrrr}
    1 & 1 & 1 & 0 & 0 & 0 & 0 \\
    0 & 0 & 1 & 1 & 1 & 1 & 1 
\end{array}\right]$
&
$(2,2)$
&
$(1,3)$
&
$216$
&
$T_{1}^{2}T_{4}T_{7}+T_{2}^{2}T_{5}T_{6}+T_{3}^{2}$
\\ \midrule

$73$
&
$\left[\tiny\begin{array}{rrrrrrr}
    1 & 1 & 1 & 0 & 0 & 0 & 0 \\
    0 & 0 & 2 & 1 & 1 & 1 & 1 
\end{array}\right]$
&
$(2,4)$
&
$(1,2)$
&
$64$
&
$T_{1}^{2}T_{4}T_{6}^{3}+T_{2}^{2}T_{5}T_{7}^{3}+T_{3}^{2}$
\\ \bottomrule
\end{longtable}

\goodbreak

\begin{longtable}{cccccl}
\toprule
\emph{ID} & $[w_1,\dots,w_7]$ & $\mu$ & $-\KKK$ & $\KKK^4$ & \multicolumn{1}{c}{$g$}
\\ \midrule

$77$
&
$\left[\tiny\begin{array}{rrrrrrr}
    1 & 1 & 1 & 1 & 0 & 0 & 0 \\
    0 & 0 & 0 & 1 & 1 & 1 & 1 
\end{array}\right]$
&
$(2,2)$
&
$(2,2)$
&
$192$
&
$T_{1}T_{3}T_{7}^{2}+T_{2}^{2}T_{5}T_{6}+T_{4}^{2}$
\\ \midrule

$78$
&
$\left[\tiny\begin{array}{rrrrrrr}
    1 & 1 & 1 & 1 & 0 & 0 & 0 \\
    0 & 0 & 0 & 1 & 1 & 1 & 1 
\end{array}\right]$
&
$(3,3)$
&
$(1,1)$
&
$18$
&
$T_{1}^{3}T_{5}T_{6}^{2}+T_{2}T_{3}^{2}T_{7}^{3}+T_{4}^{3}$
\\ \midrule

$79$
&
$\left[\tiny\begin{array}{rrrrrrr}
    1 & 1 & 1 & 2 & 0 & 0 & 0 \\
    0 & 0 & 0 & 1 & 1 & 1 & 1 
\end{array}\right]$
&
$(4,2)$
&
$(1,2)$
&
$48$
&
$T_{1}^{3}T_{2}T_{5}^{2}+T_{3}^{4}T_{6}T_{7}+T_{4}^{2}$
\\ \midrule

$80$
&
$\left[\tiny\begin{array}{rrrrrrr}
    1 & 1 & 1 & 2 & 0 & 0 & 0 \\
    0 & 0 & 0 & 2 & 1 & 1 & 1 
\end{array}\right]$
&
$(4,4)$
&
$(1,1)$
&
$12$
&
$T_{1}T_{2}^{3}T_{7}^{4}+T_{3}^{4}T_{5}^{3}T_{6}+T_{4}^{2}$
\\ \midrule

$149$
&
$\left[\tiny\begin{array}{rrrrrrr}
    1 & 1 & 1 & 1 & 0 & 0 & 0 \\
    0 & 1 & 1 & 1 & 1 & 1 & 1 
\end{array}\right]$
&
$(2,2)$
&
$(2,4)$
&
$352$
&
$T_{1}^{2}T_{5}T_{6}+T_{2}^{2}+T_{3}T_{4}$
\\ \midrule

$150$
&
\multirow{2}{*}{
$\left[\tiny\begin{array}{rrrrrrr}
    1 & 1 & 1 & 1 & 0 & 0 & 0 \\
    0 & 1 & 1 & 1 & 1 & 1 & 1 
\end{array}\right]$
}
&
\multirow{2}{*}{
$(3,3)$
}
&
\multirow{2}{*}{
$(1,3)$
}
&
\multirow{2}{*}{
$99$
}
&
$T_{1}^{3}T_{5}T_{6}^{2}+T_{2}^{3}+T_{3}T_{4}^{2}$
\\

$151$
&
&
&
&
&
$T_{1}^{3}T_{5}T_{6}T_{7}+T_{2}T_{4}^{2}+T_{3}^{3}$
\\ \midrule

$152$
&
\multirow{2}{*}{
$\left[\tiny\begin{array}{rrrrrrr}
    1 & 1 & 1 & 1 & 0 & 0 & 0 \\
    0 & 2 & 2 & 2 & 1 & 1 & 1 
\end{array}\right]$
}
&
\multirow{2}{*}{
$(2,4)$
}
&
\multirow{2}{*}{
$(2,5)$
}
&
\multirow{2}{*}{
$304$
}
&
$T_{1}^{2}T_{5}T_{6}^{3}+T_{2}T_{4}+T_{3}^{2}$
\\

$153$
&
&
&
&
&
$T_{1}^{2}T_{5}T_{6}^{2}T_{7}+T_{2}T_{3}+T_{4}^{2}$
\\ \midrule

$159$
&
\multirow{2}{*}{
$\left[\tiny\begin{array}{rrrrrrr}
    1 & 1 & 1 & 2 & 0 & 0 & 0 \\
    0 & 1 & 1 & 2 & 1 & 1 & 1 
\end{array}\right]$
}
&
\multirow{2}{*}{
$(4,4)$
}
&
\multirow{2}{*}{
$(1,3)$
}
&
\multirow{2}{*}{
$66$
}
&
$T_{1}^{4}T_{5}T_{7}^{3}+T_{2}T_{3}^{3}+T_{4}^{2}$
\\

$160$
&
&
&
&
&
$T_{1}^{4}T_{5}T_{6}T_{7}^{2}+T_{2}^{3}T_{3}+T_{4}^{2}$
\\ \midrule

$223$
&
$\left[\tiny\begin{array}{rrrrrrr}
    1 & 1 & 1 & 1 & 2 & 1 & 0 \\
    0 & 0 & 0 & 0 & 0 & 1 & 1 
\end{array}\right]$
&
$(4,0)$
&
$(3,2)$
&
$160$
&
$T_{1}T_{3}^{3}+T_{2}^{3}T_{4}+T_{5}^{2}$
\\ \midrule

$224$
&
$\left[\tiny\begin{array}{rrrrrrr}
    1 & 1 & 1 & 1 & 1 & 1 & 0 \\
    0 & 0 & 0 & 0 & 0 & 1 & 1 
\end{array}\right]$
&
$(3,0)$
&
$(3,2)$
&
$240$
&
$T_{1}^{2}T_{4}+T_{2}^{2}T_{5}+T_{3}^{3}$
\\ \midrule

$225$
&
$\left[\tiny\begin{array}{rrrrrrr}
    1 & 1 & 1 & 1 & 1 & 1 & 0 \\
    0 & 0 & 0 & 0 & 0 & 1 & 1 
\end{array}\right]$
&
$(2,0)$
&
$(4,2)$
&
$480$
&
$T_{1}T_{3}+T_{2}T_{5}+T_{4}^{2}$
\\ \midrule

$226$
&
$\left[\tiny\begin{array}{rrrrrrr}
    1 & 1 & 1 & 1 & 1 & 2 & 0 \\
    0 & 0 & 0 & 0 & 0 & 1 & 1 
\end{array}\right]$
&
$(2,0)$
&
$(5,2)$
&
$624$
&
$T_{1}T_{3}+T_{2}T_{5}+T_{4}^{2}$
\\ \bottomrule
\end{longtable}
\end{center}

\goodbreak

\begin{center}
\small\setlength{\arraycolsep}{2pt}
\begin{longtable}{llllll}
\toprule
\\[-8pt]

\multicolumn{6}{l}
{
\small\setlength{\tabcolsep}{8pt}
\begin{tabular}{cccc}
$
 Q \ = \ 
\left[\tiny\begin{array}{rrrrrrr}
    1 & 1 & 1 & 1 & 0 & 0 & -1 \\
    0 & 0 & 0 & 0 & 1 & 1 & 1 
\end{array}\right]
$
&
$\mu \ = \ (2,1)$
&
$-\KKK \ = \ (1,2)$
&
$\KKK^4 \ = \ 224$
\end{tabular}
}
\\[6pt]

{\small\emph{ID}}
&
\multicolumn{1}{c}{{\small$g$}}
&
{\small\emph{ID}}
&
\multicolumn{1}{c}{{\small$g$}}
&
{\small\emph{ID}}
&
\multicolumn{1}{c}{{\small$g$}}
\\[3pt]

$23$
&
${\scriptstyle T_{1}^{2}T_{3}T_{7}+T_{2}^{2}T_{6}+T_{4}^{2}T_{5}}$
&
$24$
&
${\scriptstyle T_{3}^{3}T_{7}+T_{2}^{2}T_{5}+T_{4}^{2}T_{6}}$
&
$25$
&
${\scriptstyle T_{4}^{3}T_{7}+T_{1}T_{3}T_{5}+T_{2}^{2}T_{6}}$
\\[2pt] \midrule \\[-8pt]

\multicolumn{6}{l}
{
\small\setlength{\tabcolsep}{8pt}
\begin{tabular}{cccc}
$
 Q \ = \ 
\left[\tiny\begin{array}{rrrrrrr}
    1 & 1 & 1 & 1 & 0 & 0 & -2 \\
    0 & 0 & 0 & 0 & 1 & 1 & 1 
\end{array}\right]
$
&
$\mu \ = \ (1,2)$
&
$-\KKK \ = \ (1,1)$
&
$\KKK^4 \ = \ 98$
\end{tabular}
}
\\[6pt]

{\small\emph{ID}}
&
\multicolumn{1}{c}{{\small$g$}}
&
{\small\emph{ID}}
&
\multicolumn{1}{c}{{\small$g$}}
&
{\small\emph{ID}}
&
\multicolumn{1}{c}{{\small$g$}}
\\[3pt]

$28$
&
${\scriptstyle T_{1}^{4}T_{3}T_{7}^{2}+T_{2}T_{5}^{2}+T_{4}T_{6}^{2}}$
&
$29$
&
${\scriptstyle T_{1}^{3}T_{2}^{2}T_{7}^{2}+T_{3}T_{5}^{2}+T_{4}T_{6}^{2}}$
&
$30$
&
${\scriptstyle T_{4}^{5}T_{7}^{2}+T_{2}T_{6}^{2}+T_{3}T_{5}^{2}}$
\\[2pt] \midrule \\[-8pt]

\multicolumn{6}{l}
{
\small\setlength{\tabcolsep}{8pt}
\begin{tabular}{cccc}
$
 Q \ = \ 
\left[\tiny\begin{array}{rrrrrrr}
    1 & 1 & 1 & 1 & 0 & 0 & 0 \\
    -1 & 0 & 0 & 0 & 1 & 1 & 1 
\end{array}\right]
$
&
$\mu \ = \ (3,1)$
&
$-\KKK \ = \ (1,1)$
&
$\KKK^4 \ = \ 38$
\end{tabular}
}
\\[6pt]

{\small\emph{ID}}
&
\multicolumn{1}{c}{{\small$g$}}
&
{\small\emph{ID}}
&
\multicolumn{1}{c}{{\small$g$}}
&
{\small\emph{ID}}
&
\multicolumn{1}{c}{{\small$g$}}
\\[3pt]

$35$
&
${\scriptstyle T_{1}^{3}T_{7}^{4}+T_{2}^{3}T_{6}+T_{3}T_{4}^{2}T_{5}}$
&
$36$
&
${\scriptstyle T_{2}^{3}T_{5}+T_{3}^{3}T_{7}+T_{4}^{3}T_{6}}$
&
$37$
&
${\scriptstyle T_{1}T_{3}^{2}T_{5}^{2}+T_{2}^{3}T_{6}+T_{4}^{3}T_{7}}$
\\

$38$
&
${\scriptstyle T_{1}^{3}T_{7}^{4}+T_{2}^{3}T_{5}+T_{4}^{3}T_{6}}$
&
$39$
&
${\scriptstyle T_{1}^{2}T_{3}T_{6}^{3}+T_{2}^{3}T_{7}+T_{4}^{3}T_{5}}$
&
&
\\[2pt] \midrule \\[-8pt]

\multicolumn{6}{l}
{
\small\setlength{\tabcolsep}{8pt}
\begin{tabular}{cccc}
$
 Q \ = \ 
\left[\tiny\begin{array}{rrrrrrr}
    1 & 1 & 1 & 1 & 0 & 0 & 0 \\
    -1 & 0 & 0 & 0 & 1 & 1 & 1 
\end{array}\right]
$
&
$\mu \ = \ (2,1)$
&
$-\KKK \ = \ (2,1)$
&
$\KKK^4 \ = \ 192$
\end{tabular}
}
\\[6pt]

{\small\emph{ID}}
&
\multicolumn{1}{c}{{\small$g$}}
&
{\small\emph{ID}}
&
\multicolumn{1}{c}{{\small$g$}}
&
{\small\emph{ID}}
&
\multicolumn{1}{c}{{\small$g$}}
\\[3pt]

$40$
&
${\scriptstyle T_{1}T_{4}T_{5}^{2}+T_{2}^{2}T_{7}+T_{3}^{2}T_{6}}$
&
$41$
&
${\scriptstyle T_{1}^{2}T_{7}^{3}+T_{2}T_{4}T_{6}+T_{3}^{2}T_{5}}$
&
$42$
&
${\scriptstyle T_{2}^{2}T_{7}+T_{3}^{2}T_{5}+T_{4}^{2}T_{6}}$
\\

$43$
&
${\scriptstyle T_{1}^{2}T_{6}^{3}+T_{3}^{2}T_{7}+T_{4}^{2}T_{5}}$
&
&

&
&
\\ \bottomrule
\end{longtable}

\goodbreak

\begin{longtable}{llllll}
\toprule
\\[-8pt]

\multicolumn{6}{l}
{
\small\setlength{\tabcolsep}{8pt}
\begin{tabular}{cccc}
$
 Q \ = \ 
\left[\tiny\begin{array}{rrrrrrr}
    1 & 1 & 1 & 1 & 0 & 0 & 0 \\
    0 & 0 & 1 & 1 & 1 & 1 & 1 
\end{array}\right]
$
&
$\mu \ = \ (3,3)$
&
$-\KKK \ = \ (1,2)$
&
$\KKK^4 \ = \ 51$
\end{tabular}
}
\\[6pt]

{\small\emph{ID}}
&
\multicolumn{1}{c}{{\small$g$}}
&
{\small\emph{ID}}
&
\multicolumn{1}{c}{{\small$g$}}
&
{\small\emph{ID}}
&
\multicolumn{1}{c}{{\small$g$}}
\\[3pt]

$48$
&
${\scriptstyle T_{2}^{3}T_{5}T_{6}^{2}+T_{1}T_{4}^{2}T_{7}+T_{3}^{3}}$
&
$49$
&
${\scriptstyle T_{2}^{3}T_{5}^{2}T_{7}+T_{1}^{2}T_{4}T_{6}^{2}+T_{3}^{3}}$
&
$50$
&
${\scriptstyle T_{1}^{3}T_{6}T_{7}^{2}+T_{2}^{3}T_{5}^{3}+T_{3}T_{4}^{2}}$
\\[2pt] \midrule \\[-8pt]

\multicolumn{6}{l}
{
\small\setlength{\tabcolsep}{8pt}
\begin{tabular}{cccc}
$
 Q \ = \ 
\left[\tiny\begin{array}{rrrrrrr}
    1 & 1 & 2 & 3 & 0 & 0 & 0 \\
    0 & 0 & 2 & 3 & 1 & 1 & 1 
\end{array}\right]
$
&
$\mu \ = \ (6,6)$
&
$-\KKK \ = \ (1,2)$
&
$\KKK^4 \ = \ 17$
\end{tabular}
}
\\[6pt]

{\small\emph{ID}}
&
\multicolumn{1}{c}{{\small$g$}}
&
{\small\emph{ID}}
&
\multicolumn{1}{c}{{\small$g$}}
&
{\small\emph{ID}}
&
\multicolumn{1}{c}{{\small$g$}}
\\[3pt]

$53$
&
${\scriptstyle T_{1}T_{2}^{5}T_{5}T_{6}T_{7}^{4}+T_{3}^{3}+T_{4}^{2}}$
&
$54$
&
${\scriptstyle T_{1}T_{2}^{5}T_{5}^{3}T_{6}T_{7}^{2}+T_{3}^{3}+T_{4}^{2}}$
&
$55$
&
${\scriptstyle T_{1}T_{2}^{5}T_{5}T_{7}^{5}+T_{3}^{3}+T_{4}^{2}}$
\\
$56$
&
${\scriptstyle T_{1}^{2}T_{2}^{4}T_{5}T_{6}^{5}+T_{3}^{3}+T_{4}^{2}}$
&
$57$
&
${\scriptstyle T_{1}^{6}T_{5}^{5}T_{6}+T_{3}^{3}+T_{4}^{2}}$
&
$58$
&
${\scriptstyle T_{1}^{2}T_{2}^{4}T_{5}^{3}T_{7}^{3}+T_{3}^{3}+T_{4}^{2}}$
\\
$59$
&
${\scriptstyle T_{1}^{6}T_{5}^{2}T_{6}T_{7}^{3}+T_{3}^{3}+T_{4}^{2}}$
&
$60$
&
${\scriptstyle T_{1}^{6}T_{5}^{4}T_{6}T_{7}+T_{3}^{3}+T_{4}^{2}}$
&
$61$
&
${\scriptstyle T_{1}^{3}T_{2}^{3}T_{5}^{2}T_{6}T_{7}^{3}+T_{3}^{3}+T_{4}^{2}}$
\\
$62$
&
${\scriptstyle T_{1}^{4}T_{2}^{2}T_{5}^{3}T_{6}T_{7}^{2}+T_{3}^{3}+T_{4}^{2}}$
&
$63$
&
${\scriptstyle T_{1}^{5}T_{2}T_{5}^{4}T_{6}^{2}+T_{3}^{3}+T_{4}^{2}}$
&
$64$
&
${\scriptstyle T_{1}^{5}T_{2}T_{5}^{2}T_{6}^{2}T_{7}^{2}+T_{3}^{3}+T_{4}^{2}}$
\\
$65$
&
${\scriptstyle T_{1}^{3}T_{2}^{3}T_{5}^{5}T_{6}+T_{3}^{3}+T_{4}^{2}}$
&
$66$
&
${\scriptstyle T_{1}^{3}T_{2}^{3}T_{5}^{2}T_{7}^{4}+T_{3}^{3}+T_{4}^{2}}$
&
$67$
&
${\scriptstyle T_{1}T_{2}^{5}T_{7}^{6}+T_{3}^{3}+T_{4}^{2}}$
\\
$68$
&
${\scriptstyle T_{1}^{2}T_{2}^{4}T_{5}^{4}T_{6}T_{7}+T_{3}^{3}+T_{4}^{2}}$
&
$69$
&
${\scriptstyle T_{1}T_{2}^{5}T_{5}^{3}T_{6}^{3}+T_{3}^{3}+T_{4}^{2}}$
&
$70$
&
${\scriptstyle T_{1}^{3}T_{2}^{3}T_{5}T_{6}T_{7}^{4}+T_{3}^{3}+T_{4}^{2}}$
\\

$71$
&
${\scriptstyle T_{1}^{3}T_{2}^{3}T_{5}^{2}T_{6}^{2}T_{7}^{2}+T_{3}^{3}+T_{4}^{2}}$
&
&

&
&
\\[2pt] \midrule \\[-8pt]

\multicolumn{6}{l}
{
\small\setlength{\tabcolsep}{8pt}
\begin{tabular}{cccc}
$
 Q \ = \ 
\left[\tiny\begin{array}{rrrrrrr}
    1 & 1 & 1 & 0 & 0 & 0 & 0 \\
    0 & 0 & 3 & 1 & 1 & 1 & 1 
\end{array}\right]
$
&
$\mu \ = \ (2,6)$
&
$-\KKK \ = \ (1,1)$
&
$\KKK^4 \ = \ 8$
\end{tabular}
}
\\[6pt]

{\small\emph{ID}}
&
\multicolumn{1}{c}{{\small$g$}}
&
{\small\emph{ID}}
&
\multicolumn{1}{c}{{\small$g$}}
&
{\small\emph{ID}}
&
\multicolumn{1}{c}{{\small$g$}}
\\[3pt]

$74$
&
${\scriptstyle T_{1}^{2}T_{5}^{5}T_{7}+T_{2}^{2}T_{4}T_{6}^{5}+T_{3}^{2}}$
&
$75$
&
${\scriptstyle T_{1}^{2}T_{5}^{5}T_{6}+T_{2}^{2}T_{4}^{3}T_{7}^{3}+T_{3}^{2}}$
&
$76$
&
${\scriptstyle T_{1}^{2}T_{6}^{3}T_{7}^{3}+T_{2}^{2}T_{4}^{3}T_{5}^{3}+T_{3}^{2}}$
\\[2pt] \midrule \\[-8pt]

\multicolumn{6}{l}
{
\small\setlength{\tabcolsep}{8pt}
\begin{tabular}{cccc}
$
 Q \ = \ 
\left[\tiny\begin{array}{rrrrrrr}
    1 & 1 & 1 & 1 & 1 & 0 & 0 \\
    0 & 1 & 1 & 1 & 1 & 1 & 1 
\end{array}\right]
$
&
$\mu \ = \ (2,2)$
&
$-\KKK \ = \ (3,4)$
&
$\KKK^4 \ = \ 378$
\end{tabular}
}
\\[6pt]

{\small\emph{ID}}
&
\multicolumn{1}{c}{{\small$g$}}
&
{\small\emph{ID}}
&
\multicolumn{1}{c}{{\small$g$}}
&
{\small\emph{ID}}
&
\multicolumn{1}{c}{{\small$g$}}
\\[3pt]

$81$
&
${\scriptstyle T_{1}T_{2}T_{7}+T_{3}^{2}+T_{4}T_{5}}$
&
$82$
&
${\scriptstyle T_{1}^{2}T_{6}^{2}+T_{2}T_{3}+T_{4}T_{5}}$
&
$83$
&
${\scriptstyle T_{1}^{2}T_{6}T_{7}+T_{3}T_{5}+T_{4}^{2}}$
\\

$84$
&
${\scriptstyle T_{1}^{2}T_{6}T_{7}+T_{2}T_{4}+T_{3}T_{5}}$
&
&

&
&
\\[2pt] \midrule \\[-8pt]

\multicolumn{6}{l}
{
\small\setlength{\tabcolsep}{8pt}
\begin{tabular}{cccc}
$
 Q \ = \ 
\left[\tiny\begin{array}{rrrrrrr}
    1 & 1 & 1 & 1 & 1 & 0 & 0 \\
    0 & 1 & 1 & 1 & 1 & 1 & 1 
\end{array}\right]
$
&
$\mu \ = \ (3,3)$
&
$-\KKK \ = \ (2,3)$
&
$\KKK^4 \ = \ 144$
\end{tabular}
}
\\[6pt]

{\small\emph{ID}}
&
\multicolumn{1}{c}{{\small$g$}}
&
{\small\emph{ID}}
&
\multicolumn{1}{c}{{\small$g$}}
&
{\small\emph{ID}}
&
\multicolumn{1}{c}{{\small$g$}}
\\[3pt]

$85$
&
${\scriptstyle T_{1}^{2}T_{5}T_{7}^{2}+T_{2}^{2}T_{4}+T_{3}^{3}}$
&
$86$
&
${\scriptstyle T_{1}^{3}T_{6}^{2}T_{7}+T_{2}T_{5}^{2}+T_{3}T_{4}^{2}}$
&
$87$
&
${\scriptstyle T_{1}^{3}T_{6}T_{7}^{2}+T_{3}^{3}+T_{4}T_{5}^{2}}$
\\
$88$
&
${\scriptstyle T_{1}T_{5}^{2}T_{7}+T_{2}^{3}+T_{3}T_{4}^{2}}$
&
$89$
&
${\scriptstyle T_{1}^{3}T_{6}^{3}+T_{2}^{2}T_{4}+T_{3}^{2}T_{5}}$
&
$90$
&
${\scriptstyle T_{1}^{2}T_{4}T_{6}T_{7}+T_{2}^{2}T_{3}+T_{5}^{3}}$
\\

$91$
&
${\scriptstyle T_{1}^{3}T_{6}^{2}T_{7}+T_{2}T_{3}T_{5}+T_{4}^{3}}$
&
&

&
&
\\[2pt] \midrule \\[-8pt]

\multicolumn{6}{l}
{
\small\setlength{\tabcolsep}{8pt}
\begin{tabular}{cccc}
$
 Q \ = \ 
\left[\tiny\begin{array}{rrrrrrr}
    1 & 1 & 1 & 1 & 1 & 0 & 0 \\
    0 & 1 & 1 & 1 & 1 & 1 & 1 
\end{array}\right]
$
&
$\mu \ = \ (4,4)$
&
$-\KKK \ = \ (1,2)$
&
$\KKK^4 \ = \ 20$
\end{tabular}
}
\\[6pt]

{\small\emph{ID}}
&
\multicolumn{1}{c}{{\small$g$}}
&
{\small\emph{ID}}
&
\multicolumn{1}{c}{{\small$g$}}
&
{\small\emph{ID}}
&
\multicolumn{1}{c}{{\small$g$}}
\\[3pt]

$92$
&
${\scriptstyle T_{1}^{3}T_{4}T_{6}^{3}+T_{2}T_{5}^{3}+T_{3}^{4}}$
&
$93$
&
${\scriptstyle T_{1}^{3}T_{2}T_{6}^{2}T_{7}+T_{3}^{4}+T_{4}T_{5}^{3}}$
&
$94$
&
${\scriptstyle T_{1}^{4}T_{6}T_{7}^{3}+T_{2}T_{5}^{3}+T_{3}^{3}T_{4}}$
\\
$95$
&
${\scriptstyle T_{1}^{4}T_{6}^{3}T_{7}+T_{2}T_{3}^{3}+T_{4}^{2}T_{5}^{2}}$
&
$96$
&
${\scriptstyle T_{1}^{2}T_{2}^{2}T_{6}T_{7}+T_{3}^{3}T_{5}+T_{4}^{4}}$
&
$97$
&
${\scriptstyle T_{1}T_{5}^{3}T_{7}+T_{2}^{3}T_{4}+T_{3}^{4}}$
\\
$98$
&
${\scriptstyle T_{1}^{4}T_{6}T_{7}^{3}+T_{2}T_{5}^{3}+T_{4}^{4}}$
&
$99$
&
${\scriptstyle T_{1}^{4}T_{6}^{2}T_{7}^{2}+T_{2}^{3}T_{4}+T_{3}T_{5}^{3}}$
&
$100$
&
${\scriptstyle T_{1}^{4}T_{6}^{3}T_{7}+T_{2}T_{3}^{2}T_{4}+T_{5}^{4}}$
\\

$101$
&
${\scriptstyle T_{1}^{4}T_{6}^{4}+T_{2}T_{3}^{3}+T_{4}T_{5}^{3}}$
&
&

&
&
\\[2pt] \midrule \\[-8pt]

\multicolumn{6}{l}
{
\small\setlength{\tabcolsep}{8pt}
\begin{tabular}{cccc}
$
 Q \ = \ 
\left[\tiny\begin{array}{rrrrrrr}
    1 & 1 & 1 & 1 & 2 & 0 & 0 \\
    0 & 1 & 1 & 1 & 2 & 1 & 1 
\end{array}\right]
$
&
$\mu \ = \ (4,4)$
&
$-\KKK \ = \ (2,3)$
&
$\KKK^4 \ = \ 96$
\end{tabular}
}
\\[6pt]

{\small\emph{ID}}
&
\multicolumn{1}{c}{{\small$g$}}
&
{\small\emph{ID}}
&
\multicolumn{1}{c}{{\small$g$}}
&
{\small\emph{ID}}
&
\multicolumn{1}{c}{{\small$g$}}
\\[3pt]

$102$
&
${\scriptstyle T_{1}^{3}T_{4}T_{6}T_{7}^{2}+T_{2}^{3}T_{3}+T_{5}^{2}}$
&
$103$
&
${\scriptstyle T_{1}T_{2}^{3}T_{7}+T_{3}^{3}T_{4}+T_{5}^{2}}$
&
$104$
&
${\scriptstyle T_{1}^{4}T_{6}^{3}T_{7}+T_{2}T_{3}^{2}T_{4}+T_{5}^{2}}$
\\
$105$
&
${\scriptstyle T_{1}^{4}T_{6}T_{7}^{3}+T_{3}^{3}T_{4}+T_{5}^{2}}$
&
$106$
&
${\scriptstyle T_{1}^{3}T_{4}T_{6}^{3}+T_{2}^{3}T_{3}+T_{5}^{2}}$
&
$107$
&
${\scriptstyle T_{1}^{2}T_{3}^{2}T_{6}T_{7}+T_{2}T_{4}^{3}+T_{5}^{2}}$
\\[2pt] \midrule \\[-8pt]

\multicolumn{6}{l}
{
\small\setlength{\tabcolsep}{8pt}
\begin{tabular}{cccc}
$
 Q \ = \ 
\left[\tiny\begin{array}{rrrrrrr}
    1 & 1 & 1 & 1 & 3 & 0 & 0 \\
    0 & 1 & 1 & 1 & 3 & 1 & 1 
\end{array}\right]
$
&
$\mu \ = \ (6,6)$
&
$-\KKK \ = \ (1,2)$
&
$\KKK^4 \ = \ 10$
\end{tabular}
}
\\[6pt]

{\small\emph{ID}}
&
\multicolumn{1}{c}{{\small$g$}}
&
{\small\emph{ID}}
&
\multicolumn{1}{c}{{\small$g$}}
&
{\small\emph{ID}}
&
\multicolumn{1}{c}{{\small$g$}}
\\[3pt]

$108$
&
${\scriptstyle T_{1}^{5}T_{2}T_{7}^{5}+T_{3}T_{4}^{5}+T_{5}^{2}}$
&
$109$
&
${\scriptstyle T_{1}T_{4}^{5}T_{7}+T_{2}^{5}T_{3}+T_{5}^{2}}$
&
$110$
&
${\scriptstyle T_{1}^{5}T_{3}T_{6}^{3}T_{7}^{2}+T_{2}T_{4}^{5}+T_{5}^{2}}$
\\
$111$
&
${\scriptstyle T_{1}^{6}T_{6}^{5}T_{7}+T_{2}^{3}T_{3}^{2}T_{4}+T_{5}^{2}}$
&
$112$
&
${\scriptstyle T_{1}^{5}T_{3}T_{6}^{5}+T_{2}^{3}T_{4}^{3}+T_{5}^{2}}$
&
$113$
&
${\scriptstyle T_{1}^{4}T_{2}^{2}T_{6}T_{7}^{3}+T_{3}^{3}T_{4}^{3}+T_{5}^{2}}$
\\
$114$
&
${\scriptstyle T_{1}^{3}T_{3}^{3}T_{6}^{2}T_{7}+T_{2}^{3}T_{4}^{3}+T_{5}^{2}}$
&
$115$
&
${\scriptstyle T_{1}^{6}T_{6}^{3}T_{7}^{3}+T_{2}T_{3}T_{4}^{4}+T_{5}^{2}}$
&
$116$
&
${\scriptstyle T_{1}^{5}T_{4}T_{6}^{4}T_{7}+T_{2}^{3}T_{3}^{3}+T_{5}^{2}}$
\\
$117$
&
${\scriptstyle T_{1}^{4}T_{3}^{2}T_{6}^{3}T_{7}+T_{2}T_{4}^{5}+T_{5}^{2}}$
&
$118$
&
${\scriptstyle T_{1}^{6}T_{6}^{5}T_{7}+T_{3}^{5}T_{4}+T_{5}^{2}}$
&
$119$
&
${\scriptstyle T_{1}^{6}T_{6}^{3}T_{7}^{3}+T_{2}T_{4}^{5}+T_{5}^{2}}$
\\ \bottomrule
\end{longtable}

\goodbreak

\begin{longtable}{llllll}
\toprule
\\[-8pt]

\multicolumn{6}{l}
{
\small\setlength{\tabcolsep}{8pt}
\begin{tabular}{cccc}
$
 Q \ = \ 
\left[\tiny\begin{array}{rrrrrrr}
    1 & 1 & 1 & 1 & 3 & 0 & 0 \\
    0 & 1 & 1 & 1 & 3 & 1 & 1 
\end{array}\right]
$
&
$\mu \ = \ (6,6)$
&
$-\KKK \ = \ (1,2)$
&
$\KKK^4 \ = \ 10$
\end{tabular}
}
\\[6pt]

{\small\emph{ID}}
&
\multicolumn{1}{c}{{\small$g$}}
&
{\small\emph{ID}}
&
\multicolumn{1}{c}{{\small$g$}}
&
{\small\emph{ID}}
&
\multicolumn{1}{c}{{\small$g$}}
\\[3pt]

$120$
&
${\scriptstyle T_{1}^{2}T_{2}^{4}T_{6}T_{7}+T_{3}^{5}T_{4}+T_{5}^{2}}$
&
$121$
&
${\scriptstyle T_{1}^{3}T_{2}^{3}T_{7}^{3}+T_{3}T_{4}^{5}+T_{5}^{2}}$
&
$122$
&
${\scriptstyle T_{1}^{3}T_{4}^{3}T_{6}T_{7}^{2}+T_{2}T_{3}^{5}+T_{5}^{2}}$
\\
$123$
&
${\scriptstyle T_{1}^{6}T_{6}^{5}T_{7}+T_{3}^{3}T_{4}^{3}+T_{5}^{2}}$
&
$124$
&
${\scriptstyle T_{1}^{5}T_{3}T_{6}T_{7}^{4}+T_{2}T_{4}^{5}+T_{5}^{2}}$
&
$125$
&
${\scriptstyle T_{1}^{6}T_{6}^{3}T_{7}^{3}+T_{2}^{2}T_{3}^{3}T_{4}+T_{5}^{2}}$
\\
$126$
&
${\scriptstyle T_{1}^{6}T_{6}T_{7}^{5}+T_{2}^{4}T_{3}T_{4}+T_{5}^{2}}$
&
$127$
&
${\scriptstyle T_{1}^{5}T_{4}T_{6}^{2}T_{7}^{3}+T_{2}^{3}T_{3}^{3}+T_{5}^{2}}$
&
$128$
&
${\scriptstyle T_{1}^{2}T_{4}^{4}T_{6}T_{7}+T_{2}^{3}T_{3}^{3}+T_{5}^{2}}$
\\

$129$
&
${\scriptstyle T_{1}T_{4}^{5}T_{7}+T_{2}^{3}T_{3}^{3}+T_{5}^{2}}$
&
&

&
&
\\[2pt] \midrule \\[-8pt]

\multicolumn{6}{l}
{
\small\setlength{\tabcolsep}{8pt}
\begin{tabular}{cccc}
$
 Q \ = \ 
\left[\tiny\begin{array}{rrrrrrr}
    1 & 1 & 1 & 2 & 3 & 0 & 0 \\
    0 & 1 & 1 & 2 & 3 & 1 & 1 
\end{array}\right]
$
&
$\mu \ = \ (6,6)$
&
$-\KKK \ = \ (2,3)$
&
$\KKK^4 \ = \ 48$
\end{tabular}
}
\\[6pt]

{\small\emph{ID}}
&
\multicolumn{1}{c}{{\small$g$}}
&
{\small\emph{ID}}
&
\multicolumn{1}{c}{{\small$g$}}
&
{\small\emph{ID}}
&
\multicolumn{1}{c}{{\small$g$}}
\\[3pt]

$130$
&
${\scriptstyle T_{1}^{4}T_{3}^{2}T_{6}T_{7}^{3}+T_{4}^{3}+T_{5}^{2}}$
&
$131$
&
${\scriptstyle T_{1}^{4}T_{2}T_{3}T_{6}T_{7}^{3}+T_{4}^{3}+T_{5}^{2}}$
&
$132$
&
${\scriptstyle T_{1}^{3}T_{2}^{2}T_{3}T_{6}^{3}+T_{4}^{3}+T_{5}^{2}}$
\\
$133$
&
${\scriptstyle T_{1}^{2}T_{2}T_{3}^{3}T_{6}T_{7}+T_{4}^{3}+T_{5}^{2}}$
&
$134$
&
${\scriptstyle T_{1}T_{2}^{4}T_{3}T_{6}+T_{4}^{3}+T_{5}^{2}}$
&
$135$
&
${\scriptstyle T_{1}^{4}T_{2}T_{3}T_{6}^{4}+T_{4}^{3}+T_{5}^{2}}$
\\
$136$
&
${\scriptstyle T_{1}T_{2}^{2}T_{3}^{3}T_{7}+T_{4}^{3}+T_{5}^{2}}$
&
$137$
&
${\scriptstyle T_{1}^{2}T_{2}^{4}T_{6}T_{7}+T_{4}^{3}+T_{5}^{2}}$
&
$138$
&
${\scriptstyle T_{2}^{5}T_{3}+T_{4}^{3}+T_{5}^{2}}$
\\
$139$
&
${\scriptstyle T_{1}^{2}T_{2}T_{3}^{3}T_{6}^{2}+T_{4}^{3}+T_{5}^{2}}$
&
$140$
&
${\scriptstyle T_{1}^{5}T_{3}T_{6}T_{7}^{4}+T_{4}^{3}+T_{5}^{2}}$
&
$141$
&
${\scriptstyle T_{1}^{3}T_{2}^{3}T_{6}T_{7}^{2}+T_{4}^{3}+T_{5}^{2}}$
\\
$142$
&
${\scriptstyle T_{1}^{5}T_{3}T_{6}^{5}+T_{4}^{3}+T_{5}^{2}}$
&
$143$
&
${\scriptstyle T_{1}^{6}T_{6}^{5}T_{7}+T_{4}^{3}+T_{5}^{2}}$
&
$144$
&
${\scriptstyle T_{1}^{4}T_{2}T_{3}T_{6}^{2}T_{7}^{2}+T_{4}^{3}+T_{5}^{2}}$
\\
$145$
&
${\scriptstyle T_{1}^{5}T_{2}T_{6}^{2}T_{7}^{3}+T_{4}^{3}+T_{5}^{2}}$
&
$146$
&
${\scriptstyle T_{1}T_{3}^{5}T_{6}+T_{4}^{3}+T_{5}^{2}}$
&
$147$
&
${\scriptstyle T_{1}^{3}T_{2}T_{3}^{2}T_{6}^{2}T_{7}+T_{4}^{3}+T_{5}^{2}}$
\\

$148$
&
${\scriptstyle T_{1}^{2}T_{2}^{2}T_{3}^{2}T_{6}T_{7}+T_{4}^{3}+T_{5}^{2}}$
&
&

&
&
\\[2pt] \midrule \\[-8pt]

\multicolumn{6}{l}
{
\small\setlength{\tabcolsep}{8pt}
\begin{tabular}{cccc}
$
 Q \ = \ 
\left[\tiny\begin{array}{rrrrrrr}
    1 & 1 & 1 & 1 & 0 & 0 & 0 \\
    0 & 2 & 2 & 2 & 1 & 1 & 1 
\end{array}\right]
$
&
$\mu \ = \ (3,6)$
&
$-\KKK \ = \ (1,3)$
&
$\KKK^4 \ = \ 54$
\end{tabular}
}
\\[6pt]

{\small\emph{ID}}
&
\multicolumn{1}{c}{{\small$g$}}
&
{\small\emph{ID}}
&
\multicolumn{1}{c}{{\small$g$}}
&
{\small\emph{ID}}
&
\multicolumn{1}{c}{{\small$g$}}
\\[3pt]

$154$
&
${\scriptstyle T_{1}^{3}T_{5}^{2}T_{6}T_{7}^{3}+T_{2}^{2}T_{4}+T_{3}^{3}}$
&
$155$
&
${\scriptstyle T_{1}^{3}T_{5}^{2}T_{7}^{4}+T_{2}^{3}+T_{3}^{2}T_{4}}$
&
$156$
&
${\scriptstyle T_{1}^{3}T_{5}^{4}T_{6}T_{7}+T_{2}^{3}+T_{3}^{2}T_{4}}$
\\

$157$
&
${\scriptstyle T_{1}^{3}T_{6}T_{7}^{5}+T_{2}^{2}T_{3}+T_{4}^{3}}$
&
$158$
&
${\scriptstyle T_{1}^{3}T_{5}^{2}T_{6}^{2}T_{7}^{2}+T_{2}T_{3}^{2}+T_{4}^{3}}$
&
&
\\[2pt] \midrule \\[-8pt]

\multicolumn{6}{l}
{
\small\setlength{\tabcolsep}{8pt}
\begin{tabular}{cccc}
$
 Q \ = \ 
\left[\tiny\begin{array}{rrrrrrr}
    1 & 1 & 1 & 2 & 0 & 0 & 0 \\
    0 & 2 & 2 & 4 & 1 & 1 & 1 
\end{array}\right]
$
&
$\mu \ = \ (4,8)$
&
$-\KKK \ = \ (1,3)$
&
$\KKK^4 \ = \ 36$
\end{tabular}
}
\\[6pt]

{\small\emph{ID}}
&
\multicolumn{1}{c}{{\small$g$}}
&
{\small\emph{ID}}
&
\multicolumn{1}{c}{{\small$g$}}
&
{\small\emph{ID}}
&
\multicolumn{1}{c}{{\small$g$}}
\\[3pt]

$161$
&
${\scriptstyle T_{1}^{4}T_{5}^{5}T_{6}T_{7}^{2}+T_{2}T_{3}^{3}+T_{4}^{2}}$
&
$162$
&
${\scriptstyle T_{1}^{4}T_{5}T_{6}^{4}T_{7}^{3}+T_{2}^{3}T_{3}+T_{4}^{2}}$
&
$163$
&
${\scriptstyle T_{1}^{4}T_{5}^{3}T_{6}^{5}+T_{2}T_{3}^{3}+T_{4}^{2}}$
\\
$164$
&
${\scriptstyle T_{1}^{4}T_{5}^{2}T_{6}^{3}T_{7}^{3}+T_{2}^{3}T_{3}+T_{4}^{2}}$
&
$165$
&
${\scriptstyle T_{1}^{4}T_{5}^{7}T_{7}+T_{2}T_{3}^{3}+T_{4}^{2}}$
&
$166$
&
${\scriptstyle T_{1}^{4}T_{5}^{6}T_{6}T_{7}+T_{2}T_{3}^{3}+T_{4}^{2}}$
\\[2pt] \midrule \\[-8pt]

\multicolumn{6}{l}
{
\small\setlength{\tabcolsep}{8pt}
\begin{tabular}{cccc}
$
 Q \ = \ 
\left[\tiny\begin{array}{rrrrrrr}
    1 & 1 & 2 & 3 & 0 & 0 & 0 \\
    0 & 1 & 2 & 3 & 1 & 1 & 1 
\end{array}\right]
$
&
$\mu \ = \ (6,6)$
&
$-\KKK \ = \ (1,3)$
&
$\KKK^4 \ = \ 33$
\end{tabular}
}
\\[6pt]

{\small\emph{ID}}
&
\multicolumn{1}{c}{{\small$g$}}
&
{\small\emph{ID}}
&
\multicolumn{1}{c}{{\small$g$}}
&
{\small\emph{ID}}
&
\multicolumn{1}{c}{{\small$g$}}
\\[3pt]

$167$
&
${\scriptstyle T_{1}^{6}T_{5}T_{6}^{2}T_{7}^{3}+T_{3}^{3}+T_{4}^{2}}$
&
$168$
&
${\scriptstyle T_{1}^{6}T_{6}^{5}T_{7}+T_{3}^{3}+T_{4}^{2}}$
&
$169$
&
${\scriptstyle T_{1}T_{2}^{5}T_{7}+T_{3}^{3}+T_{4}^{2}}$
\\
$170$
&
${\scriptstyle T_{1}^{3}T_{2}^{3}T_{6}^{2}T_{7}+T_{3}^{3}+T_{4}^{2}}$
&
$171$
&
${\scriptstyle T_{1}^{5}T_{2}T_{5}^{3}T_{6}T_{7}+T_{3}^{3}+T_{4}^{2}}$
&
$172$
&
${\scriptstyle T_{1}^{5}T_{2}T_{5}^{2}T_{6}^{3}+T_{3}^{3}+T_{4}^{2}}$
\\
$173$
&
${\scriptstyle T_{1}^{2}T_{2}^{4}T_{5}T_{7}+T_{3}^{3}+T_{4}^{2}}$
&
$174$
&
${\scriptstyle T_{1}^{6}T_{5}^{4}T_{6}T_{7}+T_{3}^{3}+T_{4}^{2}}$
&
$175$
&
${\scriptstyle T_{1}^{5}T_{2}T_{5}^{4}T_{6}+T_{3}^{3}+T_{4}^{2}}$
\\
$176$
&
${\scriptstyle T_{1}^{5}T_{2}T_{5}^{5}+T_{3}^{3}+T_{4}^{2}}$
&
$177$
&
${\scriptstyle T_{1}^{4}T_{2}^{2}T_{5}T_{6}^{3}+T_{3}^{3}+T_{4}^{2}}$
&
$178$
&
${\scriptstyle T_{1}^{5}T_{2}T_{5}T_{6}^{2}T_{7}^{2}+T_{3}^{3}+T_{4}^{2}}$
\\

$179$
&
${\scriptstyle T_{1}^{4}T_{2}^{2}T_{5}T_{6}T_{7}^{2}+T_{3}^{3}+T_{4}^{2}}$
&
$180$
&
${\scriptstyle T_{1}^{3}T_{2}^{3}T_{5}T_{6}T_{7}+T_{3}^{3}+T_{4}^{2}}$
&
&
\\[2pt] \midrule \\[-8pt]

\multicolumn{6}{l}
{
\small\setlength{\tabcolsep}{8pt}
\begin{tabular}{cccc}
$
 Q \ = \ 
\left[\tiny\begin{array}{rrrrrrr}
    1 & 1 & 2 & 3 & 0 & 0 & 0 \\
    0 & 2 & 4 & 6 & 1 & 1 & 1 
\end{array}\right]
$
&
$\mu \ = \ (6,12)$
&
$-\KKK \ = \ (1,3)$
&
$\KKK^4 \ = \ 18$
\end{tabular}
}
\\[6pt]

{\small\emph{ID}}
&
\multicolumn{1}{c}{{\small$g$}}
&
{\small\emph{ID}}
&
\multicolumn{1}{c}{{\small$g$}}
&
{\small\emph{ID}}
&
\multicolumn{1}{c}{{\small$g$}}
\\[3pt]

$181$
&
${\scriptstyle T_{1}T_{2}^{5}T_{5}^{2}+T_{3}^{3}+T_{4}^{2}}$
&
$182$
&
${\scriptstyle T_{1}^{6}T_{6}^{5}T_{7}^{7}+T_{3}^{3}+T_{4}^{2}}$
&
$183$
&
${\scriptstyle T_{1}^{6}T_{5}^{5}T_{6}^{2}T_{7}^{5}+T_{3}^{3}+T_{4}^{2}}$
\\
$184$
&
${\scriptstyle T_{1}^{5}T_{2}T_{5}^{3}T_{6}^{5}T_{7}^{2}+T_{3}^{3}+T_{4}^{2}}$
&
$185$
&
${\scriptstyle T_{1}^{5}T_{2}T_{5}^{9}T_{7}+T_{3}^{3}+T_{4}^{2}}$
&
$186$
&
${\scriptstyle T_{1}^{5}T_{2}T_{5}^{3}T_{6}^{7}+T_{3}^{3}+T_{4}^{2}}$
\\
$187$
&
${\scriptstyle T_{1}^{3}T_{2}^{3}T_{5}^{3}T_{6}T_{7}^{2}+T_{3}^{3}+T_{4}^{2}}$
&
$188$
&
${\scriptstyle T_{1}^{4}T_{2}^{2}T_{5}T_{6}T_{7}^{6}+T_{3}^{3}+T_{4}^{2}}$
&
$189$
&
${\scriptstyle T_{1}^{6}T_{5}^{4}T_{6}^{3}T_{7}^{5}+T_{3}^{3}+T_{4}^{2}}$
\\
$190$
&
${\scriptstyle T_{1}^{3}T_{2}^{3}T_{5}^{4}T_{6}T_{7}+T_{3}^{3}+T_{4}^{2}}$
&
$191$
&
${\scriptstyle T_{1}^{4}T_{2}^{2}T_{5}^{5}T_{6}^{3}+T_{3}^{3}+T_{4}^{2}}$
&
$192$
&
${\scriptstyle T_{1}^{5}T_{2}T_{5}^{6}T_{6}^{2}T_{7}^{2}+T_{3}^{3}+T_{4}^{2}}$
\\
$193$
&
${\scriptstyle T_{1}^{3}T_{2}^{3}T_{5}^{5}T_{7}+T_{3}^{3}+T_{4}^{2}}$
&
$194$
&
${\scriptstyle T_{1}^{6}T_{5}^{11}T_{7}+T_{3}^{3}+T_{4}^{2}}$
&
$195$
&
${\scriptstyle T_{1}^{6}T_{5}^{7}T_{6}^{3}T_{7}^{2}+T_{3}^{3}+T_{4}^{2}}$
\\
$196$
&
${\scriptstyle T_{1}^{6}T_{5}T_{6}^{10}T_{7}+T_{3}^{3}+T_{4}^{2}}$
&
$197$
&
${\scriptstyle T_{1}^{5}T_{2}T_{5}^{7}T_{6}T_{7}^{2}+T_{3}^{3}+T_{4}^{2}}$
&
$198$
&
${\scriptstyle T_{1}^{2}T_{2}^{4}T_{5}T_{6}^{3}+T_{3}^{3}+T_{4}^{2}}$
\\
$199$
&
${\scriptstyle T_{1}^{6}T_{5}T_{6}^{7}T_{7}^{4}+T_{3}^{3}+T_{4}^{2}}$
&
$200$
&
${\scriptstyle T_{1}^{6}T_{5}^{9}T_{6}T_{7}^{2}+T_{3}^{3}+T_{4}^{2}}$
&
$201$
&
${\scriptstyle T_{1}^{4}T_{2}^{2}T_{5}^{2}T_{6}^{3}T_{7}^{3}+T_{3}^{3}+T_{4}^{2}}$
\\
$202$
&
${\scriptstyle T_{1}^{4}T_{2}^{2}T_{5}^{5}T_{6}^{2}T_{7}+T_{3}^{3}+T_{4}^{2}}$
&
$203$
&
${\scriptstyle T_{1}^{6}T_{5}^{5}T_{6}T_{7}^{6}+T_{3}^{3}+T_{4}^{2}}$
&
$204$
&
${\scriptstyle T_{1}^{5}T_{2}T_{5}^{3}T_{6}^{4}T_{7}^{3}+T_{3}^{3}+T_{4}^{2}}$
\\
$205$
&
${\scriptstyle T_{1}^{5}T_{2}T_{5}^{8}T_{7}^{2}+T_{3}^{3}+T_{4}^{2}}$
&
$206$
&
${\scriptstyle T_{1}^{3}T_{2}^{3}T_{5}^{2}T_{7}^{4}+T_{3}^{3}+T_{4}^{2}}$
&
$207$
&
${\scriptstyle T_{1}^{5}T_{2}T_{5}^{4}T_{6}^{5}T_{7}+T_{3}^{3}+T_{4}^{2}}$
\\
$208$
&
${\scriptstyle T_{1}^{5}T_{2}T_{7}^{10}+T_{3}^{3}+T_{4}^{2}}$
&
$209$
&
${\scriptstyle T_{1}^{5}T_{2}T_{5}^{2}T_{6}^{4}T_{7}^{4}+T_{3}^{3}+T_{4}^{2}}$
&
$210$
&
${\scriptstyle T_{1}^{5}T_{2}T_{5}^{3}T_{6}^{6}T_{7}+T_{3}^{3}+T_{4}^{2}}$
\\
$211$
&
${\scriptstyle T_{1}T_{2}^{5}T_{6}T_{7}+T_{3}^{3}+T_{4}^{2}}$
&
$212$
&
${\scriptstyle T_{1}^{5}T_{2}T_{5}^{8}T_{6}T_{7}+T_{3}^{3}+T_{4}^{2}}$
&
$213$
&
${\scriptstyle T_{1}^{4}T_{2}^{2}T_{5}T_{6}^{3}T_{7}^{4}+T_{3}^{3}+T_{4}^{2}}$
\\
$214$
&
${\scriptstyle T_{1}^{2}T_{2}^{4}T_{5}T_{6}^{2}T_{7}+T_{3}^{3}+T_{4}^{2}}$
&
$215$
&
${\scriptstyle T_{1}^{4}T_{2}^{2}T_{5}T_{6}^{7}+T_{3}^{3}+T_{4}^{2}}$
&
$216$
&
${\scriptstyle T_{1}^{5}T_{2}T_{6}^{5}T_{7}^{5}+T_{3}^{3}+T_{4}^{2}}$
\\ \bottomrule
\end{longtable}

\goodbreak

\begin{longtable}{llllll}
\toprule
\\[-8pt]

\multicolumn{6}{l}
{
\small\setlength{\tabcolsep}{8pt}
\begin{tabular}{cccc}
$
 Q \ = \ 
\left[\tiny\begin{array}{rrrrrrr}
    1 & 1 & 2 & 3 & 0 & 0 & 0 \\
    0 & 2 & 4 & 6 & 1 & 1 & 1 
\end{array}\right]
$
&
$\mu \ = \ (6,12)$
&
$-\KKK \ = \ (1,3)$
&
$\KKK^4 \ = \ 18$
\end{tabular}
}
\\[6pt]

{\small\emph{ID}}
&
\multicolumn{1}{c}{{\small$g$}}
&
{\small\emph{ID}}
&
\multicolumn{1}{c}{{\small$g$}}
&
{\small\emph{ID}}
&
\multicolumn{1}{c}{{\small$g$}}
\\[3pt]

$217$
&
${\scriptstyle T_{1}^{5}T_{2}T_{5}^{6}T_{6}^{4}+T_{3}^{3}+T_{4}^{2}}$
&
$218$
&
${\scriptstyle T_{1}^{6}T_{5}T_{6}^{8}T_{7}^{3}+T_{3}^{3}+T_{4}^{2}}$
&
$219$
&
${\scriptstyle T_{1}^{3}T_{2}^{3}T_{5}^{2}T_{6}^{2}T_{7}^{2}+T_{3}^{3}+T_{4}^{2}}$
\\[2pt] \midrule \\[-8pt]

\multicolumn{6}{l}
{
\small\setlength{\tabcolsep}{8pt}
\begin{tabular}{cccc}
$
 Q \ = \ 
\left[\tiny\begin{array}{rrrrrrr}
    1 & 1 & 1 & 2 & 3 & 1 & 0 \\
    0 & 0 & 0 & 0 & 0 & 1 & 1 
\end{array}\right]
$
&
$\mu \ = \ (6,0)$
&
$-\KKK \ = \ (3,2)$
&
$\KKK^4 \ = \ 80$
\end{tabular}
}
\\[6pt]

{\small\emph{ID}}
&
\multicolumn{1}{c}{{\small$g$}}
&
{\small\emph{ID}}
&
\multicolumn{1}{c}{{\small$g$}}
&
{\small\emph{ID}}
&
\multicolumn{1}{c}{{\small$g$}}
\\[3pt]

$220$
&
${\scriptstyle T_{1}^{4}T_{2}T_{3}+T_{4}^{3}+T_{5}^{2}}$
&
$221$
&
${\scriptstyle T_{1}^{5}T_{3}+T_{4}^{3}+T_{5}^{2}}$
&
$222$
&
${\scriptstyle T_{1}^{3}T_{2}^{2}T_{3}+T_{4}^{3}+T_{5}^{2}}$
\\ \bottomrule
\end{longtable}
\end{center}

\end{class-list}

\medskip

\begin{class-list}\label{class:s=3-nsample}
Locally factorial Fano fourfoulds of Picard number two with a hypersurface Cox ring and an effective three-torus action: Specifying data for the sporadic cases with $s = 3$ and $\mu \not\in \lambda$.
\begin{center}
\small\setlength{\arraycolsep}{2pt}
\begin{longtable}{cccccl}
\toprule
\emph{ID} & $[w_1,\dots,w_7]$ & $\mu$ & $-\KKK$ & $\KKK^4$ & \multicolumn{1}{c}{$g$}
\\ \midrule

$227$
&
\multirow{2}{*}{
$\left[\tiny\begin{array}{rrrrrrr}
    1 & 1 & 1 & 1 & 0 & 0 & 0 \\
    0 & 0 & 0 & 1 & 1 & 1 & 1 
\end{array}\right]$
}
&
\multirow{2}{*}{
$(1,2)$
}
&
\multirow{2}{*}{
$(3,2)$
}
&
\multirow{2}{*}{
$352$
}
&
$T_{1}T_{5}^{2}+T_{2}T_{6}^{2}+T_{3}T_{7}^{2}$
\\

$228$
&
&
&
&
&
$T_{1}T_{5}^{2}+T_{2}T_{6}^{2}+T_{4}T_{7}$
\\ \midrule

$229$
&
$\left[\tiny\begin{array}{rrrrrrr}
    1 & 1 & 1 & 1 & 0 & 0 & 0 \\
    0 & 0 & 0 & 1 & 1 & 1 & 1 
\end{array}\right]$
&
$(1,3)$
&
$(3,1)$
&
$140$
&
$T_{1}T_{5}^{3}+T_{2}T_{6}^{3}+T_{4}T_{7}^{2}$
\\ \midrule

$230$
&
\multirow{3}{*}{
$\left[\tiny\begin{array}{rrrrrrr}
    1 & 1 & 1 & 1 & 0 & 0 & 0 \\
    0 & 0 & 0 & 1 & 1 & 1 & 1 
\end{array}\right]$
}
&
\multirow{3}{*}{
$(2,3)$
}
&
\multirow{3}{*}{
$(2,1)$
}
&
\multirow{3}{*}{
$65$
}
&
$T_{1}^{2}T_{5}^{3}+T_{2}^{2}T_{6}^{3}+T_{3}T_{4}T_{7}^{2}$
\\

$231$
&
&
&
&
&
$T_{1}^{2}T_{5}^{3}+T_{2}^{2}T_{6}^{3}+T_{4}^{2}T_{7}$
\\

$232$
&
&
&
&
&
$T_{1}^{2}T_{5}^{3}+T_{2}T_{3}T_{6}^{3}+T_{4}^{2}T_{7}^{2}$
\\ \midrule

$233$
&
\multirow{2}{*}{
$\left[\tiny\begin{array}{rrrrrrr}
    1 & 1 & 1 & 2 & 0 & 0 & 0 \\
    0 & 0 & 0 & 1 & 1 & 1 & 1 
\end{array}\right]$
}
&
\multirow{2}{*}{
$(2,3)$
}
&
\multirow{2}{*}{
$(3,1)$
}
&
\multirow{2}{*}{
$122$
}
&
$T_{1}^{2}T_{5}^{3}+T_{2}^{2}T_{6}^{3}+T_{4}T_{7}^{2}$
\\

$234$
&
&
&
&
&
$T_{1}^{2}T_{5}^{3}+T_{2}T_{3}T_{6}^{3}+T_{4}T_{7}^{2}$
\\ \midrule

$235$
&
\multirow{2}{*}{
$\left[\tiny\begin{array}{rrrrrrr}
    1 & 1 & 1 & 1 & 0 & 0 & 0 \\
    0 & 0 & 1 & 1 & 1 & 1 & 1 
\end{array}\right]$
}
&
\multirow{2}{*}{
$(1,3)$
}
&
\multirow{2}{*}{
$(3,2)$
}
&
\multirow{2}{*}{
$208$
}
&
$T_{1}T_{5}^{3}+T_{2}T_{6}^{3}+T_{4}T_{7}^{2}$
\\

$236$
&
&
&
&
&
$T_{1}T_{5}^{3}+T_{3}T_{6}^{2}+T_{4}T_{7}^{2}$
\\ \midrule

$237$
&
\multirow{4}{*}{
$\left[\tiny\begin{array}{rrrrrrr}
    1 & 1 & 1 & 0 & 0 & 0 & 0 \\
    0 & 0 & 1 & 1 & 1 & 1 & 1 
\end{array}\right]$
}
&
\multirow{4}{*}{
$(1,4)$
}
&
\multirow{4}{*}{
$(2,1)$
}
&
\multirow{4}{*}{
$29$
}
&
$T_{1}T_{4}^{4}+T_{2}T_{6}^{4}+T_{3}T_{7}^{3}$
\\

$238$
&
&
&
&
&
$T_{1}T_{4}^{3}T_{5}+T_{2}T_{6}^{4}+T_{3}T_{7}^{3}$
\\

$239$
&
&
&
&
&
$T_{1}T_{4}^{2}T_{5}^{2}+T_{2}T_{6}^{4}+T_{3}T_{7}^{3}$
\\

$240$
&
&
&
&
&
$T_{1}T_{4}^{4}+T_{2}T_{5}^{4}+T_{3}T_{6}T_{7}^{2}$
\\ \midrule

$241$
&
$\left[\tiny\begin{array}{rrrrrrr}
    1 & 1 & 1 & 0 & 0 & 0 & 0 \\
    0 & 0 & 1 & 1 & 1 & 2 & 3 
\end{array}\right]$
&
$(0,6)$
&
$(3,2)$
&
$80$
&
$T_{4}^{5}T_{5}+T_{6}^{3}+T_{7}^{2}$
\\ \bottomrule
\end{longtable}
\end{center}
\end{class-list}

\medskip

\begin{class-list}\label{class:s=3-series}
Locally factorial Fano fourfoulds of Picard number two with a hypersurface Cox ring and an effective three-torus action: Specifying data for the series with $s = 3$.
\begin{center}
\small\setlength{\arraycolsep}{2pt}
\small\setlength{\tabcolsep}{4pt}
\begin{longtable}{ccccll}
\toprule
\emph{ID} & $[w_1,\dots,w_7]$ & $\mu$ & $-\KKK$ & \multicolumn{1}{c}{$g$} & 
\\ \midrule

\emph{S1}
&
$
\left[\tiny\begin{array}{rrrrrrr}
    1 & 1 & 1 & a & 0 & 0 & 0 \\
    0 & 0 & 0 & 1 & 1 & 1 & 1
\end{array}\right]
$
&
$
(1,3)
$
&
$
(a\!+\!2,1)
$
&
$T_1 T_5^3 + T_2 T_6^3 + T_3 T_7^3$
&
${\scriptstyle a\, \ge\, 1}$
\\ \midrule

\emph{S2}
&
$
\left[\tiny\begin{array}{rrrrrrr}
    1 & 1 & 1 & a & 0 & 0 & 0 \\
    0 & 0 & 0 & 1 & 1 & 1 & 1
\end{array}\right]
$
&
$
(2,3)
$
&
$
(a\!+\!1,1)
$
&
$T_1^2 T_5^3 + T_2^2 T_6^3 + T_3^2 T_7^3$
&
${\scriptstyle a\, \ge\, 1}$
\\ \midrule

\emph{S3}
&
\multirow{2}{*}{
$
\left[\tiny\begin{array}{rrrrrrr}
    1 & 1 & a & a & 0 & 0 & 0 \\
    0 & 0 & 1 & 1 & 1 & 1 & 1
\end{array}\right]
$
}
&
\multirow{2}{*}{
$
(a,4)
$
}
&
\multirow{2}{*}{
$
(a\!+\!2,1)
$
}
&
$T_1^{a-l} T_2^l T_5^4 + T_3 T_6^3 + T_4 T_7^3$
&
$
{\scriptstyle a\, \ge\, 1, \ 0 \, \le \, l \, \le \, a/2}
$
\\

\emph{S4}
&
&
&
&
$T_1^a T_5^4 + T_2^a T_6^4 + T_4 T_7^3$
&
$
{\scriptstyle a\, \ge\, 1, \ a \text{\emph{ odd}}}
$
\\ \bottomrule
\end{longtable}
\end{center}
\end{class-list}

\medskip

\begin{class-list}\label{class:s=4-sample}
Locally factorial Fano fourfoulds of Picard number two with a hypersurface Cox ring and an effective three-torus action: Specifying data for the cases with $s = 4$ and $\mu \in \lambda$.
\begin{center}
\small\setlength{\arraycolsep}{2pt}
\begin{longtable}{cccccl}
\toprule
\emph{ID} & $[w_1,\dots,w_7]$ & $\mu$ & $-\KKK$ & $\KKK^4$ & \multicolumn{1}{c}{$g$}
\\ \midrule

$242$
&
\multirow{2}{*}{
$\left[\tiny\begin{array}{rrrrrrr}
    1 & 1 & 1 & 1 & 1 & 0 & 0 \\
    -1 & 0 & 0 & 0 & 1 & 1 & 1 
\end{array}\right]$
}
&
\multirow{2}{*}{
$(3,1)$
}
&
\multirow{2}{*}{
$(2,1)$
}
&
\multirow{2}{*}{
$113$
}
&
$T_{2}T_{4}^{2}T_{7}+T_{3}^{3}T_{6}+T_{1}T_{5}^{2}$
\\

$243$
&
&
&
&
&
$T_{3}^{3}T_{6}+T_{4}^{3}T_{7}+T_{1}T_{5}^{2}$
\\ \midrule

$244$
&
$\left[\tiny\begin{array}{rrrrrrr}
    1 & 1 & 1 & 1 & 1 & 0 & 0 \\
    0 & 1 & 1 & 1 & 2 & 1 & 1 
\end{array}\right]$
&
$(2,2)$
&
$(3,5)$
&
$433$
&
$T_{1}T_{5}+T_{2}T_{4}+T_{3}^{2}$
\\ \midrule

$245$
&
$\left[\tiny\begin{array}{rrrrrrr}
    1 & 1 & 1 & 1 & 1 & 0 & 0 \\
    0 & 2 & 2 & 2 & 3 & 1 & 1 
\end{array}\right]$
&
$(3,6)$
&
$(2,5)$
&
$145$
&
$T_{1}T_{5}^{2}+T_{2}^{2}T_{3}+T_{4}^{3}$
\\ \midrule

$246$
&
\multirow{2}{*}{
$\left[\tiny\begin{array}{rrrrrrr}
    1 & 1 & 1 & 1 & 0 & 0 & 0 \\
    0 & 1 & 1 & 2 & 1 & 1 & 1 
\end{array}\right]$
}
&
\multirow{2}{*}{
$(2,4)$
}
&
\multirow{2}{*}{
$(2,3)$
}
&
\multirow{2}{*}{
$144$
}
&
$T_{1}^{2}T_{6}T_{7}^{3}+T_{2}T_{3}T_{5}^{2}+T_{4}^{2}$
\\

$247$
&
&
&
&
&
$T_{1}T_{3}T_{6}^{3}+T_{2}^{2}T_{5}T_{7}+T_{4}^{2}$
\\ \midrule

$248$
&
\multirow{4}{*}{
$\left[\tiny\begin{array}{rrrrrrr}
    1 & 1 & 1 & 2 & 0 & 0 & 0 \\
    0 & 1 & 1 & 3 & 1 & 1 & 1 
\end{array}\right]$
}
&
\multirow{4}{*}{
$(4,6)$
}
&
\multirow{4}{*}{
$(1,2)$
}
&
\multirow{4}{*}{
$22$
}
&
$T_{1}^{4}T_{5}T_{7}^{5}+T_{2}T_{3}^{3}T_{6}^{2}+T_{4}^{2}$
\\

$249$
&
&
&
&
&
$T_{1}T_{3}^{3}T_{6}^{3}+T_{2}^{4}T_{5}T_{7}+T_{4}^{2}$
\\

$250$
&
&
&
&
&
$T_{1}^{3}T_{2}T_{7}^{5}+T_{3}^{4}T_{5}T_{6}+T_{4}^{2}$
\\

$251$
&
&
&
&
&
$T_{1}^{4}T_{6}^{3}T_{7}^{3}+T_{2}^{3}T_{3}T_{5}^{2}+T_{4}^{2}$
\\ \midrule

$252$
&
\multirow{2}{*}{
$\left[\tiny\begin{array}{rrrrrrr}
    1 & 1 & 1 & 1 & 1 & 1 & 0 \\
    -1 & 0 & 0 & 0 & 0 & 1 & 1 
\end{array}\right]$
}
&
\multirow{2}{*}{
$(2,0)$
}
&
\multirow{2}{*}{
$(4,1)$
}
&
\multirow{2}{*}{
$431$
}
&
$T_{1}T_{6}+T_{2}T_{4}+T_{3}^{2}$
\\

$253$
&
&
&
&
&
$T_{1}T_{6}+T_{2}T_{3}+T_{4}T_{5}$
\\ \midrule

$254$
&
$\left[\tiny\begin{array}{rrrrrrr}
    1 & 1 & 1 & 1 & 1 & 1 & 0 \\
    -1 & 0 & 0 & 0 & 0 & 1 & 1 
\end{array}\right]$
&
$(4,0)$
&
$(2,1)$
&
$62$
&
$T_{1}^{2}T_{6}^{2}+T_{2}^{3}T_{4}+T_{3}^{3}T_{5}$
\\ \midrule

$255$
&
\multirow{3}{*}{
$\left[\tiny\begin{array}{rrrrrrr}
    1 & 1 & 1 & 1 & 1 & 2 & 0 \\
    -1 & 0 & 0 & 0 & 0 & 1 & 1 
\end{array}\right]$
}
&
\multirow{3}{*}{
$(3,0)$
}
&
\multirow{3}{*}{
$(4,1)$
}
&
\multirow{3}{*}{
$376$
}
&
$T_{2}^{2}T_{3}+T_{4}^{3}+T_{1}T_{6}$
\\

$256$
&
&
&
&
&
$T_{2}T_{5}^{2}+T_{3}^{2}T_{4}+T_{1}T_{6}$
\\

$257$
&
&
&
&
&
$T_{2}T_{3}T_{5}+T_{4}^{3}+T_{1}T_{6}$
\\ \midrule

$258$
&
\multirow{4}{*}{
$\left[\tiny\begin{array}{rrrrrrr}
    1 & 1 & 1 & 1 & 1 & 3 & 0 \\
    -1 & 0 & 0 & 0 & 0 & 1 & 1 
\end{array}\right]$
}
&
\multirow{4}{*}{
$(4,0)$
}
&
\multirow{4}{*}{
$(4,1)$
}
&
\multirow{4}{*}{
$341$
}
&
$T_{2}^{3}T_{4}+T_{3}^{2}T_{5}^{2}+T_{1}T_{6}$
\\

$259$
&
&
&
&
&
$T_{2}^{3}T_{3}+T_{4}T_{5}^{3}+T_{1}T_{6}$
\\

$260$
&
&
&
&
&
$T_{2}^{4}+T_{3}T_{4}^{3}+T_{1}T_{6}$
\\

$261$
&
&
&
&
&
$T_{2}T_{3}^{2}T_{5}+T_{4}^{4}+T_{1}T_{6}$
\\ \midrule

$262$
&
\multirow{3}{*}{
$\left[\tiny\begin{array}{rrrrrrr}
    1 & 1 & 1 & 1 & 3 & 1 & 0 \\
    -1 & 0 & 0 & 0 & 0 & 1 & 1 
\end{array}\right]$
}
&
\multirow{3}{*}{
$(6,0)$
}
&
\multirow{3}{*}{
$(2,1)$
}
&
\multirow{3}{*}{
$31$
}
&
$T_{1}^{3}T_{6}^{3}+T_{2}T_{3}T_{4}^{4}+T_{5}^{2}$
\\

$263$
&
&
&
&
&
$T_{1}^{3}T_{6}^{3}+T_{3}^{5}T_{4}+T_{5}^{2}$
\\

$264$
&
&
&
&
&
$T_{1}^{3}T_{6}^{3}+T_{2}T_{3}^{3}T_{4}^{2}+T_{5}^{2}$
\\ \bottomrule
\end{longtable}
\end{center}

\end{class-list}

\medskip

\begin{class-list}\label{class:s=4-nsample}
Locally factorial Fano fourfoulds of Picard number two with a hypersurface Cox ring and an effective three-torus action: Specifying data for the sporadic cases with $s = 4$ and $\mu \not\in \lambda$.


\begin{center}
\small\setlength{\arraycolsep}{2pt}
\begin{longtable}{cccccl}
\toprule
\emph{ID} & $[w_1,\dots,w_7]$ & $\mu$ & $-\KKK$ & $\KKK^4$ & \multicolumn{1}{c}{$g$}
\\ \midrule

$279$
&
\multirow{2}{*}{
$\left[\tiny\begin{array}{rrrrrrr}
    1 & 1 & 1 & 4 & 2 & 3 & 0 \\
    0 & 0 & 0 & 1 & 2 & 3 & 1 
\end{array}\right]$
}
&
\multirow{2}{*}{
$(6,6)$
}
&
\multirow{2}{*}{
$(6,1)$
}
&
\multirow{2}{*}{
$117$
}
&
$T_{1}^{2}T_{4}T_{7}^{5}+T_{5}^{3}+T_{6}^{2}$
\\

$280$
&
&
&
&
&
$T_{1}T_{2}T_{4}T_{7}^{5}+T_{5}^{3}+T_{6}^{2}$
\\ \midrule

$281$
&
$\left[\tiny\begin{array}{rrrrrrr}
    1 & 1 & 1 & 5 & 2 & 3 & 0 \\
    0 & 0 & 0 & 1 & 2 & 3 & 1 
\end{array}\right]$
&
$(6,6)$
&
$(7,1)$
&
$157$
&
$T_{1}T_{4}T_{7}^{5}+T_{5}^{3}+T_{6}^{2}$
\\ \midrule

$282$
&
$\left[\tiny\begin{array}{rrrrrrr}
    1 & 1 & 1 & 6 & 2 & 3 & 0 \\
    0 & 0 & 0 & 1 & 2 & 3 & 1 
\end{array}\right]$
&
$(6,6)$
&
$(8,1)$
&
$203$
&
$T_{4}T_{7}^{5}+T_{5}^{3}+T_{6}^{2}$
\\ \midrule

$341$
&
\multirow{2}{*}{
$\left[\tiny\begin{array}{rrrrrrr}
    1 & 1 & 1 & 10 & 4 & 6 & 0 \\
    0 & 0 & 0 & 1 & 2 & 3 & 1 
\end{array}\right]$
}
&
\multirow{2}{*}{
$(12,6)$
}
&
\multirow{2}{*}{
$(11,1)$
}
&
\multirow{2}{*}{
$322$
}
&
$T_{1}^{2}T_{4}T_{7}^{5}+T_{5}^{3}+T_{6}^{2}$
\\

$342$
&
&
&
&
&
$T_{1}T_{2}T_{4}T_{7}^{5}+T_{5}^{3}+T_{6}^{2}$
\\ \midrule

$343$
&
$\left[\tiny\begin{array}{rrrrrrr}
    1 & 1 & 1 & 11 & 4 & 6 & 0 \\
    0 & 0 & 0 & 1 & 2 & 3 & 1 
\end{array}\right]$
&
$(12,6)$
&
$(12,1)$
&
$387$
&
$T_{1}T_{4}T_{7}^{5}+T_{5}^{3}+T_{6}^{2}$
\\ \bottomrule
\end{longtable}

\goodbreak

\begin{longtable}{cccccl}
\toprule
\emph{ID} & $[w_1,\dots,w_7]$ & $\mu$ & $-\KKK$ & $\KKK^4$ & \multicolumn{1}{c}{$g$}
\\ \midrule

$344$
&
$\left[\tiny\begin{array}{rrrrrrr}
    1 & 1 & 1 & 12 & 4 & 6 & 0 \\
    0 & 0 & 0 & 1 & 2 & 3 & 1 
\end{array}\right]$
&
$(12,6)$
&
$(13,1)$
&
$458$
&
$T_{4}T_{7}^{5}+T_{5}^{3}+T_{6}^{2}$
\\ \midrule

$345$
&
$\left[\tiny\begin{array}{rrrrrrr}
    1 & 1 & 1 & 2 & 2 & 0 & 0 \\
    0 & 0 & 0 & 1 & 2 & 1 & 1 
\end{array}\right]$
&
$(4,4)$
&
$(3,1)$
&
$68$
&
$T_{1}^{3}T_{2}T_{6}^{4}+T_{3}^{2}T_{4}T_{7}^{3}+T_{5}^{2}$
\\ \midrule

$346$
&
$\left[\tiny\begin{array}{rrrrrrr}
    1 & 1 & 1 & 3 & 2 & 0 & 0 \\
    0 & 0 & 0 & 1 & 2 & 1 & 1 
\end{array}\right]$
&
$(4,4)$
&
$(4,1)$
&
$114$
&
$T_{1}^{3}T_{2}T_{6}^{4}+T_{3}T_{4}T_{7}^{3}+T_{5}^{2}$
\\ \midrule

$347$
&
\multirow{2}{*}{
$\left[\tiny\begin{array}{rrrrrrr}
    1 & 1 & 1 & 4 & 2 & 0 & 0 \\
    0 & 0 & 0 & 1 & 2 & 1 & 1 
\end{array}\right]$
}
&
\multirow{2}{*}{
$(4,4)$
}
&
\multirow{2}{*}{
$(5,1)$
}
&
\multirow{2}{*}{
$172$
}
&
$T_{1}^{3}T_{2}T_{6}^{4}+T_{4}T_{7}^{3}+T_{5}^{2}$
\\

$348$
&
&
&
&
&
$T_{1}^{2}T_{2}T_{3}T_{6}^{4}+T_{4}T_{7}^{3}+T_{5}^{2}$
\\ \midrule

$349$
&
$\left[\tiny\begin{array}{rrrrrrr}
    1 & 1 & 1 & 2 & 1 & 0 & 0 \\
    0 & 0 & 0 & 1 & 1 & 1 & 1 
\end{array}\right]$
&
$(3,3)$
&
$(3,1)$
&
$102$
&
$T_{1}^{2}T_{2}T_{6}^{3}+T_{3}T_{4}T_{7}^{2}+T_{5}^{3}$
\\ \midrule

$350$
&
\multirow{2}{*}{
$\left[\tiny\begin{array}{rrrrrrr}
    1 & 1 & 1 & 3 & 1 & 0 & 0 \\
    0 & 0 & 0 & 1 & 1 & 1 & 1 
\end{array}\right]$
}
&
\multirow{2}{*}{
$(3,3)$
}
&
\multirow{2}{*}{
$(4,1)$
}
&
\multirow{2}{*}{
$171$
}
&
$T_{1}^{2}T_{2}T_{6}^{3}+T_{4}T_{7}^{2}+T_{5}^{3}$
\\

$351$
&
&
&
&
&
$T_{1}T_{2}T_{3}T_{6}^{3}+T_{4}T_{7}^{2}+T_{5}^{3}$
\\ \midrule

$352$
&
$\left[\tiny\begin{array}{rrrrrrr}
    1 & 1 & 2 & 2 & 1 & 1 & 0 \\
    0 & 0 & 1 & 1 & 1 & 1 & 1 
\end{array}\right]$
&
$(5,4)$
&
$(3,1)$
&
$29$
&
$T_{1}T_{6}^{4}+T_{2}T_{4}^{2}T_{7}^{2}+T_{3}T_{5}^{3}$
\\ \midrule

$353$
&
$\left[\tiny\begin{array}{rrrrrrr}
    1 & 1 & 3 & 3 & 1 & 1 & 0 \\
    0 & 0 & 1 & 1 & 1 & 1 & 1 
\end{array}\right]$
&
$(6,4)$
&
$(4,1)$
&
$38$
&
$T_{1}T_{2}T_{6}^{4}+T_{3}T_{5}^{3}+T_{4}^{2}T_{7}^{2}$
\\ \midrule

$364$
&
\multirow{2}{*}{
$\left[\tiny\begin{array}{rrrrrrr}
    1 & 1 & 4 & 4 & 2 & 3 & 0 \\
    0 & 0 & 1 & 1 & 2 & 3 & 1 
\end{array}\right]$
}
&
\multirow{2}{*}{
$(6,6)$
}
&
\multirow{2}{*}{
$(9,2)$
}
&
\multirow{2}{*}{
$144$
}
&
$T_{1}^{2}T_{3}T_{7}^{5}+T_{5}^{3}+T_{6}^{2}$
\\

$365$
&
&
&
&
&
$T_{1}T_{2}T_{3}T_{7}^{5}+T_{5}^{3}+T_{6}^{2}$
\\ \midrule

$366$
&
$\left[\tiny\begin{array}{rrrrrrr}
    1 & 1 & 5 & 5 & 2 & 3 & 0 \\
    0 & 0 & 1 & 1 & 2 & 3 & 1 
\end{array}\right]$
&
$(6,6)$
&
$(11,2)$
&
$176$
&
$T_{1}T_{3}T_{7}^{5}+T_{5}^{3}+T_{6}^{2}$
\\ \midrule

$367$
&
$\left[\tiny\begin{array}{rrrrrrr}
    1 & 1 & 6 & 6 & 2 & 3 & 0 \\
    0 & 0 & 1 & 1 & 2 & 3 & 1 
\end{array}\right]$
&
$(6,6)$
&
$(13,2)$
&
$208$
&
$T_{3}T_{7}^{5}+T_{5}^{3}+T_{6}^{2}$
\\ \midrule

$368$
&
$\left[\tiny\begin{array}{rrrrrrr}
    1 & 1 & 2 & 2 & 1 & 0 & 0 \\
    0 & 0 & 1 & 1 & 1 & 1 & 1 
\end{array}\right]$
&
$(4,4)$
&
$(3,1)$
&
$32$
&
$T_{1}^{3}T_{2}T_{6}^{4}+T_{3}T_{4}T_{7}^{2}+T_{5}^{4}$
\\ \midrule

$369$
&
$\left[\tiny\begin{array}{rrrrrrr}
    1 & 1 & 1 & 1 & 1 & 0 & 0 \\
    0 & 0 & 1 & 1 & 2 & 1 & 1 
\end{array}\right]$
&
$(2,4)$
&
$(3,2)$
&
$128$
&
$T_{1}T_{3}T_{6}^{3}+T_{2}T_{4}T_{7}^{3}+T_{5}^{2}$
\\ \midrule

$370$
&
$\left[\tiny\begin{array}{rrrrrrr}
    1 & 1 & 2 & 2 & 1 & 0 & 0 \\
    0 & 0 & 1 & 1 & 2 & 1 & 1 
\end{array}\right]$
&
$(2,4)$
&
$(5,2)$
&
$192$
&
$T_{3}T_{6}^{3}+T_{4}T_{7}^{3}+T_{5}^{2}$
\\ \midrule

$378$
&
\multirow{2}{*}{
$\left[\tiny\begin{array}{rrrrrrr}
    1 & 1 & 2 & 1 & 1 & 1 & 0 \\
    0 & 0 & 1 & 1 & 1 & 1 & 1 
\end{array}\right]$
}
&
\multirow{2}{*}{
$(4,4)$
}
&
\multirow{2}{*}{
$(3,1)$
}
&
\multirow{2}{*}{
$32$
}
&
$T_{1}^{2}T_{3}T_{7}^{3}+T_{4}^{4}+T_{5}^{3}T_{6}$
\\

$379$
&
&
&
&
&
$T_{1}T_{2}T_{3}T_{7}^{3}+T_{4}^{4}+T_{5}^{3}T_{6}$
\\ \midrule

$380$
&
$\left[\tiny\begin{array}{rrrrrrr}
    1 & 1 & 3 & 1 & 1 & 1 & 0 \\
    0 & 0 & 1 & 1 & 1 & 1 & 1 
\end{array}\right]$
&
$(4,4)$
&
$(4,1)$
&
$44$
&
$T_{1}T_{3}T_{7}^{3}+T_{4}^{4}+T_{5}^{3}T_{6}$
\\ \midrule

$381$
&
$\left[\tiny\begin{array}{rrrrrrr}
    1 & 1 & 4 & 1 & 1 & 1 & 0 \\
    0 & 0 & 1 & 1 & 1 & 1 & 1 
\end{array}\right]$
&
$(4,4)$
&
$(5,1)$
&
$56$
&
$T_{3}T_{7}^{3}+T_{4}^{4}+T_{5}^{3}T_{6}$
\\ \midrule

$399$
&
\multirow{2}{*}{
$\left[\tiny\begin{array}{rrrrrrr}
    1 & 1 & 5 & 1 & 1 & 3 & 0 \\
    0 & 0 & 1 & 1 & 1 & 3 & 1 
\end{array}\right]$
}
&
\multirow{2}{*}{
$(6,6)$
}
&
\multirow{2}{*}{
$(6,1)$
}
&
\multirow{2}{*}{
$34$
}
&
$T_{1}T_{3}T_{7}^{5}+T_{4}^{5}T_{5}+T_{6}^{2}$
\\

$400$
&
&
&
&
&
$T_{1}T_{3}T_{7}^{5}+T_{4}^{3}T_{5}^{3}+T_{6}^{2}$
\\ \midrule

$401$
&
\multirow{2}{*}{
$\left[\tiny\begin{array}{rrrrrrr}
    1 & 1 & 6 & 1 & 1 & 3 & 0 \\
    0 & 0 & 1 & 1 & 1 & 3 & 1 
\end{array}\right]$
}
&
\multirow{2}{*}{
$(6,6)$
}
&
\multirow{2}{*}{
$(7,1)$
}
&
\multirow{2}{*}{
$40$
}
&
$T_{3}T_{7}^{5}+T_{4}^{5}T_{5}+T_{6}^{2}$
\\

$402$
&
&
&
&
&
$T_{3}T_{7}^{5}+T_{4}^{3}T_{5}^{3}+T_{6}^{2}$
\\ \midrule

$403$
&
\multirow{2}{*}{
$\left[\tiny\begin{array}{rrrrrrr}
    1 & 1 & 1 & 1 & 0 & 0 & 0 \\
    0 & 0 & 1 & 3 & 1 & 1 & 1 
\end{array}\right]$
}
&
\multirow{2}{*}{
$(2,6)$
}
&
\multirow{2}{*}{
$(2,1)$
}
&
\multirow{2}{*}{
$14$
}
&
$T_{2}^{2}T_{5}^{5}T_{6}+T_{1}T_{3}T_{7}^{5}+T_{4}^{2}$
\\

$404$
&
&
&
&
&
$T_{2}^{2}T_{5}^{3}T_{6}^{3}+T_{1}T_{3}T_{7}^{5}+T_{4}^{2}$
\\ \bottomrule
\end{longtable}
\end{center}


\begin{center}
\small\setlength{\arraycolsep}{2pt}
\begin{longtable}{llllll}
\toprule
\\[-8pt]

\multicolumn{6}{l}
{
\small\setlength{\tabcolsep}{8pt}
\begin{tabular}{cccc}
$
 Q \ = \ 
\left[\tiny\begin{array}{rrrrrrr}
    1 & 1 & 1 & 1 & 0 & 0 & -1 \\
    0 & 0 & 0 & 1 & 1 & 1 & 1 
\end{array}\right]
$
&
$\mu \ = \ (1,3)$
&
$-\KKK \ = \ (2,1)$
&
$\KKK^4 \ = \ 83$
\end{tabular}
}
\\[6pt]

{\small\emph{ID}}
&
\multicolumn{1}{c}{{\small$g$}}
&
{\small\emph{ID}}
&
\multicolumn{1}{c}{{\small$g$}}
&
{\small\emph{ID}}
&
\multicolumn{1}{c}{{\small$g$}}
\\[3pt]

$265$
&
${\scriptstyle T_{3}^{2}T_{4}T_{7}^{2}+T_{1}T_{5}^{3}+T_{2}T_{6}^{3}}$
&
$266$
&
${\scriptstyle T_{2}^{4}T_{7}^{3}+T_{1}T_{5}^{3}+T_{4}T_{6}^{2}}$
&
$267$
&
${\scriptstyle T_{2}^{4}T_{7}^{3}+T_{1}T_{5}^{3}+T_{3}T_{6}^{3}}$
\\
$268$
&
${\scriptstyle T_{2}^{3}T_{3}T_{7}^{3}+T_{1}T_{6}^{3}+T_{4}T_{5}^{2}}$
&
$269$
&
${\scriptstyle T_{1}^{2}T_{2}^{2}T_{7}^{3}+T_{3}T_{6}^{3}+T_{4}T_{5}^{2}}$
&
$270$
&
${\scriptstyle T_{1}T_{5}^{3}+T_{3}T_{6}^{3}+T_{4}^{2}T_{7}}$
\\ \bottomrule
\end{longtable}

\goodbreak

\begin{longtable}{llllll}
\toprule
\\[-8pt]

\multicolumn{6}{l}
{
\small\setlength{\tabcolsep}{8pt}
\begin{tabular}{cccc}
$
 Q \ = \ 
\left[\tiny\begin{array}{rrrrrrr}
    1 & 1 & 1 & 2 & 2 & 3 & 0 \\
    0 & 0 & 0 & 1 & 2 & 3 & 1 
\end{array}\right]
$
&
$\mu \ = \ (6,6)$
&
$-\KKK \ = \ (4,1)$
&
$\KKK^4 \ = \ 55$
\end{tabular}
}
\\[6pt]

{\small\emph{ID}}
&
\multicolumn{1}{c}{{\small$g$}}
&
{\small\emph{ID}}
&
\multicolumn{1}{c}{{\small$g$}}
&
{\small\emph{ID}}
&
\multicolumn{1}{c}{{\small$g$}}
\\[3pt]

$271$
&
${\scriptstyle T_{1}T_{2}T_{4}^{2}T_{7}^{4}+T_{5}^{3}+T_{6}^{2}}$
&
$272$
&
${\scriptstyle T_{1}^{4}T_{4}T_{7}^{5}+T_{5}^{3}+T_{6}^{2}}$
&
$273$
&
${\scriptstyle T_{1}^{3}T_{2}T_{4}T_{7}^{5}+T_{5}^{3}+T_{6}^{2}}$
\\

$274$
&
${\scriptstyle T_{1}^{2}T_{2}^{2}T_{4}T_{7}^{5}+T_{5}^{3}+T_{6}^{2}}$
&
$275$
&
${\scriptstyle T_{1}^{2}T_{2}T_{3}T_{4}T_{7}^{5}+T_{5}^{3}+T_{6}^{2}}$
&
&
\\[2pt] \midrule \\[-8pt]

\multicolumn{6}{l}
{
\small\setlength{\tabcolsep}{8pt}
\begin{tabular}{cccc}
$
 Q \ = \ 
\left[\tiny\begin{array}{rrrrrrr}
    1 & 1 & 1 & 3 & 2 & 3 & 0 \\
    0 & 0 & 0 & 1 & 2 & 3 & 1 
\end{array}\right]
$
&
$\mu \ = \ (6,6)$
&
$-\KKK \ = \ (5,1)$
&
$\KKK^4 \ = \ 83$
\end{tabular}
}
\\[6pt]

{\small\emph{ID}}
&
\multicolumn{1}{c}{{\small$g$}}
&
{\small\emph{ID}}
&
\multicolumn{1}{c}{{\small$g$}}
&
{\small\emph{ID}}
&
\multicolumn{1}{c}{{\small$g$}}
\\[3pt]

$276$
&
${\scriptstyle T_{1}^{3}T_{4}T_{7}^{5}+T_{5}^{3}+T_{6}^{2}}$
&
$277$
&
${\scriptstyle T_{1}^{2}T_{2}T_{4}T_{7}^{5}+T_{5}^{3}+T_{6}^{2}}$
&
$278$
&
${\scriptstyle T_{1}T_{2}T_{3}T_{4}T_{7}^{5}+T_{5}^{3}+T_{6}^{2}}$
\\[2pt] \midrule \\[-8pt]

\multicolumn{6}{l}
{
\small\setlength{\tabcolsep}{8pt}
\begin{tabular}{cccc}
$
 Q \ = \ 
\left[\tiny\begin{array}{rrrrrrr}
    1 & 1 & 1 & 3 & 4 & 6 & 0 \\
    0 & 0 & 0 & 1 & 2 & 3 & 1 
\end{array}\right]
$
&
$\mu \ = \ (12,6)$
&
$-\KKK \ = \ (4,1)$
&
$\KKK^4 \ = \ 35$
\end{tabular}
}
\\[6pt]

{\small\emph{ID}}
&
\multicolumn{1}{c}{{\small$g$}}
&
{\small\emph{ID}}
&
\multicolumn{1}{c}{{\small$g$}}
&
{\small\emph{ID}}
&
\multicolumn{1}{c}{{\small$g$}}
\\[3pt]

$283$
&
${\scriptstyle T_{1}^{2}T_{2}T_{4}^{3}T_{7}^{3}+T_{5}^{3}+T_{6}^{2}}$
&
$284$
&
${\scriptstyle T_{1}T_{2}T_{3}T_{4}^{3}T_{7}^{3}+T_{5}^{3}+T_{6}^{2}}$
&
$285$
&
${\scriptstyle T_{1}^{5}T_{2}T_{4}^{2}T_{7}^{4}+T_{5}^{3}+T_{6}^{2}}$
\\
$286$
&
${\scriptstyle T_{1}^{3}T_{2}^{3}T_{4}^{2}T_{7}^{4}+T_{5}^{3}+T_{6}^{2}}$
&
$287$
&
${\scriptstyle T_{1}^{4}T_{2}T_{3}T_{4}^{2}T_{7}^{4}+T_{5}^{3}+T_{6}^{2}}$
&
$288$
&
${\scriptstyle T_{1}^{3}T_{2}^{2}T_{3}T_{4}^{2}T_{7}^{4}+T_{5}^{3}+T_{6}^{2}}$
\\
$289$
&
${\scriptstyle T_{1}^{9}T_{4}T_{7}^{5}+T_{5}^{3}+T_{6}^{2}}$
&
$290$
&
${\scriptstyle T_{1}^{8}T_{2}T_{4}T_{7}^{5}+T_{5}^{3}+T_{6}^{2}}$
&
$291$
&
${\scriptstyle T_{1}^{7}T_{2}^{2}T_{4}T_{7}^{5}+T_{5}^{3}+T_{6}^{2}}$
\\
$292$
&
${\scriptstyle T_{1}^{6}T_{2}^{3}T_{4}T_{7}^{5}+T_{5}^{3}+T_{6}^{2}}$
&
$293$
&
${\scriptstyle T_{1}^{5}T_{2}^{4}T_{4}T_{7}^{5}+T_{5}^{3}+T_{6}^{2}}$
&
$294$
&
${\scriptstyle T_{1}^{7}T_{2}T_{3}T_{4}T_{7}^{5}+T_{5}^{3}+T_{6}^{2}}$
\\
$295$
&
${\scriptstyle T_{1}^{6}T_{2}^{2}T_{3}T_{4}T_{7}^{5}+T_{5}^{3}+T_{6}^{2}}$
&
$296$
&
${\scriptstyle T_{1}^{5}T_{2}^{3}T_{3}T_{4}T_{7}^{5}+T_{5}^{3}+T_{6}^{2}}$
&
$297$
&
${\scriptstyle T_{1}^{4}T_{2}^{4}T_{3}T_{4}T_{7}^{5}+T_{5}^{3}+T_{6}^{2}}$
\\
$298$
&
${\scriptstyle T_{1}^{5}T_{2}^{2}T_{3}^{2}T_{4}T_{7}^{5}+T_{5}^{3}+T_{6}^{2}}$
&
$299$
&
${\scriptstyle T_{1}^{4}T_{2}^{3}T_{3}^{2}T_{4}T_{7}^{5}+T_{5}^{3}+T_{6}^{2}}$
&
$300$
&
${\scriptstyle T_{1}^{3}T_{2}^{3}T_{3}^{3}T_{4}T_{7}^{5}+T_{5}^{3}+T_{6}^{2}}$
\\[2pt] \midrule \\[-8pt]

\multicolumn{6}{l}
{
\small\setlength{\tabcolsep}{8pt}
\begin{tabular}{cccc}
$
 Q \ = \ 
\left[\tiny\begin{array}{rrrrrrr}
    1 & 1 & 1 & 4 & 4 & 6 & 0 \\
    0 & 0 & 0 & 1 & 2 & 3 & 1 
\end{array}\right]
$
&
$\mu \ = \ (12,6)$
&
$-\KKK \ = \ (5,1)$
&
$\KKK^4 \ = \ 58$
\end{tabular}
}
\\[6pt]

{\small\emph{ID}}
&
\multicolumn{1}{c}{{\small$g$}}
&
{\small\emph{ID}}
&
\multicolumn{1}{c}{{\small$g$}}
&
{\small\emph{ID}}
&
\multicolumn{1}{c}{{\small$g$}}
\\[3pt]

$301$
&
${\scriptstyle T_{1}^{3}T_{2}T_{4}^{2}T_{7}^{4}+T_{5}^{3}+T_{6}^{2}}$
&
$302$
&
${\scriptstyle T_{1}^{2}T_{2}T_{3}T_{4}^{2}T_{7}^{4}+T_{5}^{3}+T_{6}^{2}}$
&
$303$
&
${\scriptstyle T_{1}^{8}T_{4}T_{7}^{5}+T_{5}^{3}+T_{6}^{2}}$
\\
$304$
&
${\scriptstyle T_{1}^{7}T_{2}T_{4}T_{7}^{5}+T_{5}^{3}+T_{6}^{2}}$
&
$305$
&
${\scriptstyle T_{1}^{6}T_{2}^{2}T_{4}T_{7}^{5}+T_{5}^{3}+T_{6}^{2}}$
&
$306$
&
${\scriptstyle T_{1}^{5}T_{2}^{3}T_{4}T_{7}^{5}+T_{5}^{3}+T_{6}^{2}}$
\\
$307$
&
${\scriptstyle T_{1}^{4}T_{2}^{4}T_{4}T_{7}^{5}+T_{5}^{3}+T_{6}^{2}}$
&
$308$
&
${\scriptstyle T_{1}^{6}T_{2}T_{3}T_{4}T_{7}^{5}+T_{5}^{3}+T_{6}^{2}}$
&
$309$
&
${\scriptstyle T_{1}^{5}T_{2}^{2}T_{3}T_{4}T_{7}^{5}+T_{5}^{3}+T_{6}^{2}}$
\\
$310$
&
${\scriptstyle T_{1}^{4}T_{2}^{3}T_{3}T_{4}T_{7}^{5}+T_{5}^{3}+T_{6}^{2}}$
&
$311$
&
${\scriptstyle T_{1}^{4}T_{2}^{2}T_{3}^{2}T_{4}T_{7}^{5}+T_{5}^{3}+T_{6}^{2}}$
&
$312$
&
${\scriptstyle T_{1}^{3}T_{2}^{3}T_{3}^{2}T_{4}T_{7}^{5}+T_{5}^{3}+T_{6}^{2}}$
\\[2pt] \midrule \\[-8pt]

\multicolumn{6}{l}
{
\small\setlength{\tabcolsep}{8pt}
\begin{tabular}{cccc}
$
 Q \ = \ 
\left[\tiny\begin{array}{rrrrrrr}
    1 & 1 & 1 & 5 & 4 & 6 & 0 \\
    0 & 0 & 0 & 1 & 2 & 3 & 1 
\end{array}\right]
$
&
$\mu \ = \ (12,6)$
&
$-\KKK \ = \ (6,1)$
&
$\KKK^4 \ = \ 87$
\end{tabular}
}
\\[6pt]

{\small\emph{ID}}
&
\multicolumn{1}{c}{{\small$g$}}
&
{\small\emph{ID}}
&
\multicolumn{1}{c}{{\small$g$}}
&
{\small\emph{ID}}
&
\multicolumn{1}{c}{{\small$g$}}
\\[3pt]

$313$
&
${\scriptstyle T_{1}T_{2}T_{4}^{2}T_{7}^{4}+T_{5}^{3}+T_{6}^{2}}$
&
$314$
&
${\scriptstyle T_{1}^{7}T_{4}T_{7}^{5}+T_{5}^{3}+T_{6}^{2}}$
&
$315$
&
${\scriptstyle T_{1}^{6}T_{2}T_{4}T_{7}^{5}+T_{5}^{3}+T_{6}^{2}}$
\\
$316$
&
${\scriptstyle T_{1}^{5}T_{2}^{2}T_{4}T_{7}^{5}+T_{5}^{3}+T_{6}^{2}}$
&
$317$
&
${\scriptstyle T_{1}^{4}T_{2}^{3}T_{4}T_{7}^{5}+T_{5}^{3}+T_{6}^{2}}$
&
$318$
&
${\scriptstyle T_{1}^{5}T_{2}T_{3}T_{4}T_{7}^{5}+T_{5}^{3}+T_{6}^{2}}$
\\
$319$
&
${\scriptstyle T_{1}^{4}T_{2}^{2}T_{3}T_{4}T_{7}^{5}+T_{5}^{3}+T_{6}^{2}}$
&
$320$
&
${\scriptstyle T_{1}^{3}T_{2}^{3}T_{3}T_{4}T_{7}^{5}+T_{5}^{3}+T_{6}^{2}}$
&
$321$
&
${\scriptstyle T_{1}^{3}T_{2}^{2}T_{3}^{2}T_{4}T_{7}^{5}+T_{5}^{3}+T_{6}^{2}}$
\\[2pt] \midrule \\[-8pt]

\multicolumn{6}{l}
{
\small\setlength{\tabcolsep}{8pt}
\begin{tabular}{cccc}
$
 Q \ = \ 
\left[\tiny\begin{array}{rrrrrrr}
    1 & 1 & 1 & 6 & 4 & 6 & 0 \\
    0 & 0 & 0 & 1 & 2 & 3 & 1 
\end{array}\right]
$
&
$\mu \ = \ (12,6)$
&
$-\KKK \ = \ (7,1)$
&
$\KKK^4 \ = \ 122$
\end{tabular}
}
\\[6pt]

{\small\emph{ID}}
&
\multicolumn{1}{c}{{\small$g$}}
&
{\small\emph{ID}}
&
\multicolumn{1}{c}{{\small$g$}}
&
{\small\emph{ID}}
&
\multicolumn{1}{c}{{\small$g$}}
\\[3pt]

$322$
&
${\scriptstyle T_{1}^{6}T_{4}T_{7}^{5}+T_{5}^{3}+T_{6}^{2}}$
&
$323$
&
${\scriptstyle T_{1}^{5}T_{2}T_{4}T_{7}^{5}+T_{5}^{3}+T_{6}^{2}}$
&
$324$
&
${\scriptstyle T_{1}^{4}T_{2}^{2}T_{4}T_{7}^{5}+T_{5}^{3}+T_{6}^{2}}$
\\
$325$
&
${\scriptstyle T_{1}^{3}T_{2}^{3}T_{4}T_{7}^{5}+T_{5}^{3}+T_{6}^{2}}$
&
$326$
&
${\scriptstyle T_{1}^{4}T_{2}T_{3}T_{4}T_{7}^{5}+T_{5}^{3}+T_{6}^{2}}$
&
$327$
&
${\scriptstyle T_{1}^{3}T_{2}^{2}T_{3}T_{4}T_{7}^{5}+T_{5}^{3}+T_{6}^{2}}$
\\

$328$
&
${\scriptstyle T_{1}^{2}T_{2}^{2}T_{3}^{2}T_{4}T_{7}^{5}+T_{5}^{3}+T_{6}^{2}}$
&
&

&
&
\\[2pt] \midrule \\[-8pt]

\multicolumn{6}{l}
{
\small\setlength{\tabcolsep}{8pt}
\begin{tabular}{cccc}
$
 Q \ = \ 
\left[\tiny\begin{array}{rrrrrrr}
    1 & 1 & 1 & 7 & 4 & 6 & 0 \\
    0 & 0 & 0 & 1 & 2 & 3 & 1 
\end{array}\right]
$
&
$\mu \ = \ (12,6)$
&
$-\KKK \ = \ (8,1)$
&
$\KKK^4 \ = \ 163$
\end{tabular}
}
\\[6pt]

{\small\emph{ID}}
&
\multicolumn{1}{c}{{\small$g$}}
&
{\small\emph{ID}}
&
\multicolumn{1}{c}{{\small$g$}}
&
{\small\emph{ID}}
&
\multicolumn{1}{c}{{\small$g$}}
\\[3pt]

$329$
&
${\scriptstyle T_{1}^{5}T_{4}T_{7}^{5}+T_{5}^{3}+T_{6}^{2}}$
&
$330$
&
${\scriptstyle T_{1}^{4}T_{2}T_{4}T_{7}^{5}+T_{5}^{3}+T_{6}^{2}}$
&
$331$
&
${\scriptstyle T_{1}^{3}T_{2}^{2}T_{4}T_{7}^{5}+T_{5}^{3}+T_{6}^{2}}$
\\

$332$
&
${\scriptstyle T_{1}^{3}T_{2}T_{3}T_{4}T_{7}^{5}+T_{5}^{3}+T_{6}^{2}}$
&
$333$
&
${\scriptstyle T_{1}^{2}T_{2}^{2}T_{3}T_{4}T_{7}^{5}+T_{5}^{3}+T_{6}^{2}}$
&
&
\\[2pt] \midrule \\[-8pt]

\multicolumn{6}{l}
{
\small\setlength{\tabcolsep}{8pt}
\begin{tabular}{cccc}
$
 Q \ = \ 
\left[\tiny\begin{array}{rrrrrrr}
    1 & 1 & 1 & 8 & 4 & 6 & 0 \\
    0 & 0 & 0 & 1 & 2 & 3 & 1 
\end{array}\right]
$
&
$\mu \ = \ (12,6)$
&
$-\KKK \ = \ (9,1)$
&
$\KKK^4 \ = \ 210$
\end{tabular}
}
\\[6pt]

{\small\emph{ID}}
&
\multicolumn{1}{c}{{\small$g$}}
&
{\small\emph{ID}}
&
\multicolumn{1}{c}{{\small$g$}}
&
{\small\emph{ID}}
&
\multicolumn{1}{c}{{\small$g$}}
\\[3pt]

$334$
&
${\scriptstyle T_{1}^{4}T_{4}T_{7}^{5}+T_{5}^{3}+T_{6}^{2}}$
&
$335$
&
${\scriptstyle T_{1}^{3}T_{2}T_{4}T_{7}^{5}+T_{5}^{3}+T_{6}^{2}}$
&
$336$
&
${\scriptstyle T_{1}^{2}T_{2}^{2}T_{4}T_{7}^{5}+T_{5}^{3}+T_{6}^{2}}$
\\

$337$
&
${\scriptstyle T_{1}^{2}T_{2}T_{3}T_{4}T_{7}^{5}+T_{5}^{3}+T_{6}^{2}}$
&
&

&
&
\\ \bottomrule
\end{longtable}

\goodbreak

\begin{longtable}{llllll}
\toprule
\\[-8pt]

\multicolumn{6}{l}
{
\small\setlength{\tabcolsep}{8pt}
\begin{tabular}{cccc}
$
 Q \ = \ 
\left[\tiny\begin{array}{rrrrrrr}
    1 & 1 & 1 & 9 & 4 & 6 & 0 \\
    0 & 0 & 0 & 1 & 2 & 3 & 1 
\end{array}\right]
$
&
$\mu \ = \ (12,6)$
&
$-\KKK \ = \ (10,1)$
&
$\KKK^4 \ = \ 263$
\end{tabular}
}
\\[6pt]

{\small\emph{ID}}
&
\multicolumn{1}{c}{{\small$g$}}
&
{\small\emph{ID}}
&
\multicolumn{1}{c}{{\small$g$}}
&
{\small\emph{ID}}
&
\multicolumn{1}{c}{{\small$g$}}
\\[3pt]

$338$
&
${\scriptstyle T_{1}^{3}T_{4}T_{7}^{5}+T_{5}^{3}+T_{6}^{2}}$
&
$339$
&
${\scriptstyle T_{1}^{2}T_{2}T_{4}T_{7}^{5}+T_{5}^{3}+T_{6}^{2}}$
&
$340$
&
${\scriptstyle T_{1}T_{2}T_{3}T_{4}T_{7}^{5}+T_{5}^{3}+T_{6}^{2}}$
\\[2pt] \midrule \\[-8pt]

\multicolumn{6}{l}
{
\small\setlength{\tabcolsep}{8pt}
\begin{tabular}{cccc}
$
 Q \ = \ 
\left[\tiny\begin{array}{rrrrrrr}
    1 & 1 & 2 & 2 & 2 & 3 & 0 \\
    0 & 0 & 1 & 1 & 2 & 3 & 1 
\end{array}\right]
$
&
$\mu \ = \ (6,6)$
&
$-\KKK \ = \ (5,2)$
&
$\KKK^4 \ = \ 80$
\end{tabular}
}
\\[6pt]

{\small\emph{ID}}
&
\multicolumn{1}{c}{{\small$g$}}
&
{\small\emph{ID}}
&
\multicolumn{1}{c}{{\small$g$}}
&
{\small\emph{ID}}
&
\multicolumn{1}{c}{{\small$g$}}
\\[3pt]

$354$
&
${\scriptstyle T_{3}^{2}T_{4}T_{7}^{3}+T_{5}^{3}+T_{6}^{2}}$
&
$355$
&
${\scriptstyle T_{1}T_{2}T_{3}^{2}T_{7}^{4}+T_{5}^{3}+T_{6}^{2}}$
&
$356$
&
${\scriptstyle T_{1}^{2}T_{3}T_{4}T_{7}^{4}+T_{5}^{3}+T_{6}^{2}}$
\\
$357$
&
${\scriptstyle T_{1}T_{2}T_{3}T_{4}T_{7}^{4}+T_{5}^{3}+T_{6}^{2}}$
&
$358$
&
${\scriptstyle T_{1}^{4}T_{3}T_{7}^{5}+T_{5}^{3}+T_{6}^{2}}$
&
$359$
&
${\scriptstyle T_{1}^{3}T_{2}T_{3}T_{7}^{5}+T_{5}^{3}+T_{6}^{2}}$
\\

$360$
&
${\scriptstyle T_{1}^{2}T_{2}^{2}T_{3}T_{7}^{5}+T_{5}^{3}+T_{6}^{2}}$
&
&

&
&
\\[2pt] \midrule \\[-8pt]

\multicolumn{6}{l}
{
\small\setlength{\tabcolsep}{8pt}
\begin{tabular}{cccc}
$
 Q \ = \ 
\left[\tiny\begin{array}{rrrrrrr}
    1 & 1 & 3 & 3 & 2 & 3 & 0 \\
    0 & 0 & 1 & 1 & 2 & 3 & 1 
\end{array}\right]
$
&
$\mu \ = \ (6,6)$
&
$-\KKK \ = \ (7,2)$
&
$\KKK^4 \ = \ 112$
\end{tabular}
}
\\[6pt]

{\small\emph{ID}}
&
\multicolumn{1}{c}{{\small$g$}}
&
{\small\emph{ID}}
&
\multicolumn{1}{c}{{\small$g$}}
&
{\small\emph{ID}}
&
\multicolumn{1}{c}{{\small$g$}}
\\[3pt]

$361$
&
${\scriptstyle T_{3}T_{4}T_{7}^{4}+T_{5}^{3}+T_{6}^{2}}$
&
$362$
&
${\scriptstyle T_{1}^{3}T_{3}T_{7}^{5}+T_{5}^{3}+T_{6}^{2}}$
&
$363$
&
${\scriptstyle T_{1}^{2}T_{2}T_{3}T_{7}^{5}+T_{5}^{3}+T_{6}^{2}}$
\\[2pt] \midrule \\[-8pt]

\multicolumn{6}{l}
{
\small\setlength{\tabcolsep}{8pt}
\begin{tabular}{cccc}
$
 Q \ = \ 
\left[\tiny\begin{array}{rrrrrrr}
    1 & 1 & 1 & 1 & 2 & 0 & 0 \\
    0 & 0 & 1 & 1 & 3 & 1 & 1 
\end{array}\right]
$
&
$\mu \ = \ (4,6)$
&
$-\KKK \ = \ (2,1)$
&
$\KKK^4 \ = \ 12$
\end{tabular}
}
\\[6pt]

{\small\emph{ID}}
&
\multicolumn{1}{c}{{\small$g$}}
&
{\small\emph{ID}}
&
\multicolumn{1}{c}{{\small$g$}}
&
{\small\emph{ID}}
&
\multicolumn{1}{c}{{\small$g$}}
\\[3pt]

$371$
&
${\scriptstyle T_{1}^{3}T_{2}T_{6}^{6}+T_{3}^{3}T_{4}T_{7}^{2}+T_{5}^{2}}$
&
$372$
&
${\scriptstyle T_{1}T_{3}^{3}T_{6}^{3}+T_{2}T_{4}^{3}T_{7}^{3}+T_{5}^{2}}$
&
$373$
&
${\scriptstyle T_{1}^{3}T_{3}T_{6}^{5}+T_{2}T_{4}^{3}T_{7}^{3}+T_{5}^{2}}$
\\[2pt] \midrule \\[-8pt]

\multicolumn{6}{l}
{
\small\setlength{\tabcolsep}{8pt}
\begin{tabular}{cccc}
$
 Q \ = \ 
\left[\tiny\begin{array}{rrrrrrr}
    1 & 1 & 2 & 2 & 3 & 0 & 0 \\
    0 & 0 & 1 & 1 & 3 & 1 & 1 
\end{array}\right]
$
&
$\mu \ = \ (6,6)$
&
$-\KKK \ = \ (3,1)$
&
$\KKK^4 \ = \ 16$
\end{tabular}
}
\\[6pt]

{\small\emph{ID}}
&
\multicolumn{1}{c}{{\small$g$}}
&
{\small\emph{ID}}
&
\multicolumn{1}{c}{{\small$g$}}
&
{\small\emph{ID}}
&
\multicolumn{1}{c}{{\small$g$}}
\\[3pt]

$374$
&
${\scriptstyle T_{1}T_{2}T_{3}^{2}T_{6}^{4}+T_{4}^{3}T_{7}^{3}+T_{5}^{2}}$
&
$375$
&
${\scriptstyle T_{1}^{4}T_{3}T_{6}^{5}+T_{4}^{3}T_{7}^{3}+T_{5}^{2}}$
&
$376$
&
${\scriptstyle T_{1}^{3}T_{2}T_{3}T_{6}^{5}+T_{4}^{3}T_{7}^{3}+T_{5}^{2}}$
\\

$377$
&
${\scriptstyle T_{1}^{2}T_{2}^{2}T_{3}T_{6}^{5}+T_{4}^{3}T_{7}^{3}+T_{5}^{2}}$
&
&

&
&
\\[2pt] \midrule \\[-8pt]

\multicolumn{6}{l}
{
\small\setlength{\tabcolsep}{8pt}
\begin{tabular}{cccc}
$
 Q \ = \ 
\left[\tiny\begin{array}{rrrrrrr}
    1 & 1 & 2 & 1 & 1 & 3 & 0 \\
    0 & 0 & 1 & 1 & 1 & 3 & 1 
\end{array}\right]
$
&
$\mu \ = \ (6,6)$
&
$-\KKK \ = \ (3,1)$
&
$\KKK^4 \ = \ 16$
\end{tabular}
}
\\[6pt]

{\small\emph{ID}}
&
\multicolumn{1}{c}{{\small$g$}}
&
{\small\emph{ID}}
&
\multicolumn{1}{c}{{\small$g$}}
&
{\small\emph{ID}}
&
\multicolumn{1}{c}{{\small$g$}}
\\[3pt]

$382$
&
${\scriptstyle T_{3}^{3}T_{7}^{3}+T_{4}^{5}T_{5}+T_{6}^{2}}$
&
$383$
&
${\scriptstyle T_{1}T_{2}T_{3}^{2}T_{7}^{4}+T_{4}^{5}T_{5}+T_{6}^{2}}$
&
$384$
&
${\scriptstyle T_{1}T_{2}T_{3}^{2}T_{7}^{4}+T_{4}^{3}T_{5}^{3}+T_{6}^{2}}$
\\
$385$
&
${\scriptstyle T_{1}^{4}T_{3}T_{7}^{5}+T_{4}^{5}T_{5}+T_{6}^{2}}$
&
$386$
&
${\scriptstyle T_{1}^{4}T_{3}T_{7}^{5}+T_{4}^{3}T_{5}^{3}+T_{6}^{2}}$
&
$387$
&
${\scriptstyle T_{1}^{3}T_{2}T_{3}T_{7}^{5}+T_{4}^{5}T_{5}+T_{6}^{2}}$
\\
$388$
&
${\scriptstyle T_{1}^{3}T_{2}T_{3}T_{7}^{5}+T_{4}^{3}T_{5}^{3}+T_{6}^{2}}$
&
$389$
&
${\scriptstyle T_{1}^{2}T_{2}^{2}T_{3}T_{7}^{5}+T_{4}^{5}T_{5}+T_{6}^{2}}$
&
$390$
&
${\scriptstyle T_{1}^{2}T_{2}^{2}T_{3}T_{7}^{5}+T_{4}^{3}T_{5}^{3}+T_{6}^{2}}$
\\[2pt] \midrule \\[-8pt]

\multicolumn{6}{l}
{
\small\setlength{\tabcolsep}{8pt}
\begin{tabular}{cccc}
$
 Q \ = \ 
\left[\tiny\begin{array}{rrrrrrr}
    1 & 1 & 3 & 1 & 1 & 3 & 0 \\
    0 & 0 & 1 & 1 & 1 & 3 & 1 
\end{array}\right]
$
&
$\mu \ = \ (6,6)$
&
$-\KKK \ = \ (4,1)$
&
$\KKK^4 \ = \ 22$
\end{tabular}
}
\\[6pt]

{\small\emph{ID}}
&
\multicolumn{1}{c}{{\small$g$}}
&
{\small\emph{ID}}
&
\multicolumn{1}{c}{{\small$g$}}
&
{\small\emph{ID}}
&
\multicolumn{1}{c}{{\small$g$}}
\\[3pt]

$391$
&
${\scriptstyle T_{1}^{3}T_{3}T_{7}^{5}+T_{4}^{5}T_{5}+T_{6}^{2}}$
&
$392$
&
${\scriptstyle T_{1}^{3}T_{3}T_{7}^{5}+T_{4}^{3}T_{5}^{3}+T_{6}^{2}}$
&
$393$
&
${\scriptstyle T_{1}^{2}T_{2}T_{3}T_{7}^{5}+T_{4}^{5}T_{5}+T_{6}^{2}}$
\\

$394$
&
${\scriptstyle T_{1}^{2}T_{2}T_{3}T_{7}^{5}+T_{4}^{3}T_{5}^{3}+T_{6}^{2}}$
&
&

&
&
\\[2pt] \midrule \\[-8pt]

\multicolumn{6}{l}
{
\small\setlength{\tabcolsep}{8pt}
\begin{tabular}{cccc}
$
 Q \ = \ 
\left[\tiny\begin{array}{rrrrrrr}
    1 & 1 & 4 & 1 & 1 & 3 & 0 \\
    0 & 0 & 1 & 1 & 1 & 3 & 1 
\end{array}\right]
$
&
$\mu \ = \ (6,6)$
&
$-\KKK \ = \ (5,1)$
&
$\KKK^4 \ = \ 28$
\end{tabular}
}
\\[6pt]

{\small\emph{ID}}
&
\multicolumn{1}{c}{{\small$g$}}
&
{\small\emph{ID}}
&
\multicolumn{1}{c}{{\small$g$}}
&
{\small\emph{ID}}
&
\multicolumn{1}{c}{{\small$g$}}
\\[3pt]

$395$
&
${\scriptstyle T_{1}^{2}T_{3}T_{7}^{5}+T_{4}^{5}T_{5}+T_{6}^{2}}$
&
$396$
&
${\scriptstyle T_{1}^{2}T_{3}T_{7}^{5}+T_{4}^{3}T_{5}^{3}+T_{6}^{2}}$
&
$397$
&
${\scriptstyle T_{1}T_{2}T_{3}T_{7}^{5}+T_{4}^{5}T_{5}+T_{6}^{2}}$
\\

$398$
&
${\scriptstyle T_{1}T_{2}T_{3}T_{7}^{5}+T_{4}^{3}T_{5}^{3}+T_{6}^{2}}$
&
&

&
&
\\[2pt] \midrule \\[-8pt]

\multicolumn{6}{l}
{
\small\setlength{\tabcolsep}{8pt}
\begin{tabular}{cccc}
$
 Q \ = \ 
\left[\tiny\begin{array}{rrrrrrr}
    1 & 1 & 2 & 1 & 0 & 0 & 0 \\
    0 & 0 & 1 & 3 & 1 & 1 & 1 
\end{array}\right]
$
&
$\mu \ = \ (2,6)$
&
$-\KKK \ = \ (3,1)$
&
$\KKK^4 \ = \ 20$
\end{tabular}
}
\\[6pt]

{\small\emph{ID}}
&
\multicolumn{1}{c}{{\small$g$}}
&
{\small\emph{ID}}
&
\multicolumn{1}{c}{{\small$g$}}
&
{\small\emph{ID}}
&
\multicolumn{1}{c}{{\small$g$}}
\\[3pt]

$405$
&
${\scriptstyle T_{1}^{2}T_{5}^{5}T_{6}+T_{3}T_{7}^{5}+T_{4}^{2}}$
&
$406$
&
${\scriptstyle T_{1}^{2}T_{5}^{3}T_{6}^{3}+T_{3}T_{7}^{5}+T_{4}^{2}}$
&
$407$
&
${\scriptstyle T_{1}T_{2}T_{5}^{6}+T_{3}T_{7}^{5}+T_{4}^{2}}$
\\
$408$
&
${\scriptstyle T_{1}T_{2}T_{5}^{5}T_{6}+T_{3}T_{7}^{5}+T_{4}^{2}}$
&
$409$
&
${\scriptstyle T_{1}T_{2}T_{5}^{4}T_{6}^{2}+T_{3}T_{7}^{5}+T_{4}^{2}}$
&
$410$
&
${\scriptstyle T_{1}T_{2}T_{5}^{3}T_{6}^{3}+T_{3}T_{7}^{5}+T_{4}^{2}}$
\\
\bottomrule
\end{longtable}
\end{center}

\end{class-list}

\medskip

\begin{class-list}\label{class:s=4-series}
Locally factorial Fano fourfoulds of Picard number two with a hypersurface Cox ring and an effective three-torus action: Specifying data for the series with $s = 4$.


\begin{center}
\small\setlength{\tabcolsep}{4pt}
\small\setlength{\arraycolsep}{2pt}
\begin{longtable}{cccll}
\toprule
\emph{ID} & $[w_1,\dots,w_7]$ & $\left[\tiny\begin{array}{r}\mu \\ -\KKK\end{array}\right]$ & \multicolumn{1}{c}{$g$} & 
\\ \midrule

\emph{S18}
&
$\left[\tiny\begin{array}{rrrrrrr}
    1 & 1 & a & a & c & c & 0 \\
    0 & 0 & 1 & 1 & 1 & 1 & 1 
\end{array}\right]$
&
$
\tiny\begin{array}{c}
    (a,4) \\
    (a\!+\!2c\!+\!2,1)
\end{array}
$
&
$T_{1}^{(a-4c)}T_{5}^{4}+T_{2}^{(a-4c)}T_{6}^{4}+T_{4}T_{7}^{3}$
&
$
\def\arraystretch{0.6}\begin{array}{l}
{\scriptstyle c \,\ge\, 1,\, a \,>\, 4c,} \\
{\scriptstyle a \text{\emph{ odd}} }
\end{array}
$
\\ \midrule

\emph{S22}
&
$\left[\tiny\begin{array}{rrrrrrr}
    1 & 1 & a & a & 2 & 3 & 0 \\
    0 & 0 & 1 & 1 & 2 & 3 & 1 
\end{array}\right]$
&
$
\tiny\begin{array}{c}
    (6,6) \\
    (2a\!+\!1,2)
\end{array}
$
&
$T_{1}^{5}T_{2}T_{7}^{6}+T_{5}^{3}+T_{6}^{2}$
&
$
\def\arraystretch{0.6}\begin{array}{l}
{\scriptstyle a \,\ge\, 2 }
\end{array}
$
\\ \midrule

\emph{S25}
&
$\left[\tiny\begin{array}{rrrrrrr}
    1 & 1 & a & a & c & 0 & 0 \\
    0 & 0 & 1 & 1 & 1 & 1 & 1 
\end{array}\right]$
&
$
\tiny\begin{array}{c}
    (4c,4) \\
    (2a\!-\!3c\!+\!2,1)
\end{array}
$
&
$T_{1}^{(a-4c)}T_{3}T_{6}^{3}+T_{2}^{(a-4c)}T_{4}T_{7}^{3}+T_{5}^{4}$
&
$
\def\arraystretch{0.6}\begin{array}{l}
{\scriptstyle c \,\ge\, 1,\, a \,>\, 4c }
\end{array}
$
\\ \midrule

\emph{S26}
&
$\left[\tiny\begin{array}{rrrrrrr}
    1 & 1 & 4a & 4a & a & 0 & 0 \\
    0 & 0 & 1 & 1 & 1 & 1 & 1 
\end{array}\right]$
&
$
\tiny\begin{array}{c}
    (4a,4) \\
    (5a\!+\!2,1)
\end{array}
$
&
$T_{1}^{(4a-l)}T_{2}^{l}T_{6}^{4}+T_{4}T_{7}^{3}+T_{5}^{4}$
&
$
\def\arraystretch{0.6}\begin{array}{l}
{\scriptstyle a \,\ge\, 1,\, l \text{\emph{ odd}},}\\
{\scriptstyle 0 \,\le\, l \,<\, 2c }
\end{array}
$
\\ \midrule

\emph{S27}
&
$\left[\tiny\begin{array}{rrrrrrr}
    1 & 1 & a & a & c & 0 & 0 \\
    0 & 0 & 1 & 1 & 3 & 1 & 1 
\end{array}\right]$
&
$
\tiny\begin{array}{c}
    (2c,6) \\
    (2a\!-\!c\!+\!2,1)
\end{array}
$
&
$T_{1}^{(2c-a)}T_{3}T_{6}^{5}+T_{2}^{(2c-a)}T_{4}T_{7}^{5}+T_{5}^{2}$
&
$
\def\arraystretch{0.6}\begin{array}{l}
{\scriptstyle a,\,c \,\ge\, 1,}\\
{\scriptstyle c-1 \,\le\, a \,<\, 2c }
\end{array}
$
\\ \midrule

\emph{S28}
&
$\left[\tiny\begin{array}{rrrrrrr}
    1 & 1 & a & a & a & 0 & 0 \\
    0 & 0 & 1 & 1 & 3 & 1 & 1 
\end{array}\right]$
&
$
\tiny\begin{array}{c}
    (2a,6) \\
    (a\!+\!2,1)
\end{array}
$
&
$T_{1}^{(2a-l)}T_{2}^{l}T_{6}^{6}+T_{3}T_{4}T_{7}^{4}+T_{5}^{2}$
&
$
\def\arraystretch{0.6}\begin{array}{l}
{\scriptstyle a \,\ge\, 1,\, l \text{\emph{ odd}},}\\
{\scriptstyle 1 \,\le\, l \,\le\, a }
\end{array}
$
\\ \midrule

\emph{S29}
&
$\left[\tiny\begin{array}{rrrrrrr}
    1 & 1 & 2a & 2a & a & 0 & 0 \\
    0 & 0 & 1 & 1 & 3 & 1 & 1 
\end{array}\right]$
&
$
\tiny\begin{array}{c}
    (2a,6) \\
    (3a\!+\!2,1)
\end{array}
$
&
$T_{5}^{2}+T_{1}^{(2a-l)}T_{2}^{l}T_{6}^{6}+T_{4}T_{7}^{5}$
&
$
\def\arraystretch{0.6}\begin{array}{l}
{\scriptstyle a \,\ge\, 1,\, l \text{\emph{ odd}}, } \\
{\scriptstyle 1 \,\le\, l \,\le\, a }
\end{array}
$
\\ \midrule

\emph{S30}
&
$\left[\tiny\begin{array}{rrrrrrr}
    1 & 1 & a & 1 & 1 & 1 & 0 \\
    0 & 0 & 1 & 1 & 1 & 1 & 1 
\end{array}\right]$
&
$
\tiny\begin{array}{c}
    (4,4) \\
    (a\!+\!1,1)
\end{array}
$
&
$T_{1}^{3}T_{2}T_{7}^{4}+T_{4}^{3}T_{5}+T_{6}^{4}$
&
$
\def\arraystretch{0.6}\begin{array}{l}
{\scriptstyle a \,\ge\, 2 }
\end{array}
$
\\ \midrule

\emph{S34}
&
$\left[\tiny\begin{array}{rrrrrrr}
    1 & 1 & a & c & 0 & 0 & 0 \\
    0 & 0 & 1 & 1 & 1 & 1 & 1 
\end{array}\right]$
&
$
\tiny\begin{array}{c}
    (c,4) \\
    (a\!+\!2,1)
\end{array}
$
&
$T_{1}^{c}T_{5}^{4}+T_{2}^{c}T_{6}^{4}+T_{4}T_{7}^{3}$
&
$
\def\arraystretch{0.6}\begin{array}{l}
{\scriptstyle c \,\ge\, 1,\, a \,>\, c, } \\
{\scriptstyle c \text{\emph{ odd}} }
\end{array}
$
\\ \bottomrule
\end{longtable}
\end{center}


\goodbreak

\begin{center}
\small\setlength{\tabcolsep}{4pt}
\small\setlength{\arraycolsep}{2pt}
\begin{longtable}{llllll}
\toprule
\\[-8pt]

\multicolumn{6}{l}
{
\small\setlength{\tabcolsep}{8pt}
\begin{tabular}{cccc}
$
 Q \ = \ 
\left[\tiny\begin{array}{rrrrrrr}
    1 & 1 & 1 & a & 2 & 3 & 0 \\
    0 & 0 & 0 & 1 & 2 & 3 & 1 
\end{array}\right]
$
&
$\mu \ = \ (6,6)$
&
$-\KKK \ = \ (a\!+\!2,1)$
\end{tabular}
}
\\[6pt]

{\small\emph{ID}}
&
\multicolumn{1}{c}{{\small$g$}}
&
{\small\emph{ID}}
&
\multicolumn{1}{c}{{\small$g$}}
&
{\small\emph{ID}}
&
\multicolumn{1}{c}{{\small$g$}}
\\[3pt]

\emph{S5}
&

$
\def\arraystretch{0.75}\begin{array}{l}
{\scriptstyle T_{1}^{5}T_{2}T_{7}^{6}+T_{5}^{3}+T_{6}^{2},} \\[2pt]
{\scriptstyle a \,\ge\, 2 }
\end{array}
$
&
\emph{S6}
&

$
\def\arraystretch{0.75}\begin{array}{l}
{\scriptstyle T_{1}^{4}T_{2}T_{3}T_{7}^{6}+T_{5}^{3}+T_{6}^{2},} \\[2pt]
{\scriptstyle a \,\ge\, 2 }
\end{array}
$
&

\emph{S7}
&

$
\def\arraystretch{0.75}\begin{array}{l}
{\scriptstyle T_{1}^{3}T_{2}^{2}T_{3}T_{7}^{6}+T_{5}^{3}+T_{6}^{2},} \\[2pt]
{\scriptstyle a \,\ge\, 2 }
\end{array}
$
\\[4pt] \midrule \\[-8pt]

\multicolumn{6}{l}
{
\small\setlength{\tabcolsep}{8pt}
\begin{tabular}{cccc}
$
 Q \ = \ 
\left[\tiny\begin{array}{rrrrrrr}
    1 & 1 & 1 & a & 4 & 6 & 0 \\
    0 & 0 & 0 & 1 & 2 & 3 & 1 
\end{array}\right]
$
&
$\mu \ = \ (12,6)$
&
$-\KKK \ = \ (a\!+\!1,1)$
\end{tabular}
}
\\[6pt]

{\small\emph{ID}}
&
\multicolumn{1}{c}{{\small$g$}}
&
{\small\emph{ID}}
&
\multicolumn{1}{c}{{\small$g$}}
&
{\small\emph{ID}}
&
\multicolumn{1}{c}{{\small$g$}}
\\[3pt]

\emph{S8}
&

$
\def\arraystretch{0.75}\begin{array}{l}
{\scriptstyle T_{1}^{11}T_{2}T_{7}^{6}+T_{5}^{3}+T_{6}^{2},} \\[2pt]
{\scriptstyle a \,\ge\, 3 }
\end{array}
$
&
\emph{S9}
&

$
\def\arraystretch{0.75}\begin{array}{l}
{\scriptstyle T_{1}^{7}T_{2}^{5}T_{7}^{6}+T_{5}^{3}+T_{6}^{2},} \\[2pt]
{\scriptstyle a \,\ge\, 3 }
\end{array}
$
&
\emph{S10}
&

$
\def\arraystretch{0.75}\begin{array}{l}
{\scriptstyle T_{1}^{10}T_{2}T_{3}T_{7}^{6}+T_{5}^{3}+T_{6}^{2},} \\[2pt]
{\scriptstyle a \,\ge\, 3 }
\end{array}
$
\\[10pt]
\emph{S11}
&

$
\def\arraystretch{0.75}\begin{array}{l}
{\scriptstyle T_{1}^{9}T_{2}^{2}T_{3}T_{7}^{6}+T_{5}^{3}+T_{6}^{2},} \\[2pt]
{\scriptstyle a \,\ge\, 3 }
\end{array}
$
&
\emph{S12}
&

$
\def\arraystretch{0.75}\begin{array}{l}
{\scriptstyle T_{1}^{8}T_{2}^{3}T_{3}T_{7}^{6}+T_{5}^{3}+T_{6}^{2},} \\[2pt]
{\scriptstyle a \,\ge\, 3 }
\end{array}
$
&
\emph{S13}
&

$
\def\arraystretch{0.75}\begin{array}{l}
{\scriptstyle T_{1}^{7}T_{2}^{4}T_{3}T_{7}^{6}+T_{5}^{3}+T_{6}^{2},} \\[2pt]
{\scriptstyle a \,\ge\, 3 }
\end{array}
$
\\[10pt]
\emph{S14}
&

$
\def\arraystretch{0.75}\begin{array}{l}
{\scriptstyle T_{1}^{6}T_{2}^{5}T_{3}T_{7}^{6}+T_{5}^{3}+T_{6}^{2},} \\[2pt]
{\scriptstyle a \,\ge\, 3 }
\end{array}
$
&
\emph{S15}
&

$
\def\arraystretch{0.75}\begin{array}{l}
{\scriptstyle T_{1}^{7}T_{2}^{3}T_{3}^{2}T_{7}^{6}+T_{5}^{3}+T_{6}^{2},} \\[2pt]
{\scriptstyle a \,\ge\, 3 }
\end{array}
$
&
\emph{S16}
&

$
\def\arraystretch{0.75}\begin{array}{l}
{\scriptstyle T_{1}^{5}T_{2}^{5}T_{3}^{2}T_{7}^{6}+T_{5}^{3}+T_{6}^{2},} \\[2pt]
{\scriptstyle a \,\ge\, 3 }
\end{array}
$
\\[10pt]
\emph{S17}
&

$
\def\arraystretch{0.75}\begin{array}{l}
{\scriptstyle T_{1}^{5}T_{2}^{4}T_{3}^{3}T_{7}^{6}+T_{5}^{3}+T_{6}^{2},} \\[2pt]
{\scriptstyle a \,\ge\, 3 }
\end{array}
$
&
&
&
&
\\ \bottomrule
\end{longtable}

\goodbreak

\begin{longtable}{llll}
\toprule
\\[-8pt]

\multicolumn{4}{l}
{
\small\setlength{\tabcolsep}{8pt}
\begin{tabular}{cccc}
$
 Q \ = \ 
\left[\tiny\begin{array}{rrrrrrr}
    1 & 1 & a & a & 1 & 1 & 0 \\
    0 & 0 & 1 & 1 & 1 & 1 & 1 
\end{array}\right]
$
&
$\mu \ = \ (a\!+\!3,4)$
&
$-\KKK \ = \ (a\!+\!1,1)$
\end{tabular}
}
\\[6pt]

{\small\emph{ID}}
&
\multicolumn{1}{c}{{\small$g$}}
&
{\small\emph{ID}}
&
\multicolumn{1}{c}{{\small$g$}}
\\[3pt]

\emph{S19}
&

$
\def\arraystretch{0.75}\begin{array}{l}
{\scriptstyle T_{3}T_{5}^{3}+T_{4}T_{6}^{3}+T_{1}^{(a+3-l)}T_{2}^{l}T_{7}^{4},} \\[2pt]
{\scriptstyle a \,\ge\, 2,\, 0 \,\le\, l \,\le\, (a+3)/2 }
\end{array}
$
&
\emph{S20}
&

$
\def\arraystretch{0.75}\begin{array}{l}
{\scriptstyle T_{1}^{(a-1)}T_{5}^{4}+T_{3}T_{6}^{3}+T_{2}^{(a+3)}T_{7}^{4},} \\[2pt]
{\scriptstyle a \,\ge\, 2,\, a \text{\emph{ even}} }
\end{array}
$
\\[10pt]

\emph{S21}
&

$
\def\arraystretch{0.75}\begin{array}{l}
{\scriptstyle T_{1}^{(a-1)}T_{5}^{4}+T_{3}T_{6}^{3}+T_{2}^{3}T_{4}T_{7}^{3},} \\[2pt]
{\scriptstyle a \,\ge\, 2 }
\end{array}
$
&
&
\\ \bottomrule
\end{longtable}

\goodbreak

\begin{longtable}{llll}
\toprule
\\[-8pt]

\multicolumn{4}{l}
{
\small\setlength{\tabcolsep}{8pt}
\begin{tabular}{cccc}
$
 Q \ = \ 
\left[\tiny\begin{array}{rrrrrrr}
    1 & 1 & a & a & c & 0 & 0 \\
    0 & 0 & 1 & 1 & 1 & 1 & 1 
\end{array}\right]
$
&
$\mu \ = \ (a,4)$
&
$-\KKK \ = \ (a\!+\!c\!+\!2,1)$
\end{tabular}
}
\\[6pt]

{\small\emph{ID}}
&
\multicolumn{1}{c}{{\small$g$}}
&
{\small\emph{ID}}
&
\multicolumn{1}{c}{{\small$g$}}
\\[3pt]

\emph{S23}
&

$
\def\arraystretch{0.75}\begin{array}{l}
{\scriptstyle T_{1}^{(a-4c)}T_{5}^{4}+T_{2}^{a}T_{6}^{4}+T_{4}T_{7}^{3},} \\[2pt]
{\scriptstyle c \,\ge\, 1,\, a \,>\, 4c }
\end{array}
$
&
\emph{S24}
&

$
\def\arraystretch{0.75}\begin{array}{l}
{\scriptstyle T_{1}^{(a-l)}T_{2}^{l}T_{5}^{4}+T_{3}T_{6}^{3}+T_{4}T_{7}^{3},} \\[2pt]
{\scriptstyle c \,\ge\, 1,\, a \,>\, 4c,\, 0 \,\le\, l \,\le\, a/2 }
\end{array}
$
\\[4pt] \midrule \\[-8pt]

\multicolumn{4}{l}
{
\small\setlength{\tabcolsep}{8pt}
\begin{tabular}{cccc}
$
 Q \ = \ 
\left[\tiny\begin{array}{rrrrrrr}
    1 & 1 & a & 1 & 1 & 3 & 0 \\
    0 & 0 & 1 & 1 & 1 & 3 & 1 
\end{array}\right]
$
&
$\mu \ = \ (6,6)$
&
$-\KKK \ = \ (a\!+\!1,1)$
\end{tabular}
}
\\[6pt]

{\small\emph{ID}}
&
\multicolumn{1}{c}{{\small$g$}}
&
{\small\emph{ID}}
&
\multicolumn{1}{c}{{\small$g$}}
\\[3pt]

\emph{S31}
&

$
\def\arraystretch{0.75}\begin{array}{l}
{\scriptstyle T_{1}^{5}T_{2}T_{7}^{6}+T_{4}^{5}T_{5}+T_{6}^{2},} \\[2pt]
{\scriptstyle a \,\ge\, 2 }
\end{array}
$
&
\emph{S32}
&

$
\def\arraystretch{0.75}\begin{array}{l}
{\scriptstyle T_{1}^{5}T_{2}T_{7}^{6}+T_{4}^{3}T_{5}^{3}+T_{6}^{2},} \\[2pt]
{\scriptstyle a \,\ge\, 2 }
\end{array}
$
\\[10pt]

\emph{S33}
&

$
\def\arraystretch{0.75}\begin{array}{l}
{\scriptstyle T_{1}^{3}T_{2}^{3}T_{7}^{6}+T_{4}^{5}T_{5}+T_{6}^{2},} \\[2pt]
{\scriptstyle a \,\ge\, 2 }
\end{array}
$
&
&
\\[4pt] \midrule \\[-8pt]

\multicolumn{4}{l}
{
\small\setlength{\tabcolsep}{8pt}
\begin{tabular}{cccc}
$
 Q \ = \ 
\left[\tiny\begin{array}{rrrrrrr}
    1 & 1 & a+1 & a & 0 & 0 & 0 \\
    0 & 0 & 1 & 1 & 1 & 1 & 1 
\end{array}\right]
$
&
$\mu \ = \ (a\!+\!1,4)$
&
$-\KKK \ = \ (a\!+\!2,1)$
\end{tabular}
}
\\[6pt]

{\small\emph{ID}}
&
\multicolumn{1}{c}{{\small$g$}}
&
{\small\emph{ID}}
&
\multicolumn{1}{c}{{\small$g$}}
\\[3pt]

\emph{S35}
&

$
\def\arraystretch{0.75}\begin{array}{l}
{\scriptstyle T_{1}^{(a+1)}T_{5}^{4}+T_{2}T_{4}T_{6}^{3}+T_{3}T_{7}^{3},} \\[2pt]
{\scriptstyle a \,\ge\, 1 }
\end{array}
$
&
\emph{S36}
&

$
\def\arraystretch{0.75}\begin{array}{l}
{\scriptstyle T_{1}^{(a+1)}T_{5}^{4}+T_{2}^{(a+1)}T_{6}^{4}+T_{3}T_{7}^{3},} \\[2pt]
{\scriptstyle a \,\ge\, 1,\, a \text{\emph{ even}} }
\end{array}
$
\\
\bottomrule
\end{longtable}
\end{center}
\end{class-list}

\medskip

\begin{class-list}\label{class:s=5-sample}
Locally factorial Fano fourfoulds of Picard number two with a hypersurface Cox ring and an effective three-torus action: Specifying data for the cases with $s = 5$ and $\mu \in \lambda$.
\begin{center}
\small\setlength{\arraycolsep}{2pt}
\begin{longtable}{cccccl}
\toprule
\emph{ID} & $[w_1,\dots,w_7]$ & $\mu$ & $-\KKK$ & $\KKK^4$ & \multicolumn{1}{c}{$g$}
\\ \midrule

$411$
&
\multirow{4}{*}{
$\left[\tiny\begin{array}{rrrrrrr}
    1 & 1 & 1 & 2 & 1 & 0 & 0 \\
    0 & 1 & 1 & 3 & 2 & 1 & 1 
\end{array}\right]$
}
&
\multirow{4}{*}{
$(4,6)$
}
&
\multirow{4}{*}{
$(2,3)$
}
&
\multirow{4}{*}{
$65$
}
&
$T_{2}^{3}T_{3}T_{7}^{2}+T_{1}T_{5}^{3}+T_{4}^{2}$
\\

$412$
&
&
&
&
&
$T_{2}^{3}T_{3}T_{6}T_{7}+T_{1}T_{5}^{3}+T_{4}^{2}$
\\

$413$
&
&
&
&
&
$T_{2}^{4}T_{6}T_{7}+T_{1}T_{5}^{3}+T_{4}^{2}$
\\

$414$
&
&
&
&
&
$T_{2}^{2}T_{3}^{2}T_{6}T_{7}+T_{1}T_{5}^{3}+T_{4}^{2}$
\\ \bottomrule
\end{longtable}
\end{center}
\end{class-list}

\medskip

\begin{class-list}\label{class:s=5-nsample}
Locally factorial Fano fourfoulds of Picard number two with a hypersurface Cox ring and an effective three-torus action: Specifying data for the sporadic cases with $s = 5$ and $\mu \not\in \lambda$.
\begin{center}
\small\setlength{\arraycolsep}{2pt}
\begin{longtable}{cccccl}
\toprule
\emph{ID} & $[w_1,\dots,w_7]$ & $\mu$ & $-\KKK$ & $\KKK^4$ & \multicolumn{1}{c}{$g$}
\\ \midrule

$415$
&
\multirow{4}{*}{
$\left[\tiny\begin{array}{rrrrrrr}
    1 & 1 & 1 & 2 & 3 & 1 & 0 \\
    0 & 0 & 0 & 1 & 2 & 1 & 1 
\end{array}\right]$
}
&
\multirow{4}{*}{
$(6,4)$
}
&
\multirow{4}{*}{
$(3,1)$
}
&
\multirow{4}{*}{
$50$
}
&
$T_{3}^{4}T_{4}T_{7}^{3}+T_{1}T_{2}T_{6}^{4}+T_{5}^{2}$
\\

$416$
&
&
&
&
&
$T_{1}T_{2}T_{6}^{4}+T_{4}^{3}T_{7}+T_{5}^{2}$
\\

$417$
&
&
&
&
&
$T_{1}^{5}T_{2}T_{7}^{4}+T_{3}T_{4}T_{6}^{3}+T_{5}^{2}$
\\

$418$
&
&
&
&
&
$T_{1}^{3}T_{2}^{3}T_{7}^{4}+T_{3}T_{4}T_{6}^{3}+T_{5}^{2}$
\\ \midrule

$419$
&
\multirow{4}{*}{
$\left[\tiny\begin{array}{rrrrrrr}
    1 & 1 & 1 & 3 & 3 & 1 & 0 \\
    0 & 0 & 0 & 1 & 2 & 1 & 1 
\end{array}\right]$
}
&
\multirow{4}{*}{
$(6,4)$
}
&
\multirow{4}{*}{
$(4,1)$
}
&
\multirow{4}{*}{
$90$
}
&
$T_{3}^{3}T_{4}T_{7}^{3}+T_{1}T_{2}T_{6}^{4}+T_{5}^{2}$
\\

$420$
&
&
&
&
&
$T_{1}^{5}T_{2}T_{7}^{4}+T_{4}T_{6}^{3}+T_{5}^{2}$
\\

$421$
&
&
&
&
&
$T_{1}^{4}T_{2}T_{3}T_{7}^{4}+T_{4}T_{6}^{3}+T_{5}^{2}$
\\

$422$
&
&
&
&
&
$T_{1}^{3}T_{2}^{2}T_{3}T_{7}^{4}+T_{4}T_{6}^{3}+T_{5}^{2}$
\\ \midrule

$423$
&
$\left[\tiny\begin{array}{rrrrrrr}
    1 & 1 & 1 & 4 & 3 & 1 & 0 \\
    0 & 0 & 0 & 1 & 2 & 1 & 1 
\end{array}\right]$
&
$(6,4)$
&
$(5,1)$
&
$142$
&
$T_{1}T_{2}T_{6}^{4}+T_{3}^{2}T_{4}T_{7}^{3}+T_{5}^{2}$
\\ \midrule

$424$
&
$\left[\tiny\begin{array}{rrrrrrr}
    1 & 1 & 1 & 5 & 3 & 1 & 0 \\
    0 & 0 & 0 & 1 & 2 & 1 & 1 
\end{array}\right]$
&
$(6,4)$
&
$(6,1)$
&
$206$
&
$T_{1}T_{2}T_{6}^{4}+T_{3}T_{4}T_{7}^{3}+T_{5}^{2}$
\\ \bottomrule
\end{longtable}

\goodbreak

\begin{longtable}{cccccl}
\toprule
\emph{ID} & $[w_1,\dots,w_7]$ & $\mu$ & $-\KKK$ & $\KKK^4$ & \multicolumn{1}{c}{$g$}
\\ \midrule

$425$
&
$\left[\tiny\begin{array}{rrrrrrr}
    1 & 1 & 1 & 6 & 3 & 1 & 0 \\
    0 & 0 & 0 & 1 & 2 & 1 & 1 
\end{array}\right]$
&
$(6,4)$
&
$(7,1)$
&
$282$
&
$T_{1}T_{2}T_{6}^{4}+T_{4}T_{7}^{3}+T_{5}^{2}$
\\ \midrule

$426$
&
\multirow{2}{*}{
$\left[\tiny\begin{array}{rrrrrrr}
    1 & 1 & 2 & 2 & 4 & 1 & 0 \\
    0 & 0 & 1 & 1 & 3 & 1 & 1 
\end{array}\right]$
}
&
\multirow{2}{*}{
$(8,6)$
}
&
\multirow{2}{*}{
$(3,1)$
}
&
\multirow{2}{*}{
$14$
}
&
$T_{1}T_{2}T_{6}^{6}+T_{3}^{3}T_{4}T_{7}^{2}+T_{5}^{2}$
\\

$427$
&
&
&
&
&
$T_{2}^{2}T_{4}^{3}T_{7}^{3}+T_{1}T_{3}T_{6}^{5}+T_{5}^{2}$
\\ \midrule

$428$
&
$\left[\tiny\begin{array}{rrrrrrr}
    1 & 1 & 3 & 3 & 5 & 1 & 0 \\
    0 & 0 & 1 & 1 & 3 & 1 & 1 
\end{array}\right]$
&
$(10,6)$
&
$(4,1)$
&
$18$
&
$T_{1}^{2}T_{3}T_{6}^{5}+T_{2}T_{4}^{3}T_{7}^{3}+T_{5}^{2}$
\\ \midrule

$429$
&
\multirow{2}{*}{
$\left[\tiny\begin{array}{rrrrrrr}
    1 & 1 & 4 & 4 & 6 & 1 & 0 \\
    0 & 0 & 1 & 1 & 3 & 1 & 1 
\end{array}\right]$
}
&
\multirow{2}{*}{
$(12,6)$
}
&
\multirow{2}{*}{
$(5,1)$
}
&
\multirow{2}{*}{
$22$
}
&
$T_{1}^{3}T_{3}T_{6}^{5}+T_{4}^{3}T_{7}^{3}+T_{5}^{2}$
\\

$430$
&
&
&
&
&
$T_{1}^{2}T_{2}T_{3}T_{6}^{5}+T_{4}^{3}T_{7}^{3}+T_{5}^{2}$
\\ \midrule

$431$
&
\multirow{2}{*}{
$\left[\tiny\begin{array}{rrrrrrr}
    1 & 1 & 4 & 4 & 7 & 2 & 0 \\
    0 & 0 & 1 & 1 & 3 & 1 & 1 
\end{array}\right]$
}
&
\multirow{2}{*}{
$(14,6)$
}
&
\multirow{2}{*}{
$(5,1)$
}
&
\multirow{2}{*}{
$20$
}
&
$T_{2}^{2}T_{4}^{3}T_{7}^{3}+T_{3}T_{6}^{5}+T_{5}^{2}$
\\

$432$
&
&
&
&
&
$T_{1}T_{2}T_{4}^{3}T_{7}^{3}+T_{3}T_{6}^{5}+T_{5}^{2}$
\\ \midrule

$433$
&
$\left[\tiny\begin{array}{rrrrrrr}
    1 & 1 & 5 & 5 & 8 & 2 & 0 \\
    0 & 0 & 1 & 1 & 3 & 1 & 1 
\end{array}\right]$
&
$(16,6)$
&
$(6,1)$
&
$24$
&
$T_{1}T_{3}T_{6}^{5}+T_{2}T_{4}^{3}T_{7}^{3}+T_{5}^{2}$
\\ \midrule

$434$
&
\multirow{2}{*}{
$\left[\tiny\begin{array}{rrrrrrr}
    1 & 1 & 6 & 6 & 9 & 2 & 0 \\
    0 & 0 & 1 & 1 & 3 & 1 & 1 
\end{array}\right]$
}
&
\multirow{2}{*}{
$(18,6)$
}
&
\multirow{2}{*}{
$(7,1)$
}
&
\multirow{2}{*}{
$28$
}
&
$T_{1}^{2}T_{3}T_{6}^{5}+T_{4}^{3}T_{7}^{3}+T_{5}^{2}$
\\

$435$
&
&
&
&
&
$T_{1}T_{2}T_{3}T_{6}^{5}+T_{4}^{3}T_{7}^{3}+T_{5}^{2}$
\\ \midrule

$436$
&
$\left[\tiny\begin{array}{rrrrrrr}
    1 & 1 & 7 & 7 & 11 & 3 & 0 \\
    0 & 0 & 1 & 1 & 3 & 1 & 1 
\end{array}\right]$
&
$(22,6)$
&
$(8,1)$
&
$30$
&
$T_{2}T_{4}^{3}T_{7}^{3}+T_{3}T_{6}^{5}+T_{5}^{2}$
\\ \midrule

$437$
&
$\left[\tiny\begin{array}{rrrrrrr}
    1 & 1 & 8 & 8 & 12 & 3 & 0 \\
    0 & 0 & 1 & 1 & 3 & 1 & 1 
\end{array}\right]$
&
$(24,6)$
&
$(9,1)$
&
$34$
&
$T_{1}T_{3}T_{6}^{5}+T_{4}^{3}T_{7}^{3}+T_{5}^{2}$
\\ \midrule

$438$
&
$\left[\tiny\begin{array}{rrrrrrr}
    1 & 1 & 10 & 10 & 15 & 4 & 0 \\
    0 & 0 & 1 & 1 & 3 & 1 & 1 
\end{array}\right]$
&
$(30,6)$
&
$(11,1)$
&
$40$
&
$T_{3}T_{6}^{5}+T_{4}^{3}T_{7}^{3}+T_{5}^{2}$
\\ \midrule

$439$
&
\multirow{2}{*}{
$\left[\tiny\begin{array}{rrrrrrr}
    1 & 1 & 3 & 2 & 3 & 0 & 0 \\
    0 & 0 & 1 & 1 & 3 & 1 & 1 
\end{array}\right]$
}
&
\multirow{2}{*}{
$(6,6)$
}
&
\multirow{2}{*}{
$(4,1)$
}
&
\multirow{2}{*}{
$22$
}
&
$T_{1}^{3}T_{3}T_{6}^{5}+T_{4}^{3}T_{7}^{3}+T_{5}^{2}$
\\

$440$
&
&
&
&
&
$T_{1}^{2}T_{2}T_{3}T_{6}^{5}+T_{4}^{3}T_{7}^{3}+T_{5}^{2}$
\\ \midrule

$441$
&
\multirow{2}{*}{
$\left[\tiny\begin{array}{rrrrrrr}
    1 & 1 & 4 & 2 & 3 & 0 & 0 \\
    0 & 0 & 1 & 1 & 3 & 1 & 1 
\end{array}\right]$
}
&
\multirow{2}{*}{
$(6,6)$
}
&
\multirow{2}{*}{
$(5,1)$
}
&
\multirow{2}{*}{
$28$
}
&
$T_{1}^{2}T_{3}T_{6}^{5}+T_{4}^{3}T_{7}^{3}+T_{5}^{2}$
\\

$442$
&
&
&
&
&
$T_{1}T_{2}T_{3}T_{6}^{5}+T_{4}^{3}T_{7}^{3}+T_{5}^{2}$
\\ \midrule

$443$
&
$\left[\tiny\begin{array}{rrrrrrr}
    1 & 1 & 5 & 2 & 3 & 0 & 0 \\
    0 & 0 & 1 & 1 & 3 & 1 & 1 
\end{array}\right]$
&
$(6,6)$
&
$(6,1)$
&
$34$
&
$T_{1}T_{3}T_{6}^{5}+T_{4}^{3}T_{7}^{3}+T_{5}^{2}$
\\ \midrule

$444$
&
$\left[\tiny\begin{array}{rrrrrrr}
    1 & 1 & 6 & 2 & 3 & 0 & 0 \\
    0 & 0 & 1 & 1 & 3 & 1 & 1 
\end{array}\right]$
&
$(6,6)$
&
$(7,1)$
&
$40$
&
$T_{3}T_{6}^{5}+T_{4}^{3}T_{7}^{3}+T_{5}^{2}$
\\ \midrule

$445$
&
$\left[\tiny\begin{array}{rrrrrrr}
    1 & 1 & 3 & 1 & 2 & 0 & 0 \\
    0 & 0 & 1 & 1 & 3 & 1 & 1 
\end{array}\right]$
&
$(4,6)$
&
$(4,1)$
&
$24$
&
$T_{1}T_{3}T_{6}^{5}+T_{2}T_{4}^{3}T_{7}^{3}+T_{5}^{2}$
\\ \midrule

$446$
&
$\left[\tiny\begin{array}{rrrrrrr}
    1 & 1 & 4 & 1 & 2 & 0 & 0 \\
    0 & 0 & 1 & 1 & 3 & 1 & 1 
\end{array}\right]$
&
$(4,6)$
&
$(5,1)$
&
$30$
&
$T_{2}T_{4}^{3}T_{7}^{3}+T_{3}T_{6}^{5}+T_{5}^{2}$
\\ \midrule

$447$
&
$\left[\tiny\begin{array}{rrrrrrr}
    1 & 1 & 2 & 1 & 2 & 0 & 0 \\
    0 & 0 & 1 & 1 & 3 & 1 & 1 
\end{array}\right]$
&
$(4,6)$
&
$(3,1)$
&
$18$
&
$T_{1}^{3}T_{2}T_{6}^{6}+T_{3}T_{4}^{2}T_{7}^{3}+T_{5}^{2}$
\\ \bottomrule
\end{longtable}
\end{center}
\end{class-list}

\medskip

\begin{class-list}\label{class:s=5-series}
Locally factorial Fano fourfoulds of Picard number two with a hypersurface Cox ring and an effective three-torus action: Specifying data for the series with $s = 5$.


\begin{center}
\small\setlength{\tabcolsep}{3pt}
\small\setlength{\arraycolsep}{2pt}
\begin{longtable}{cccl}
\toprule
\emph{ID} & $[w_1,\dots,w_7]$ & $\left[\tiny\begin{array}{r}\mu \\ -\KKK\end{array}\right]$ & \multicolumn{1}{c}{$g$} 
\\[6pt] \midrule \\[-10pt]

\emph{S39}
&
$\left[\tiny\begin{array}{rrrrrrr}
    1 & 1 & a & a & b & c & 0 \\
    0 & 0 & 1 & 1 & 1 & 1 & 1 
\end{array}\right]$
&
$
\tiny\begin{array}{c}
    (a,4) \\[2pt]
    (a\!+\!c\!+\!b\!+\!2,1)
\end{array}
$
&
$
\def\arraystretch{0.6}\begin{array}{l}
{\scriptstyle T_{1}^{(a-4b)}T_{5}^{4}+T_{2}^{(a-4c)}T_{6}^{4}+T_{4}T_{7}^{3},} \\[2pt]
{\scriptstyle b \,>\, c \,\ge\, 1,\, a \,>\, 4b,\, a \text{\emph{ odd}} }
\end{array}
$
\\[6pt] \midrule \\[-10pt]

\emph{S42}
&
$\left[\tiny\begin{array}{rrrrrrr}
    1 & 1 & a & a & b & c & 0 \\
    0 & 0 & 1 & 1 & 1 & 1 & 1 
\end{array}\right]$
&
$
\tiny\begin{array}{c}
    (4a,4) \\[2pt]
    (c\!+\!b\!-\!2a\!+\!2,1)
\end{array}
$
&
$
\def\arraystretch{0.6}\begin{array}{l}
{\scriptstyle T_{5}^{4}+T_{1}^{(4b-a-3c)}T_{3}T_{6}^{3}+T_{2}^{(4b-a)}T_{4}T_{7}^{3},} \\[2pt]
{\scriptstyle b \,>\, c \,\ge\, 1,\, b \,>\, 2c-1,\, 3b-c-1 \,\le\, a \,<\, 4b-3c }
\end{array}
$
\\[6pt] \midrule \\[-10pt]

\emph{S43}
&
$\left[\tiny\begin{array}{rrrrrrr}
    1 & 1 & 2a\!+\!2 & 2a\!+\!2 & a\!+\!1 & a & 0 \\
    0 & 0 & 1 & 1 & 1 & 1 & 1 
\end{array}\right]$
&
$
\tiny\begin{array}{c}
    (4a\!+\!4,4) \\[2pt]
    (2a\!+\!3,1)
\end{array}
$
&
$
\def\arraystretch{0.6}\begin{array}{l}
{\scriptstyle T_{1}^{3}T_{2}^{2}T_{6}^{4}+T_{3}T_{4}T_{7}^{2}+T_{5}^{4},} \\[2pt]
{\scriptstyle a \,\ge\, 1 }
\end{array}
$
\\ \bottomrule
\end{longtable}

\goodbreak

\begin{longtable}{cccl}
\toprule
\emph{ID} & $[w_1,\dots,w_7]$ & $\left[\tiny\begin{array}{r}\mu \\ -\KKK\end{array}\right]$ & \multicolumn{1}{c}{$g$} 
\\[6pt] \midrule \\[-10pt]

\emph{S44}
&
$\left[\tiny\begin{array}{rrrrrrr}
    1 & 1 & 4a & 4a & a & b & 0 \\
    0 & 0 & 1 & 1 & 1 & 1 & 1 
\end{array}\right]$
&
$
\tiny\begin{array}{c}
    (4a,4) \\[2pt]
    (5a\!+\!b\!+\!2,1)
\end{array}
$
&
$
\def\arraystretch{0.6}\begin{array}{l}
{\scriptstyle T_{5}^{4}+T_{1}^{(4a-4b-l)}T_{2}^{l}T_{6}^{4}+T_{4}T_{7}^{3},} \\[2pt]
{\scriptstyle a \,>\, b \,\ge\, 1,\, 0 \,<\, l \,<\, 2a+2b,\, l \text{\emph{ odd}} }
\end{array}
$
\\[6pt] \midrule \\[-10pt]

\emph{S45}
&
$\left[\tiny\begin{array}{rrrrrrr}
    1 & 1 & a & a & a\!+\!2 & 1 & 0 \\
    0 & 0 & 1 & 1 & 3 & 1 & 1 
\end{array}\right]$
&
$
\tiny\begin{array}{c}
    (2a\!+\!4,6) \\[2pt]
    (a\!+\!1,1)
\end{array}
$
&
$
\def\arraystretch{0.6}\begin{array}{l}
{\scriptstyle T_{5}^{2}+T_{3}T_{4}T_{6}^{4}+T_{1}^{(2a+4-l)}T_{2}^{l}T_{7}^{6},} \\[2pt]
{\scriptstyle a \,\ge\, 1,\, 0 \,<\, l \,\le\, a+2,\, l \text{\emph{ odd}} }
\end{array}
$
\\[6pt] \midrule \\[-10pt]

\emph{S49}
&
$\left[\tiny\begin{array}{rrrrrrr}
    1 & 1 & a & a & b & c & 0 \\
    0 & 0 & 1 & 1 & 3 & 1 & 1 
\end{array}\right]$
&
$
\tiny\begin{array}{c}
    (2b,6) \\[2pt]
    (2a\!-\!b\!+\!c\!+\!2,1)
\end{array}
$
&
$
\def\arraystretch{0.6}\begin{array}{l}
{\scriptstyle T_{5}^{2}+T_{1}^{(2b-5c-a)}T_{3}T_{6}^{5}+T_{2}^{(2b-a)}T_{4}T_{7}^{5},} \\[2pt]
{\scriptstyle c \,\ge\, 1,\, b \,\ge\, 4c,\, b/3 \,<\, a \,<\, 2b-5c }
\end{array}
$
\\[6pt] \midrule \\[-10pt]

\emph{S53}
&
$\left[\tiny\begin{array}{rrrrrrr}
    1 & 1 & a & a & a & b & 0 \\
    0 & 0 & 1 & 1 & 3 & 1 & 1 
\end{array}\right]$
&
$
\tiny\begin{array}{c}
    (2a,6) \\[2pt]
    (a\!+\!b\!+\!2,1)
\end{array}
$
&
$
\def\arraystretch{0.6}\begin{array}{l}
{\scriptstyle T_{5}^{2}+T_{1}^{(2a-6b-l)}T_{2}^{l}T_{6}^{6}+T_{3}T_{4}T_{7}^{4},} \\[2pt]
{\scriptstyle b \,\ge\, 1,\, a \,>\, 3b,\, 0 \,<\, l \,\le\, a-3b,\, l \text{\emph{ odd}} }
\end{array}
$
\\[6pt] \midrule \\[-10pt]

\emph{S54}
&
$\left[\tiny\begin{array}{rrrrrrr}
    1 & 1 & 2a & 2a & a & b & 0 \\
    0 & 0 & 1 & 1 & 3 & 1 & 1 
\end{array}\right]$
&
$
\tiny\begin{array}{c}
    (2a,6) \\[2pt]
    (b\!+\!3a\!+\!2,1)
\end{array}
$
&
$
\def\arraystretch{0.6}\begin{array}{l}
{\scriptstyle T_{5}^{2}+T_{1}^{(2a-l)}T_{2}^{l}T_{6}^{6}+T_{4}T_{7}^{5},} \\[2pt]
{\scriptstyle b \,\ge\, 1,\, a \,>\, 3b,\, 0 \,<\, l \,\le\, a,\, l \text{\emph{ odd}} }
\end{array}
$
\\[6pt] \midrule \\[-10pt]

\emph{S55}
&
$\left[\tiny\begin{array}{rrrrrrr}
    1 & 1 & a\!+\!1 & a & 1 & 1 & 0 \\
    0 & 0 & 1 & 1 & 1 & 1 & 1 
\end{array}\right]$
&
$
\tiny\begin{array}{c}
    (a\!+\!3,4) \\[2pt]
    (a\!+\!2,1)
\end{array}
$
&
$
\def\arraystretch{0.6}\begin{array}{l}
{\scriptstyle T_{1}^{(a-1)}T_{5}^{4}+T_{4}T_{6}^{3}+T_{2}^{2}T_{3}T_{7}^{3},} \\[2pt]
{\scriptstyle a \,\ge\, 2 }
\end{array}
$
\\[6pt] \midrule \\[-10pt]

\emph{S56}
&
$\left[\tiny\begin{array}{rrrrrrr}
    1 & 1 & a\!+\!2 & a & 1 & 1 & 0 \\
    0 & 0 & 1 & 1 & 1 & 1 & 1 
\end{array}\right]$
&
$
\tiny\begin{array}{c}
    (a\!+\!3,4) \\[2pt]
    (a\!+\!3,1)
\end{array}
$
&
$
\def\arraystretch{0.6}\begin{array}{l}
{\scriptstyle T_{1}^{(a-1)}T_{5}^{4}+T_{4}T_{6}^{3}+T_{2}T_{3}T_{7}^{3},} \\[2pt]
{\scriptstyle a \,\ge\, 2 }
\end{array}
$
\\[6pt] \midrule \\[-10pt]

\emph{S57}
&
$\left[\tiny\begin{array}{rrrrrrr}
    1 & 1 & a\!+\!3 & a & 1 & 1 & 0 \\
    0 & 0 & 1 & 1 & 1 & 1 & 1 
\end{array}\right]$
&
$
\tiny\begin{array}{c}
    (a\!+\!3,4) \\[2pt]
    (a\!+\!4,1)
\end{array}
$
&
$
\def\arraystretch{0.6}\begin{array}{l}
{\scriptstyle T_{1}^{(a-1-l)}T_{2}^{l}T_{5}^{4}+T_{4}T_{6}^{3}+T_{3}T_{7}^{3},} \\[2pt]
{\scriptstyle a \,\ge\, 2,\, 0 \,\le\, l \,\le\, (a-1)/2 }
\end{array}
$
\\[6pt] \midrule \\[-10pt]

\emph{S58}
&
$\left[\tiny\begin{array}{rrrrrrr}
    1 & 1 & a & b & 1 & 1 & 0 \\
    0 & 0 & 1 & 1 & 1 & 1 & 1 
\end{array}\right]$
&
$
\tiny\begin{array}{c}
    (b\!+\!3,4) \\[2pt]
    (a\!+\!1,1)
\end{array}
$
&
$
\def\arraystretch{0.6}\begin{array}{l}
{\scriptstyle T_{1}^{(b-1)}T_{5}^{4}+T_{4}T_{6}^{3}+T_{2}^{(b+3)}T_{7}^{4},} \\[2pt]
{\scriptstyle a \,>\, b \,\ge\, 2,\, b \text{\emph{ even}} }
\end{array}
$
\\[6pt] \midrule \\[-10pt]

\emph{S59}
&
$\left[\tiny\begin{array}{rrrrrrr}
    1 & 1 & 2a\!+\!1 & 2a & a & a & 0 \\
    0 & 0 & 1 & 1 & 1 & 1 & 1 
\end{array}\right]$
&
$
\tiny\begin{array}{c}
    (4a\!+\!1,4) \\[2pt]
    (2a\!+\!2,1)
\end{array}
$
&
$
\def\arraystretch{0.6}\begin{array}{l}
{\scriptstyle T_{1}T_{5}^{4}+T_{2}T_{6}^{4}+T_{3}T_{4}T_{7}^{2},} \\[2pt]
{\scriptstyle a \,\ge\, 1 }
\end{array}
$
\\[6pt] \midrule \\[-10pt]

\emph{S60}
&
$\left[\tiny\begin{array}{rrrrrrr}
    1 & 1 & a & b & c & c & 0 \\
    0 & 0 & 1 & 1 & 1 & 1 & 1 
\end{array}\right]$
&
$
\tiny\begin{array}{c}
    (a,4) \\[2pt]
    (2c\!+\!b\!+\!2,1)
\end{array}
$
&
$
\def\arraystretch{0.6}\begin{array}{l}
{\scriptstyle T_{1}^{(a-4c)}T_{5}^{4}+T_{2}^{(a-4c)}T_{6}^{4}+T_{3}T_{7}^{3},} \\[2pt]
{\scriptstyle c \,\ge\, 1,\, b \,>\, 2c-1,\, 4c \,<\, a \,\le\, 1+b+2c,\, a \text{\emph{ odd}} }
\end{array}
$
\\[6pt] \midrule \\[-10pt]

\emph{S61}
&
$\left[\tiny\begin{array}{rrrrrrr}
    1 & 1 & a & b & c & c & 0 \\
    0 & 0 & 1 & 1 & 1 & 1 & 1 
\end{array}\right]$
&
$
\tiny\begin{array}{c}
    (b,4) \\[2pt]
    (a\!+\!2c\!+\!2,1)
\end{array}
$
&
$
\def\arraystretch{0.6}\begin{array}{l}
{\scriptstyle T_{1}^{(b-4c)}T_{5}^{4}+T_{2}^{(b-4c)}T_{6}^{4}+T_{4}T_{7}^{3},} \\[2pt]
{\scriptstyle c \,\ge\, 1,\, b \,>\, 4c,\, a \,>\, b,\, b \text{\emph{ odd}} }
\end{array}
$
\\[6pt] \midrule \\[-10pt]

\emph{S64}
&
$\left[\tiny\begin{array}{rrrrrrr}
    1 & 1 & a & b & c & 0 & 0 \\
    0 & 0 & 1 & 1 & 1 & 1 & 1 
\end{array}\right]$
&
$
\tiny\begin{array}{c}
    (b,4) \\[2pt]
    (a\!+\!c\!+\!2,1)
\end{array}
$
&
$
\def\arraystretch{0.6}\begin{array}{l}
{\scriptstyle T_{1}^{(b-4c)}T_{5}^{4}+T_{2}^{b}T_{6}^{4}+T_{4}T_{7}^{3},} \\[2pt]
{\scriptstyle c \,\ge\, 1,\, b \,>\, 4c,\, b \,<\, a \,\le\, 1 + b + c,\, b \text{\emph{ odd}} }
\end{array}
$
\\[6pt] \midrule \\[-10pt]

\emph{S67}
&
$\left[\tiny\begin{array}{rrrrrrr}
    1 & 1 & a & b & c & 0 & 0 \\
    0 & 0 & 1 & 1 & 1 & 1 & 1 
\end{array}\right]$
&
$
\tiny\begin{array}{c}
    (4c,4) \\[2pt]
    (a\!+\!b\!-\!3c\!+\!2,1)
\end{array}
$
&
$
\def\arraystretch{0.6}\begin{array}{l}
{\scriptstyle T_{5}^{4}+T_{1}^{(4c-a)}T_{3}T_{6}^{3}+T_{2}^{(4c-b)}T_{4}T_{7}^{3},} \\[2pt]
{\scriptstyle c \,\ge\, 1,\, a \,>\, b \,>\, 4c }
\end{array}
$
\\[6pt] \midrule \\[-10pt]

\emph{S68}
&
$\left[\tiny\begin{array}{rrrrrrr}
    1 & 1 & a & 4b & b & 0 & 0 \\
    0 & 0 & 1 & 1 & 1 & 1 & 1 
\end{array}\right]$
&
$
\tiny\begin{array}{c}
    (4b,4) \\[2pt]
    (a\!+\!b\!+\!2,1)
\end{array}
$
&
$
\def\arraystretch{0.6}\begin{array}{l}
{\scriptstyle T_{5}^{4}+T_{1}^{(4b-l)}T_{2}^{l}T_{6}^{4}+T_{4}T_{7}^{3},} \\[2pt]
{\scriptstyle b \,\ge\, 1,\, 4b \,<\, a \,\le\, 5b+1,\, 0 \,<\, l \,<\, 2b,\, l \text{\emph{ odd}} }
\end{array}
$
\\[6pt] \midrule \\[-10pt]

\emph{S69}
&
$\left[\tiny\begin{array}{rrrrrrr}
    1 & 1 & 2a & a\!-\!1 & a & 0 & 0 \\
    0 & 0 & 1 & 1 & 3 & 1 & 1 
\end{array}\right]$
&
$
\tiny\begin{array}{c}
    (2a,6) \\[2pt]
    (1\!+\!2a,1)
\end{array}
$
&
$
\def\arraystretch{0.6}\begin{array}{l}
{\scriptstyle T_{1}T_{2}T_{4}^{2}T_{7}^{4}+T_{3}T_{6}^{5}+T_{5}^{2},} \\[2pt]
{\scriptstyle a \,\ge\, 2 }
\end{array}
$
\\[6pt] \midrule \\[-10pt]

\emph{S70}
&
$\left[\tiny\begin{array}{rrrrrrr}
    1 & 1 & 2b & a & b & 0 & 0 \\
    0 & 0 & 1 & 1 & 3 & 1 & 1 
\end{array}\right]$
&
$
\tiny\begin{array}{c}
    (2a,6) \\[2pt]
    (3b\!-\!a\!+\!2,1)
\end{array}
$
&
$
\def\arraystretch{0.6}\begin{array}{l}
{\scriptstyle T_{5}^{2}+T_{3}T_{6}^{5}+T_{1}^{(2b-a-l)}T_{2}^{l}T_{4}T_{7}^{5},} \\[2pt]
{\scriptstyle a,\,b \,\ge\, 1,\, b-1 \,\le\, a \,<\, 2b }
\end{array}
$
\\[6pt] \midrule \\[-10pt]

\emph{S71}
&
$\left[\tiny\begin{array}{rrrrrrr}
    1 & 1 & a & b & c & 0 & 0 \\
    0 & 0 & 1 & 1 & 3 & 1 & 1 
\end{array}\right]$
&
$
\tiny\begin{array}{c}
    (2c,6) \\[2pt]
    (a\!+\!b\!-\!c\!+\!2,1)
\end{array}
$
&
$
\def\arraystretch{0.6}\begin{array}{l}
{\scriptstyle T_{5}^{2}+T_{1}^{(2c-a)}T_{3}T_{6}^{5}+T_{2}^{(2c-b)}T_{4}T_{7}^{5},} \\[2pt]
{\scriptstyle c \,\ge\, 1,\, c-1 \,\le\, b \,<\, 2c,\, b \,<\, a \,<\, 2c }
\end{array}
$
\\[6pt] \midrule \\[-10pt]

\emph{S72}
&
$\left[\tiny\begin{array}{rrrrrrr}
    1 & 1 & a & 2 & 3 & 0 & 0 \\
    0 & 0 & 1 & 1 & 3 & 1 & 1 
\end{array}\right]$
&
$
\tiny\begin{array}{c}
    (6,6) \\[2pt]
    (a\!+\!1,1)
\end{array}
$
&
$
\def\arraystretch{0.6}\begin{array}{l}
{\scriptstyle T_{1}^{5}T_{2}T_{6}^{6}+T_{4}^{3}T_{7}^{3}+T_{5}^{2},} \\[2pt]
{\scriptstyle a \,\ge\, 3 }
\end{array}
$
\\[6pt] \midrule \\[-10pt]

\emph{S73}
&
$\left[\tiny\begin{array}{rrrrrrr}
    1 & 1 & a\!+\!1 & a\!-\!1 & a & 0 & 0 \\
    0 & 0 & 1 & 1 & 3 & 1 & 1 
\end{array}\right]$
&
$
\tiny\begin{array}{c}
    (2a,6) \\[2pt]
    (a\!+\!2,1)
\end{array}
$
&
$
\def\arraystretch{0.6}\begin{array}{l}
{\scriptstyle T_{1}^{(2a-l)}T_{2}^{l}T_{6}^{6}+T_{3}T_{4}T_{7}^{4}+T_{5}^{2},} \\[2pt]
{\scriptstyle a \,\ge\, 2,\, 0 \,<\, l \,\le\, a,\, l \text{\emph{ odd}} }
\end{array}
$
\\[6pt] \midrule \\[-10pt]

\emph{S74}
&
$\left[\tiny\begin{array}{rrrrrrr}
    1 & 1 & 2b & a & b & 0 & 0 \\
    0 & 0 & 1 & 1 & 3 & 1 & 1 
\end{array}\right]$
&
$
\tiny\begin{array}{c}
    (2b,6) \\[2pt]
    (a\!+\!b\!+\!2,1)
\end{array}
$
&
$
\def\arraystretch{0.6}\begin{array}{l}
{\scriptstyle T_{5}^{2}+T_{1}^{(2b-l)}T_{2}^{l}T_{6}^{6}+T_{3}T_{7}^{5},} \\[2pt]
{\scriptstyle a,\, b \,\ge\, 1,\, b-1 \,\le\, a \,<\, 2b,\, 0 \,<\, l \,\le\, b,\, l \text{\emph{ odd}} }
\end{array}
$
\\[6pt] \midrule \\[-10pt]

\emph{S75}
&
$\left[\tiny\begin{array}{rrrrrrr}
    1 & 1 & a & 2b & b & 0 & 0 \\
    0 & 0 & 1 & 1 & 3 & 1 & 1 
\end{array}\right]$
&
$
\tiny\begin{array}{c}
    (2b,6) \\[2pt]
    (a\!+\!b\!+\!2,1)
\end{array}
$
&
$
\def\arraystretch{0.6}\begin{array}{l}
{\scriptstyle T_{5}^{2}+T_{1}^{(2b-l)}T_{2}^{l}T_{6}^{6}+T_{4}T_{7}^{5},} \\[2pt]
{\scriptstyle b \,\ge\, 1,\, a \,>\, 2b,\, 0 \,<\, l \,\le\, b,\, l \text{\emph{ odd}} }
\end{array}
$
\\
\bottomrule
\end{longtable}
\end{center}


\begin{center}
\small\setlength{\tabcolsep}{3pt}
\small\setlength{\arraycolsep}{2pt}
\begin{longtable}{llll}
\toprule
\\[-8pt]

\multicolumn{4}{l}
{
\small\setlength{\tabcolsep}{6pt}
\begin{tabular}{cccc}
$
 Q \ = \ 
\left[\tiny\begin{array}{rrrrrrr}
    1 & 1 & a & a & b & c & 0 \\
    0 & 0 & 1 & 1 & 1 & 1 & 1 
\end{array}\right]
$
&
$\mu \ = \ (a\!+\!3c,4)$
&
$-\KKK \ = \ (a\!-\!2c\!+\!b\!+\!2,1)$
\end{tabular}
}
\\[6pt]

{\small\emph{ID}}
&
\multicolumn{1}{c}{{\small$g$}}
&
{\small\emph{ID}}
&
\multicolumn{1}{c}{{\small$g$}}
\\[3pt]

\emph{S37}
&

$
\def\arraystretch{0.75}\begin{array}{l}
{\scriptstyle T_{1}^{(a+3c-4b)}T_{5}^{4}+T_{3}T_{6}^{3}+T_{2}^{3c}T_{4}T_{7}^{3},} \\[2pt]
{\scriptstyle b \,>\, c \,\ge\, 1,\, b \,\ge\, 2c-1,\, a \,>\, 4b-3c }
\end{array}
$
&
\emph{S38}
&

$
\def\arraystretch{0.75}\begin{array}{l}
{\scriptstyle T_{1}^{(a+3c-4b)}T_{5}^{4}+T_{3}T_{6}^{3}+T_{2}^{(a+3c)}T_{7}^{4},} \\[2pt]
{\scriptstyle b \,>\, c \,\ge\, 1,\, b \,\ge\, 2c-1,\, a \,>\, 4b-3c,\, a \text{\emph{ or }} c \text{\emph{ odd}} }
\end{array}
$
\\[4pt] \midrule \\[-8pt]

\multicolumn{4}{l}
{
\small\setlength{\tabcolsep}{6pt}
\begin{tabular}{cccc}
$
 Q \ = \ 
\left[\tiny\begin{array}{rrrrrrr}
    1 & 1 & 4a\!-\!3b & 4a\!-\!3b & a & b & 0 \\
    0 & 0 & 1 & 1 & 1 & 1 & 1 
\end{array}\right]
$
&
$\mu \ = \ (4a,4)$
&
$-\KKK \ = \ (5a\!-\!5b\!+\!2,1)$
\end{tabular}
}
\\[6pt]

{\small\emph{ID}}
&
\multicolumn{1}{c}{{\small$g$}}
&
{\small\emph{ID}}
&
\multicolumn{1}{c}{{\small$g$}}
\\[3pt]

\emph{S40}
&

$
\def\arraystretch{0.75}\begin{array}{l}
{\scriptstyle T_{5}^{4}+T_{3}T_{6}^{3}+T_{1}^{(3b-l)}T_{2}^{l}T_{4}T_{7}^{3},} \\[2pt]
{\scriptstyle a \,>\, b \,\ge\, 1,\, a \,\ge\, 2b-1,\, 0 \,\le\, l \,\le\, 3b/2 }
\end{array}
$
&
\emph{S41}
&

$
\def\arraystretch{0.75}\begin{array}{l}
{\scriptstyle T_{5}^{4}+T_{3}T_{6}^{3}+T_{1}^{(4a-l)}T_{2}^{l}T_{7}^{4},} \\[2pt]
{\scriptstyle a \,>\, b \,\ge\, 1,\, a \,\ge\, 2b-1,\, 0 \,<\, l \,<\, 2a,\, l \text{\emph{ odd}} }
\end{array}
$
\\[4pt] \midrule \\[-8pt]

\multicolumn{4}{l}
{
\small\setlength{\tabcolsep}{6pt}
\begin{tabular}{cccc}
$
 Q \ = \ 
\left[\tiny\begin{array}{rrrrrrr}
    1 & 1 & 2a\!-\!5b & 2a\!-\!5b & a & b & 0 \\
    0 & 0 & 1 & 1 & 3 & 1 & 1 
\end{array}\right]
$
&
$\mu \ = \ (2a,6)$
&
$-\KKK \ = \ (3a\!-\!9b\!+\!2,1)$
\end{tabular}
}
\\[6pt]

{\small\emph{ID}}
&
\multicolumn{1}{c}{{\small$g$}}
&
{\small\emph{ID}}
&
\multicolumn{1}{c}{{\small$g$}}
\\[3pt]

\emph{S46}
&

$
\def\arraystretch{0.75}\begin{array}{l}
{\scriptstyle T_{5}^{2}+T_{3}T_{6}^{5}+T_{1}^{(10b-2a-l)}T_{2}^{l}T_{4}^{2}T_{7}^{4},} \\[2pt]
{\scriptstyle b \,\ge\, 1,\, 4b-1 \,\le\, a \,<\, 5b }
\end{array}
$
&
\emph{S47}
&

$
\def\arraystretch{0.75}\begin{array}{l}
{\scriptstyle T_{5}^{2}+T_{3}T_{6}^{5}+T_{1}^{(5b-l)}T_{2}^{l}T_{4}T_{7}^{5},} \\[2pt]
{\scriptstyle b \,\ge\, 1,\, a \,\ge\, 4b-1,\, 0 \,\le\, l \,\le\, 5b/2 }
\end{array}
$
\\[10pt]

\emph{S48}
&

$
\def\arraystretch{0.75}\begin{array}{l}
{\scriptstyle T_{5}^{2}+T_{3}T_{6}^{5}+T_{1}^{(2a-l)}T_{2}^{l}T_{7}^{6},} \\[2pt]
{\scriptstyle b \,\ge\, 1,\, a \,\ge\, 4b-1,\, 0 \,<\, l \,\le\, a,\, l \text{\emph{ odd}} }
\end{array}
$
&
&
\\[4pt] \midrule \\[-8pt]

\multicolumn{4}{l}
{
\small\setlength{\tabcolsep}{6pt}
\begin{tabular}{cccc}
$
 Q \ = \ 
\left[\tiny\begin{array}{rrrrrrr}
    1 & 1 & 2a\!+\!2 & 2a\!+\!2 & 3a\!+\!3 & a & 0 \\
    0 & 0 & 1 & 1 & 3 & 1 & 1 
\end{array}\right]
$
&
$\mu \ = \ (6a\!+\!6,6)$
&
$-\KKK \ = \ (2a\!+\!3,1)$
\end{tabular}
}
\\[6pt]

{\small\emph{ID}}
&
\multicolumn{1}{c}{{\small$g$}}
&
{\small\emph{ID}}
&
\multicolumn{1}{c}{{\small$g$}}
\\[3pt]

\emph{S50}
&

$
\def\arraystretch{0.75}\begin{array}{l}
{\scriptstyle T_{1}^{5}T_{2}T_{6}^{6}+T_{3}^{2}T_{4}T_{7}^{3}+T_{5}^{2},} \\[2pt]
{\scriptstyle a \,\ge\, 1 }
\end{array}
$
&
\emph{S51}
&

$
\def\arraystretch{0.75}\begin{array}{l}
{\scriptstyle T_{1}^{5}T_{2}T_{6}^{6}+T_{3}^{3}T_{7}^{3}+T_{5}^{2},} \\[2pt]
{\scriptstyle a \,\ge\, 1 }
\end{array}
$
\\[10pt]

\emph{S52}
&

$
\def\arraystretch{0.75}\begin{array}{l}
{\scriptstyle T_{1}^{3}T_{2}^{3}T_{6}^{6}+T_{3}^{2}T_{4}T_{7}^{3}+T_{5}^{2},} \\[2pt]
{\scriptstyle a \,\ge\, 1 }
\end{array}
$
&
&
\\[4pt] \midrule \\[-8pt]

\multicolumn{4}{l}
{
\small\setlength{\tabcolsep}{6pt}
\begin{tabular}{cccc}
$
 Q \ = \ 
\left[\tiny\begin{array}{rrrrrrr}
    1 & 1 & a & b & c & 0 & 0 \\
    0 & 0 & 1 & 1 & 1 & 1 & 1 
\end{array}\right]
$
&
$\mu \ = \ (a,4)$
&
$-\KKK \ = \ (c\!+\!b\!+\!2,1)$
\end{tabular}
}
\\[6pt]

{\small\emph{ID}}
&
\multicolumn{1}{c}{{\small$g$}}
&
{\small\emph{ID}}
&
\multicolumn{1}{c}{{\small$g$}}
\\[3pt]

\emph{S62}
&

$
\def\arraystretch{0.75}\begin{array}{l}
{\scriptstyle T_{1}^{(a-4c)}T_{5}^{4}+T_{3}T_{6}^{3}+T_{2}^{(a-b)}T_{4}T_{7}^{3},} \\[2pt]
{\scriptstyle c \,\ge\, 1,\, b \,>\, 3c-1,\, 4c \,<\, a \,\le\, 1 + b + c }
\end{array}
$
&
\emph{S63}
&

$
\def\arraystretch{0.75}\begin{array}{l}
{\scriptstyle T_{1}^{(a-4c)}T_{5}^{4}+T_{2}^{a}T_{6}^{4}+T_{3}T_{7}^{3},} \\[2pt]
{\scriptstyle c \,\ge\, 1,\, b \,>\, 3c-1,\, 4c \,<\, a \,\le\, 1 + b + c,\, a \text{\emph{ odd}} }
\end{array}
$
\\[4pt] \midrule \\[-8pt]

\multicolumn{4}{l}
{
\small\setlength{\tabcolsep}{6pt}
\begin{tabular}{cccc}
$
 Q \ = \ 
\left[\tiny\begin{array}{rrrrrrr}
    1 & 1 & 4b & a & b & 0 & 0 \\
    0 & 0 & 1 & 1 & 1 & 1 & 1 
\end{array}\right]
$
&
$\mu \ = \ (4b,4)$
&
$-\KKK \ = \ (a\!+\!b\!+\!2,1)$
\end{tabular}
}
\\[6pt]

{\small\emph{ID}}
&
\multicolumn{1}{c}{{\small$g$}}
&
{\small\emph{ID}}
&
\multicolumn{1}{c}{{\small$g$}}
\\[3pt]

\emph{S65}
&

$
\def\arraystretch{0.75}\begin{array}{l}
{\scriptstyle T_{5}^{4}+T_{3}T_{6}^{3}+T_{1}^{(4b-a-l)}T_{2}^{l}T_{4}T_{7}^{3},} \\[2pt]
{\scriptstyle b \,\ge\, 1,\, 3b-1 \,\le\, a \,<\, 4b,\, 0 \,\le\, l \,\le\, (4b-a)/2 }
\end{array}
$
&
\emph{S66}
&

$
\def\arraystretch{0.75}\begin{array}{l}
{\scriptstyle T_{5}^{4}+T_{1}^{(4b-l)}T_{2}^{l}T_{6}^{4}+T_{3}T_{7}^{3},} \\[2pt]
{\scriptstyle b \,\ge\, 1,\, 3b-1 \,\le\, a \,<\, 4b,\, 0 \,<\, l \,<\, 2b,\, l \text{\emph{ odd}} }
\end{array}
$
\\
\bottomrule
\end{longtable}
\end{center}
\end{class-list}

\medskip

\begin{class-list}\label{class:s=6-series}
Locally factorial Fano fourfoulds of Picard number two with a hypersurface Cox ring and an effective three-torus action: Specifying data for the series with $s = 6$.


\begin{center}
\small\setlength{\tabcolsep}{3pt}
\small\setlength{\arraycolsep}{2pt}
\begin{longtable}{cccl}
\toprule
\emph{ID} & $[w_1,\dots,w_7]$ & $\left[\tiny\begin{array}{r}\mu \\ -\KKK\end{array}\right]$ & \multicolumn{1}{c}{$g$} 
\\ \midrule

\emph{S78}
&
$\left[\tiny\begin{array}{rrrrrrr}
    1 & 1 & b+3d & b & c & d & 0 \\
    0 & 0 & 1 & 1 & 1 & 1 & 1 
\end{array}\right]$
&
$
\tiny\begin{array}{c}
    (b\!+\!3d,4) \\[2pt]
    (b\!+\!d\!+\!c\!+\!2,1)
\end{array}
$
&
$
\def\arraystretch{0.6}\begin{array}{l}
{\scriptstyle T_{1}^{(b-4c+3d-l)}T_{2}^{l}T_{5}^{4}+T_{4}T_{6}^{3}+T_{3}T_{7}^{3},} \\[2pt]
{\scriptstyle c \,>\, d \,\ge\, 1,\, c \,\ge\, 2d-1,\, b \,>\, 4c-3d,} \\[2pt]
{\scriptstyle 0 \,\le\, l \,\le\, (b-4c+3d)/2}
\end{array}
$
\\ \midrule

\emph{S83}
&
$\left[\tiny\begin{array}{rrrrrrr}
    1 & 1 & a & b & c & d & 0 \\
    0 & 0 & 1 & 1 & 1 & 1 & 1 
\end{array}\right]$
&
$
\tiny\begin{array}{c}
    (b,4) \\[2pt]
    (d\!+\!c\!+\!a\!+\!2,1)
\end{array}
$
&
$
\def\arraystretch{0.6}\begin{array}{l}
{\scriptstyle T_{1}^{(b-4c)}T_{5}^{4}+T_{2}^{(b-4d)}T_{6}^{4}+T_{4}T_{7}^{3},} \\[2pt]
{\scriptstyle c \,>\, d \,\ge\, 1,\, b \,>\, 4c,\, a \,>\, b,\, b \text{\emph{ odd}} }
\end{array}
$
\\ \midrule

\emph{S88}
&
$\left[\tiny\begin{array}{rrrrrrr}
    1 & 1 & 4c & 4c\!-\!3d & c & d & 0 \\
    0 & 0 & 1 & 1 & 1 & 1 & 1 
\end{array}\right]$
&
$
\tiny\begin{array}{c}
    (4c,4) \\[2pt]
    (5c\!-\!2d\!+\!2,1)
\end{array}
$
&
$
\def\arraystretch{0.6}\begin{array}{l}
{\scriptstyle T_{3}T_{7}^{3}+T_{4}T_{6}^{3}+T_{5}^{4},} \\[2pt]
{\scriptstyle c \,>\, d \,\ge\, 1,\, c \,\ge\, 2d-1 }
\end{array}
$
\\ \bottomrule
\end{longtable}

\goodbreak

\begin{longtable}{cccl}
\toprule
\emph{ID} & $[w_1,\dots,w_7]$ & $\left[\tiny\begin{array}{r}\mu \\ -\KKK\end{array}\right]$ & \multicolumn{1}{c}{$g$} 
\\ \midrule

\emph{S93}
&
$\left[\tiny\begin{array}{rrrrrrr}
    1 & 1 & a & 4c & c & d & 0 \\
    0 & 0 & 1 & 1 & 1 & 1 & 1 
\end{array}\right]$
&
$
\tiny\begin{array}{c}
    (4c,4) \\[2pt]
    (d\!+\!c\!+\!a\!+\!2,1)
\end{array}
$
&
$
\def\arraystretch{0.6}\begin{array}{l}
{\scriptstyle T_{5}^{4}+T_{1}^{(4c-4d-l)}T_{2}^{l}T_{6}^{4}+T_{4}T_{7}^{3},} \\[2pt]
{\scriptstyle c \,>\, d \,\ge\, 1,\, a \,>\, 4c,\, 0 \,<\, l \,<\, 2c-2d,\, l \text{\emph{ odd}} }
\end{array}
$
\\ \midrule

\emph{S102}
&
$\left[\tiny\begin{array}{rrrrrrr}
    1 & 1 & 2c & 2c\!-\!5d & c & d & 0 \\
    0 & 0 & 1 & 1 & 3 & 1 & 1 
\end{array}\right]$
&
$
\tiny\begin{array}{c}
    (2c,4) \\[2pt]
    (3c\!-\!4d\!+\!2,3)
\end{array}
$
&
$
\def\arraystretch{0.6}\begin{array}{l}
{\scriptstyle T_{3}T_{7}^{5}+T_{4}T_{6}^{5}+T_{5}^{2},} \\[2pt]
{\scriptstyle d \,\ge\, 1,\, c \,>\, 3,\, c \,\ge\, 4d-1 }
\end{array}
$
\\ \midrule

\emph{S105}
&
$\left[\tiny\begin{array}{rrrrrrr}
    1 & 1 & 2c-b & b & c & d & 0 \\
    0 & 0 & 1 & 1 & 3 & 1 & 1 
\end{array}\right]$
&
$
\tiny\begin{array}{c}
    (2c,4) \\[2pt]
    (c\!+\!d\!+\!2,3)
\end{array}
$
&
$
\def\arraystretch{0.6}\begin{array}{l}
{\scriptstyle T_{5}^{2}+T_{1}^{(2c-6d-l)}T_{2}^{l}T_{6}^{6}+T_{3}T_{4}T_{7}^{4},} \\[2pt]
{\scriptstyle d \,\ge\, 1,\, c \,>\, 3d,\, b \,>\, 1,\, c-d-1 \,\le\, b \,<\, c,} \\[2pt]
{\scriptstyle 0 \,<\, l \,\le\, c-3d,\, l \text{\emph{ odd}} }
\end{array}
$
\\ \midrule

\emph{S106}
&
$\left[\tiny\begin{array}{rrrrrrr}
    1 & 1 & a & 2c & c & d & 0 \\
    0 & 0 & 1 & 1 & 3 & 1 & 1 
\end{array}\right]$
&
$
\tiny\begin{array}{c}
    (2c,4) \\[2pt]
    (d\!+\!c\!+\!a\!+\!2,3)
\end{array}
$
&
$
\def\arraystretch{0.6}\begin{array}{l}
{\scriptstyle T_{5}^{2}+T_{1}^{(2c-6d-l)}T_{2}^{l}T_{6}^{6}+T_{4}T_{7}^{5},} \\[2pt]
{\scriptstyle d \,\ge\, 1,\, c \,>\, 3d,\, a \,>\, 2c,\, 0 \,<\, l \,\le\, c-3d,\, l \text{\emph{ odd}} }
\end{array}
$
\\
\bottomrule
\end{longtable}
\end{center}


\begin{center}
\small\setlength{\tabcolsep}{3pt}
\small\setlength{\arraycolsep}{2pt}
\begin{longtable}{llll}
\toprule
\\[-8pt]

\multicolumn{4}{l}
{
\small\setlength{\tabcolsep}{6pt}
\begin{tabular}{cccc}
$
 Q \ = \ 
\left[\tiny\begin{array}{rrrrrrr}
    1 & 1 & a & b & c & d & 0 \\
    0 & 0 & 1 & 1 & 1 & 1 & 1 
\end{array}\right]
$
&
$\mu \ = \ (a\!+\!3d,4)$
&
$-\KKK \ = \ (c\!+\!b\!-\!2d\!+\!2,1)$
\end{tabular}
}
\\[6pt]

{\small\emph{ID}}
&
\multicolumn{1}{c}{{\small$g$}}
&
{\small\emph{ID}}
&
\multicolumn{1}{c}{{\small$g$}}
\\[3pt]

\emph{S76}
&

$
\def\arraystretch{0.75}\begin{array}{l}
{\scriptstyle T_{1}^{(a-4c+3d)}T_{5}^{4}+T_{3}T_{6}^{3}+T_{2}^{(a-b+3d)}T_{4}T_{7}^{3},} \\[1pt]
{\scriptstyle b \,>\, c \,>\, d \,\ge\, 1,\, a \,>\, 4c-3d,} \\[1pt]
{\scriptstyle a \,>\, b-3d,\, a \,\le\, 1+b+c-2d }
\end{array}
$
&
\emph{S77}
&

$
\def\arraystretch{0.75}\begin{array}{l}
{\scriptstyle T_{1}^{(a-4c+3d)}T_{5}^{4}+T_{3}T_{6}^{3}+T_{2}^{(a+3d)}T_{7}^{4},} \\[1pt]
{\scriptstyle b \,>\, c \,>\, d \,\ge\, 1,\, 4c-3d \,<\, a \,\le\, 1+b+c-2d,} \\[1pt]
{\scriptstyle a+3d \text{\emph{ odd}} }
\end{array}
$
\\[4pt] \midrule \\[-8pt]

\multicolumn{4}{l}
{
\small\setlength{\tabcolsep}{6pt}
\begin{tabular}{cccc}
$
 Q \ = \ 
\left[\tiny\begin{array}{rrrrrrr}
    1 & 1 & a & b & c & d & 0 \\
    0 & 0 & 1 & 1 & 1 & 1 & 1 
\end{array}\right]
$
&
$\mu \ = \ (a,4)$
&
$-\KKK \ = \ (b\!+\!d\!+\!c\!+\!2,1)$
\end{tabular}
}
\\[6pt]

{\small\emph{ID}}
&
\multicolumn{1}{c}{{\small$g$}}
&
{\small\emph{ID}}
&
\multicolumn{1}{c}{{\small$g$}}
\\[3pt]

\emph{S79}
&

$
\def\arraystretch{0.75}\begin{array}{l}
{\scriptstyle T_{1}^{(a-4c)}T_{5}^{4}+T_{2}^{(a-b-3d)}T_{4}T_{6}^{3}+T_{3}T_{7}^{3},} \\[1pt]
{\scriptstyle b \,>\, c \,>\, d \,\ge\, 1,\, a \,>\, 4c,\, a \,>\, b+3d,} \\[1pt]
{\scriptstyle a \,\le\, 1+b+c+d }
\end{array}
$
&
\emph{S80}
&

$
\def\arraystretch{0.75}\begin{array}{l}
{\scriptstyle T_{1}^{(a-4c)}T_{5}^{4}+T_{2}^{(a-4d)}T_{6}^{4}+T_{3}T_{7}^{3},} \\[1pt]
{\scriptstyle c \,>\, d \,\ge\, 1,\, b \,\ge\, 3c-d,\, 4c \,<\, a \,\le\, 1+b+c+d,} \\[1pt]
{\scriptstyle a \text{\emph{ odd}} }
\end{array}
$
\\[4pt] \midrule \\[-8pt]

\multicolumn{4}{l}
{
\small\setlength{\tabcolsep}{6pt}
\begin{tabular}{cccc}
$
 Q \ = \ 
\left[\tiny\begin{array}{rrrrrrr}
    1 & 1 & a & b & c & d & 0 \\
    0 & 0 & 1 & 1 & 1 & 1 & 1 
\end{array}\right]
$
&
$\mu \ = \ (b\!+\!3d,4)$
&
$-\KKK \ = \ (c\!+\!a\!-\!2d\!+\!2,1)$
\end{tabular}
}
\\[6pt]

{\small\emph{ID}}
&
\multicolumn{1}{c}{{\small$g$}}
&
{\small\emph{ID}}
&
\multicolumn{1}{c}{{\small$g$}}
\\[3pt]

\emph{S81}
&
$
\def\arraystretch{0.75}\begin{array}{l}
{\scriptstyle T_{1}^{(b-4c+3d)}T_{5}^{4}+T_{4}T_{6}^{3}+T_{2}^{(b+3d-a)}T_{3}T_{7}^{3},} \\[1pt]
{\scriptstyle c \,>\, d \,\ge\, 1,\, b \,>\, 4c-3d,\, b \,<\, a \,<\, b+3d }
\end{array}
$
&
\emph{S82}
&
$
\def\arraystretch{0.75}\begin{array}{l}
{\scriptstyle T_{1}^{(b-4c+3d)}T_{5}^{4}+T_{4}T_{6}^{3}+T_{2}^{(b+3d)}T_{7}^{4},} \\[1pt]
{\scriptstyle c \,>\, d \,\ge\, 1,\, c \,\ge\, 2d-1,\, b \,>\, 4c-3d,} \\[1pt]
{\scriptstyle a \,>\, b,\, b+3d \text{\emph{ odd}} }
\end{array}
$
\\[4pt] \midrule \\[-8pt]

\multicolumn{4}{l}
{
\small\setlength{\tabcolsep}{6pt}
\begin{tabular}{cccc}
$
 Q \ = \ 
\left[\tiny\begin{array}{rrrrrrr}
    1 & 1 & 4c\!-\!3d & b & c & d & 0 \\
    0 & 0 & 1 & 1 & 1 & 1 & 1 
\end{array}\right]
$
&
$\mu \ = \ (4c,4)$
&
$-\KKK \ = \ (c\!+\!b\!-\!2d\!+\!2,1)$
\end{tabular}
}
\\[6pt]

{\small\emph{ID}}
&
\multicolumn{1}{c}{{\small$g$}}
&
{\small\emph{ID}}
&
\multicolumn{1}{c}{{\small$g$}}
\\[3pt]

\emph{S84}
&

$
\def\arraystretch{0.75}\begin{array}{l}
{\scriptstyle T_{5}^{4}+T_{3}T_{6}^{3}+T_{1}^{(4c-b-l)}T_{2}^{l}T_{4}T_{7}^{3},} \\[1pt]
{\scriptstyle c \,>\, d \,\ge\, 1,\, 3c-d-1 \,\le\, b \,<\, 4c-3d }
\end{array}
$
&
\emph{S85}
&

$
\def\arraystretch{0.75}\begin{array}{l}
{\scriptstyle T_{5}^{4}+T_{3}T_{6}^{3}+T_{1}^{(4c-l)}T_{2}^{l}T_{7}^{4},} \\[1pt]
{\scriptstyle d \,\ge\, 1,\, c \,>\, 2d-1,\, 3c-d-1 \,\le\, b \,<\, 4c-3d,} \\[1pt]
{\scriptstyle 0 \,<\, l \,<\, 2c,\, l \text{\emph{ odd}} }
\end{array}
$
\\[4pt] \midrule \\[-8pt]

\multicolumn{4}{l}
{
\small\setlength{\tabcolsep}{6pt}
\begin{tabular}{cccc}
$
 Q \ = \ 
\left[\tiny\begin{array}{rrrrrrr}
    1 & 1 & a & b & c & d & 0 \\
    0 & 0 & 1 & 1 & 1 & 1 & 1 
\end{array}\right]
$
&
$\mu \ = \ (4c,4)$
&
$-\KKK \ = \ (d\!+\!b\!+\!a\!-\!3c\!+\!2,1)$
\end{tabular}
}
\\[6pt]

{\small\emph{ID}}
&
\multicolumn{1}{c}{{\small$g$}}
&
{\small\emph{ID}}
&
\multicolumn{1}{c}{{\small$g$}}
\\[3pt]

\emph{S86}
&

$
\def\arraystretch{0.75}\begin{array}{l}
{\scriptstyle T_{5}^{4}+T_{1}^{(4c-a-3d)}T_{3}T_{6}^{3}+T_{2}^{(4c-b)}T_{4}T_{7}^{3},} \\[1pt]
{\scriptstyle c \,>\, d \,\ge\, 1,\, 3c-d-1 \,\le\, b \,<\, 4c,\, a \,>\, 4c-3d }
\end{array}
$
&
\emph{S87}
&

$
\def\arraystretch{0.75}\begin{array}{l}
{\scriptstyle T_{5}^{4}+T_{1}^{(4c-b-3d)}T_{4}T_{6}^{3}+T_{2}^{(4c-a)}T_{3}T_{7}^{3},} \\[1pt]
{\scriptstyle d \,\ge\, 1,\, c \,>\, 2d-1,\, 3c-d-1 \,\le\, b \,<\, 4c-3d,} \\[1pt]
{\scriptstyle b \,<\, a \,<\, 4c }
\end{array}
$
\\[4pt] \midrule \\[-8pt]

\multicolumn{4}{l}
{
\small\setlength{\tabcolsep}{6pt}
\begin{tabular}{cccc}
$
 Q \ = \ 
\left[\tiny\begin{array}{rrrrrrr}
    1 & 1 & 4c & b & c & d & 0 \\
    0 & 0 & 1 & 1 & 1 & 1 & 1 
\end{array}\right]
$
&
$\mu \ = \ (4c,4)$
&
$-\KKK \ = \ (b\!+\!d\!+\!c\!+\!2,1)$
\end{tabular}
}
\\[6pt]

{\small\emph{ID}}
&
\multicolumn{1}{c}{{\small$g$}}
&
{\small\emph{ID}}
&
\multicolumn{1}{c}{{\small$g$}}
\\[3pt]

\emph{S89}
&

$
\def\arraystretch{0.75}\begin{array}{l}
{\scriptstyle T_{5}^{4}+T_{1}^{(4c-b-3d-l)}T_{2}^{l}T_{4}T_{6}^{3}+T_{3}T_{7}^{3},} \\[1pt]
{\scriptstyle d \,\ge\, 1,\, c \,>\, 2d-1,\, 3c-d-1 \,\le\, b \,<\, 4c-3d }
\end{array}
$
&
\emph{S90}
&

$
\def\arraystretch{0.75}\begin{array}{l}
{\scriptstyle T_{5}^{4}+T_{1}^{(4c-4d-l)}T_{2}^{l}T_{6}^{4}+T_{3}T_{7}^{3},} \\[1pt]
{\scriptstyle c \,>\, d \,\ge\, 1,\, 3c-d-1 \,\le\, b \,<\, 4c,} \\[1pt]
{\scriptstyle 0 \,<\, l \,<\, 2c-2d,\, l \text{\emph{ odd}} }
\end{array}
$
\\ \bottomrule
\end{longtable}

\goodbreak

\begin{longtable}{llll}
\toprule
\\[-8pt]

\multicolumn{4}{l}
{
\small\setlength{\tabcolsep}{6pt}
\begin{tabular}{cccc}
$
 Q \ = \ 
\left[\tiny\begin{array}{rrrrrrr}
    1 & 1 & a & 4c\!-\!3d & c & d & 0 \\
    0 & 0 & 1 & 1 & 1 & 1 & 1 
\end{array}\right]
$
&
$\mu \ = \ (4c,4)$
&
$-\KKK \ = \ (c\!+\!a\!-\!2d\!+\!2,1)$
\end{tabular}
}
\\[6pt]

{\small\emph{ID}}
&
\multicolumn{1}{c}{{\small$g$}}
&
{\small\emph{ID}}
&
\multicolumn{1}{c}{{\small$g$}}
\\[3pt]

\emph{S91}
&

$
\def\arraystretch{0.75}\begin{array}{l}
{\scriptstyle T_{5}^{4}+T_{4}T_{6}^{3}+T_{1}^{(4c-a-l)}T_{2}^{l}T_{3}T_{7}^{3},} \\[1pt]
{\scriptstyle c \,>\, d \,\ge\, 1,\, c \,\ge\, 2d-1,\, 4c-3d \,<\, a \,<\, 4c }
\end{array}
$
&
\emph{S92}
&

$
\def\arraystretch{0.75}\begin{array}{l}
{\scriptstyle T_{5}^{4}+T_{4}T_{6}^{3}+T_{1}^{(4c-l)}T_{2}^{l}T_{7}^{4},} \\[1pt]
{\scriptstyle c,\, d \,\ge\, 1,\, c \,\ge\, 2d-1,\, a \,>\, 4c-3d,} \\[1pt]
{\scriptstyle 0 \,<\, l \,<\, 2c,\, l \text{\emph{ odd}} }
\end{array}
$
\\[4pt] \midrule \\[-8pt]

\multicolumn{4}{l}
{
\small\setlength{\tabcolsep}{6pt}
\begin{tabular}{cccc}
$
 Q \ = \ 
\left[\tiny\begin{array}{rrrrrrr}
    1 & 1 & a & 2c\!-\!5d & c & d & 0 \\
    0 & 0 & 1 & 1 & 3 & 1 & 1 
\end{array}\right]
$
&
$\mu \ = \ (2c,4)$
&
$-\KKK \ = \ (c\!-\!4d\!+\!a\!+\!2,3)$
\end{tabular}
}
\\[6pt]

{\small\emph{ID}}
&
\multicolumn{1}{c}{{\small$g$}}
&
{\small\emph{ID}}
&
\multicolumn{1}{c}{{\small$g$}}
\\[3pt]

\emph{S94}
&

$
\def\arraystretch{0.75}\begin{array}{l}
{\scriptstyle T_{5}^{2}+T_{4}T_{6}^{5}+T_{1}^{(2c-2a-l)}T_{2}^{l}T_{3}^{2}T_{7}^{4},} \\[1pt]
{\scriptstyle d \,\ge\, 1,\, c \,>\, 3d,\, 4d-1 \,\le\, c \,<\, 5d,} \\[1pt]
{\scriptstyle 2c-5d \,<\, a \,<\, c,\, 0 \,<\, l \,\le\, c-a,\, l \text{\emph{ odd}} }
\end{array}
$
&
\emph{S95}
&

$
\def\arraystretch{0.75}\begin{array}{l}
{\scriptstyle T_{5}^{2}+T_{4}T_{6}^{5}+T_{1}^{(2c-a-l)}T_{2}^{l}T_{3}T_{7}^{5},} \\[1pt]
{\scriptstyle d \,\ge\, 1,\, c \,>\, 3,\, c \,\ge\, 4d-1,\, 2c-5d \,<\, a \,<\, 2c,} \\[1pt]
{\scriptstyle 0 \,\le\, l \,\le\, (2c-a)/2 }
\end{array}
$
\\[12pt]

\emph{S96}
&

$
\def\arraystretch{0.75}\begin{array}{l}
{\scriptstyle T_{5}^{2}+T_{4}T_{6}^{5}+T_{1}^{(2c-l)}T_{2}^{l}T_{7}^{6},} \\[1pt]
{\scriptstyle d \,\ge\, 1,\, c \,>\, 3,\, c\,\ge\, 4d-1,\, a \,>\, 2c-5d,} \\[1pt]
{\scriptstyle 0 \,<\, l \,\le\, c,\, l \text{\emph{ odd}} }
\end{array}
$
&
&
\\[4pt] \midrule \\[-8pt]

\multicolumn{4}{l}
{
\small\setlength{\tabcolsep}{6pt}
\begin{tabular}{cccc}
$
 Q \ = \ 
\left[\tiny\begin{array}{rrrrrrr}
    1 & 1 & 2c\!-\!5d & b & c & d & 0 \\
    0 & 0 & 1 & 1 & 3 & 1 & 1 
\end{array}\right]
$
&
$\mu \ = \ (2c,4)$
&
$-\KKK \ = \ (c\!-\!4d\!+\!b\!+\!2,3)$
\end{tabular}
}
\\[6pt]

{\small\emph{ID}}
&
\multicolumn{1}{c}{{\small$g$}}
&
{\small\emph{ID}}
&
\multicolumn{1}{c}{{\small$g$}}
\\[3pt]

\emph{S97}
&

$
\def\arraystretch{0.75}\begin{array}{l}
{\scriptstyle T_{5}^{2}+T_{3}T_{6}^{5}+T_{1}^{(2c-2b-l)}T_{2}^{l}T_{4}^{2}T_{7}^{4},} \\[1pt]
{\scriptstyle d \,\ge\, 1,\, c \,>\, 4d-1,\, c-d-1 \,\le\, b \,<\, c ,} \\[1pt]
{\scriptstyle b \,<\, 2c-5d,\, 0 \,<\, l \,\le\, c-b,\, l \text{\emph{ odd}} }
\end{array}
$
&
\emph{S98}
&

$
\def\arraystretch{0.75}\begin{array}{l}
{\scriptstyle T_{5}^{2}+T_{3}T_{6}^{5}+T_{1}^{(2c-b-l)}T_{2}^{l}T_{4}T_{7}^{5},} \\[1pt]
{\scriptstyle d \,\ge\, 1,\, c \,>\, 3d,\, c-d-1 \,\le\, b \,<\, 2c-5d,} \\[1pt]
{\scriptstyle 0 \,\le\, l \,\le\, (2c-b)/2 }
\end{array}
$
\\[12pt]

\emph{S99}
&

$
\def\arraystretch{0.75}\begin{array}{l}
{\scriptstyle T_{5}^{2}+T_{3}T_{6}^{5}+T_{1}^{(2c-l)}T_{2}^{l}T_{7}^{6},} \\[1pt]
{\scriptstyle d \,\ge\, 1,\, c \,>\, 4d-1,\, b \,>\, 1,\, c-d-1 \,\le\, b \,<\, 2c-5d,} \\[1pt]
{\scriptstyle 0 \,<\, l \,\le\, c \text{\emph{ odd}} }
\end{array}
$
&
&
\\[4pt] \midrule \\[-8pt]

\multicolumn{4}{l}
{
\small\setlength{\tabcolsep}{6pt}
\begin{tabular}{cccc}
$
 Q \ = \ 
\left[\tiny\begin{array}{rrrrrrr}
    1 & 1 & a & b & c & d & 0 \\
    0 & 0 & 1 & 1 & 3 & 1 & 1 
\end{array}\right]
$
&
$\mu \ = \ (2c,4)$
&
$-\KKK \ = \ (d\!+\!b\!+\!a\!-\!c\!+\!2,3)$
\end{tabular}
}
\\[6pt]

{\small\emph{ID}}
&
\multicolumn{1}{c}{{\small$g$}}
&
{\small\emph{ID}}
&
\multicolumn{1}{c}{{\small$g$}}
\\[3pt]

\emph{S100}
&

$
\def\arraystretch{0.75}\begin{array}{l}
{\scriptstyle T_{5}^{2}+T_{1}^{(2c-a-5d)}T_{3}T_{6}^{5}+T_{2}^{(2c-b)}T_{4}T_{7}^{5},} \\[1pt]
{\scriptstyle d \,\ge\, 1,\, c \,>\, 3d,\, c-d-1 \,\le\, b \,<\, a \,<\, 2c-5d }
\end{array}
$
&
\emph{S101}
&

$
\def\arraystretch{0.75}\begin{array}{l}
{\scriptstyle T_{5}^{2}+T_{1}^{(2c-b-5d)}T_{4}T_{6}^{5}+T_{2}^{(2c-a)}T_{3}T_{7}^{5},} \\[1pt]
{\scriptstyle d \,\ge\, 1,\, c \,>\, 4d-1,\, b \,>\, 1,\, c-d-1 \,\le\, b \,<\, 2c-5d,} \\[1pt]
{\scriptstyle b \,<\, a \,<\, 2c }
\end{array}
$
\\[4pt] \midrule \\[-8pt]

\multicolumn{4}{l}
{
\small\setlength{\tabcolsep}{6pt}
\begin{tabular}{cccc}
$
 Q \ = \ 
\left[\tiny\begin{array}{rrrrrrr}
    1 & 1 & 2c & b & c & d & 0 \\
    0 & 0 & 1 & 1 & 3 & 1 & 1 
\end{array}\right]
$
&
$\mu \ = \ (2c,4)$
&
$-\KKK \ = \ (b\!+\!d\!+\!c\!+\!2,3)$
\end{tabular}
}
\\[6pt]

{\small\emph{ID}}
&
\multicolumn{1}{c}{{\small$g$}}
&
{\small\emph{ID}}
&
\multicolumn{1}{c}{{\small$g$}}
\\[3pt]

\emph{S103}
&

$
\def\arraystretch{0.75}\begin{array}{l}
{\scriptstyle T_{5}^{2}+T_{1}^{(2c-b-5d-l)}T_{2}^{l}T_{4}T_{6}^{5}+T_{3}T_{7}^{5},} \\[1pt]
{\scriptstyle d \,\ge\, 1,\, c \,>\, 4d-1,\, c-d-1 \,\le\, b \,<\, 2c-5d,} \\[1pt]
{\scriptstyle 0 \,\le\, l \,\le\, (2c-b-5d)/2 }
\end{array}
$
&
\emph{S104}
&

$
\def\arraystretch{0.75}\begin{array}{l}
{\scriptstyle T_{5}^{2}+T_{1}^{(2c-6d-l)}T_{2}^{l}T_{6}^{6}+T_{3}T_{7}^{5},} \\[1pt]
{\scriptstyle d \,\ge\, 1,\, c \,>\, 3d,\, b \,>\, 1,\, c-d-1 \,\le\, b \,<\, 2c,} \\[1pt]
{\scriptstyle 0 \,<\, l \,\le\, c-3d,\, l \text{\emph{ odd}} }
\end{array}
$
\\
\bottomrule
\end{longtable}
\end{center}
\end{class-list}

Finally, let us compare our results with existing classifications.

\begin{remark}
\label{rem:comp-smooth-cplx1}
The $447$ sporadic cases from Classification lists \ref{class:s=2} to \ref{class:s=6-series} encompass in particular the smooth Fano fourfolds with hypersurface Cox ring of Picard number two and torus action of complexity one. The following table translates their ID's in the present classification to the cases of \cite{FaHaNi}*{Thm.~1.2}.
\begin{center}
\newcommand{\mycolumnwidth}{.48\textwidth}
\begin{minipage}[t]{\mycolumnwidth}
\begin{center}
\begin{tabular}{l|r}
Theorem 1.2 in \cite{FaHaNi} & \multicolumn{1}{c}{ID} \\ \hline
1                  &  84 (\ref{class:s=3-sample}) \\
2                  &  20 (\ref{class:s=3-sample}) \\
4.A: $m=1$, $c=-1$ &  45 (\ref{class:s=3-sample}) \\
4.A: $m=1$, $c=0$  &   1 (\ref{class:s=2}) \\
4.B: $m=1$         &  44 (\ref{class:s=3-sample}) \\
4.C: $m=1$         &   6 (\ref{class:s=2})

\end{tabular}
\end{center}
\end{minipage}
\hspace{0cm}
\begin{minipage}[t]{\mycolumnwidth}
\begin{center}
\begin{tabular}{l|r}
Theorem 1.2 in \cite{FaHaNi} & \multicolumn{1}{c}{ID} \\ \hline
5: $m=1$, $a=0$    & 228 (\ref{class:s=3-nsample}) \\
7: $m=1$           & 253 (\ref{class:s=4-sample}) \\
10: $m=2$          & 244 (\ref{class:s=4-sample}) \\
11: $m=2$, $a_2=1$ & 225 (\ref{class:s=3-sample}) \\
11: $m=2$, $a_2=2$ & 226 (\ref{class:s=3-sample}) \\
12: $m=2$          &  15 (\ref{class:s=2}) \\
\end{tabular}
\end{center}
\end{minipage}
\end{center}
\end{remark}

\begin{remark}
\label{rem:comp-smooth-hs}
At least $268$ varieties from our classification admit a \emph{one-parameter smoothing} to a smooth Fano fourfold of Picard number two. Here, by a one-parameter smoothing of $X$ we mean a flat morphism of varieties $\varphi\colon \mathcal{X} \rightarrow \CC$ such that $\mathcal{X}_0 := \varphi^{-1}(0)$ is isomorphic to~$X$ and there is a non-empty open subset $U \subseteq \CC$ such that $\mathcal{X}_t := \varphi^{-1}(t)$ is smooth for all $t \in U$. The procedure to explicitly construct such a smoothing goes as follows: Let $X = X(Q,g)$ with $(Q,g)$ from the lists \ref{class:s=2} to \ref{class:s=6-series} and assume that up to a unimodular transformation the data $Q = [w_1,\dots,w_7]$ and $\mu = \deg(g)$ appears in \cite{HLM}*{Thm.~1.1}. Then there is a homogeneous spread polynomial $h$ of degree~$\deg(h) = \mu$ such that $X_h$ is a smooth Fano fourfold with general hypersurface Cox ring. We extend the action of $H = (\CC^*)^2$ on $\CC^7$ given by the grading map $Q$ to $\CC^8$ by letting $H$ act trivially on the last coordinate. We set
\begin{equation*}
    \bar{\mathcal{Z}} \ = \ \CC^8,
    \qquad
    \widehat{\mathcal{Z}} \ = \ \bar{\mathcal{Z}}^{ss}(\tau),
    \qquad
    \mathcal{Z} \ = \ \widehat{\mathcal{Z}} \git H,
\end{equation*}
where $\tau \in \Lambda(\CC[T_1,\dots,T_7,T])$ is the unique GIT-cone that contains the anticanonical class $-\KKK_X$ in its interior. Moreover we set
\begin{equation*}
    \bar{\mathcal{X}} \ := \ V((1-T)g+Th) \ \subseteq \ \CC^8,
    \qquad
    \widehat{\mathcal{X}} \ = \ \bar{\mathcal{X}} \cap \widehat{\mathcal{Z}},
    \qquad
    \mathcal{X} \ = \ \widehat{\mathcal{X}} \git H.
\end{equation*}
The projection $\pr\colon \widehat{\mathcal{X}} \rightarrow \CC$ to the last coordinate is $H$-invariant and thus factors through a morphism $\varphi\colon \mathcal{X} \rightarrow \CC$. We have $X \cong \varphi^{-1}(0)$ and $\varphi$ is a smoothing of $X$ with fiber over $t = 1$ isomorphic to $X_h$. In the following table, for each entry~$(Q,\mu)$ from the table in \cite{HLM}*{Thm.~1.1} we list the IDs of the varieties $X(Q,g)$ from the present classification that admit such an explicit smoothing to a smooth Fano fourfold of Picard number two with a general hypersurface Cox ring and data~$(Q,\mu)$.
\begin{center}
\newcommand{\mycolumnwidth}{.48\textwidth}
\begin{minipage}[t]{\mycolumnwidth}
\begin{center}
\begin{tabular}[t]{c|l}
\cite{HLM}*{Thm.~1.1} & IDs
\\ \hline
 1 & 1 \\
 2 & 2, 3 \\
 3 & 4, 5 \\
 4 & 6 \\
 5 & - \\
 6 & 7, 8 \\
 7 & 19, 20 \\
 8 & 21, 22 \\
 9 & 23 - 25 \\
10 &  - \\
11 & 26, 27 \\
12 & 28 - 30 \\
13 & 227, 228 \\
\multirow{2}{*}{14} & 230 - 232; \\
   & S2: $a=1$ \\
15 & 265 - 270 \\
16 & 31, 32 \\
17 & 33, 34 \\
18 & 35 - 39 \\
19 & 40 - 43 \\
20 & 44, 45 \\
21 & 242, 243 \\
22 & 46, 47 \\
23 & 48 - 50 \\
24 & 51, 52 \\
25 & 53 - 71 \\
26 & 72 \\
27 & 73 \\
28 & 74 - 76
\end{tabular}
\end{center}
\end{minipage}
\hspace{0cm}
\begin{minipage}[t]{\mycolumnwidth}
\begin{center}
\begin{tabular}[t]{c|l}
\cite{HLM}*{Thm.~1.1} & IDs
\\ \hline
29 & 77 \\
30 & 78 \\
31 & 79 \\
32 & 80 \\
33 & 415 - 418 \\
34 & 81 - 84 \\
35 & 85 - 91 \\
36 & 92 - 101 \\
37 & 102 - 107 \\
38 & 108 - 129 \\
39 & 130 - 148 \\
40 & 149 \\
41 & 150, 151 \\
42 & 152, 153 \\
43 & 154 - 158 \\
44 & 159, 160 \\
45 & 161 - 166 \\
46 & 167 - 180 \\
47 & 181 - 219 \\
48 & 244 \\
49 & 245 \\
50 & 246, 247 \\
51 & 248 - 251 \\
52 & 411 - 414 \\
53 & 252, 253 \\
54 & 254 \\
55 & 255 - 257 \\
56 & 258 - 261 \\
57 & 262 - 264
\end{tabular}
\end{center}
\end{minipage}
\end{center}

\goodbreak

\begin{center}
\newcommand{\mycolumnwidth}{.48\textwidth}
\begin{minipage}[t]{\mycolumnwidth}
\begin{center}
\begin{tabular}[t]{c|l}
\cite{HLM}*{Thm.~1.1} & IDs
\\ \hline
58 & 9, 10 \\
59 & 11 - 13 \\
60 & 220 - 222 \\
61 & 14 \\
62 & 223 \\
\end{tabular}
\end{center}
\end{minipage}
\hspace{0cm}
\begin{minipage}[t]{\mycolumnwidth}
\begin{center}
\begin{tabular}[t]{c|l}
\cite{HLM}*{Thm.~1.1} & IDs
\\ \hline
63 & 16 \\
64 & 224 \\
65 & 15 \\
66 & 225 \\
67 & 226
\end{tabular}
\end{center}
\end{minipage}
\end{center}
With the smoothing procedure from above one obtains a one-parameter smoothing of the variety no. $17$ in \cref{class:s=2} to $X = Y \times \PP^1$, where $Y \subseteq \PP^4$ is a smooth quartic. The specifying data of $X$ is missing from \cite{HLM}*{Thm.~1.1}.
\end{remark}


\end{document}